\documentclass[10pt,reqno]{amsart}

\usepackage[utf8x]{inputenc}
\usepackage[english]{babel}
\usepackage{amsmath,amsfonts,amssymb,amsthm,shuffle}
\usepackage[T1]{fontenc}
\usepackage[math]{anttor}

\usepackage[top=3.5cm,bottom=3.5cm,left=3.6cm,right=3.6cm]{geometry}

\usepackage{xcolor}
\definecolor{ColBlack}{RGB}{0,0,0} % Black.
\definecolor{ColWhite}{RGB}{255,255,255} % White.
\definecolor{Col1}{RGB}{133,6,6} % Rouge sang.
\definecolor{Col2}{RGB}{198,8,0} % Rouge ponceau.
\definecolor{Col3}{RGB}{174,74,52} % Rouge tomette.
\definecolor{Col4}{RGB}{103,113,121} % Gris de Payne.
\definecolor{Col5}{RGB}{90,94,107} % Ardoise.
\definecolor{Col6}{RGB}{70,63,50} % Taupe.

\usepackage[hyperindex=true,frenchlinks=true,colorlinks=true,
citecolor=Col2,linkcolor=Col3,urlcolor=Col4,linktocpage,
pagebackref=true]{hyperref}
\usepackage{xr}
\usepackage{tikz}
\usetikzlibrary{shapes}
\usetikzlibrary{fit}
\usetikzlibrary{decorations.pathmorphing}

\usepackage{mathtools}
\usepackage{dsfont}
\usepackage{wasysym}
\usepackage{stmaryrd}
\usepackage{cite}
\usepackage{subfig}
\usepackage{multirow}
\usepackage{enumitem}
\usepackage{multicol}

\allowdisplaybreaks

\numberwithin{equation}{subsection}

\makeatletter
\def\l@section{\@tocline{1}{3pt}{1pc}{5pc}{}}
\def\l@subsection{\@tocline{2}{2pt}{2pc}{5pc}{}}
\makeatother

\newtheorem{Theorem}{Theorem}[subsection]
\newtheorem{Proposition}[Theorem]{Proposition}
\newtheorem{Lemma}[Theorem]{Lemma}

\renewcommand{\leq}{\leqslant}
\renewcommand{\geq}{\geqslant}

\title[Operads of decorated cliques I]
    {Operads of decorated cliques I: \\ Construction and quotients}
\date{\today}
\author{Samuele Giraudo}
\address{\scriptsize LIGM, Université Gustave Eiffel, CNRS, ESIEE Paris, F-$77454$
Marne-la-Vallée, France.}
\email{samuele.giraudo@u-pem.fr}

\newcommand{\Hide}[1]{\textcolor{Col4}{\tt [hidden]}}

\newcommand{\Def}[1]{\textcolor{Col3}{\em #1}}
\newcommand{\OEIS}[1]{\href{http://oeis.org/#1}{{\bf #1}}}

\tikzstyle{Centering}=[{baseline={([yshift=-0.5ex]current
    bounding box.center)}}]

\newcommand{\N}{\mathbb{N}}
\newcommand{\Z}{\mathbb{Z}}

\newcommand{\K}{\mathbb{K}}

\newcommand{\Oca}{\mathcal{O}}
\newcommand{\Mca}{\mathcal{M}}

\newcommand{\Hsf}{\mathsf{H}}
\newcommand{\Ksf}{\mathsf{K}}

\newcommand{\Dbb}{\mathbb{D}}
\newcommand{\Ebb}{\mathbb{E}}

\newcommand{\Ubb}{\mathbb{U}}

\newcommand{\Cfr}{\mathfrak{c}}
\newcommand{\Dfr}{\mathfrak{d}}

\newcommand{\Tfr}{\mathfrak{t}}
\newcommand{\Sfr}{\mathfrak{s}}
\newcommand{\Pfr}{\mathfrak{p}}
\newcommand{\Qfr}{\mathfrak{q}}
\newcommand{\Rfr}{\mathfrak{r}}
\newcommand{\Ufr}{\mathfrak{u}}
\newcommand{\Vfr}{\mathfrak{v}}

\newcommand{\Att}{\mathtt{a}}
\newcommand{\Btt}{\mathtt{b}}
\newcommand{\Ctt}{\mathtt{c}}
\newcommand{\Dtt}{\mathtt{d}}
\newcommand{\Ett}{\mathtt{e}}

\newcommand{\As}{\mathsf{As}}

\newcommand{\Cli}{\mathsf{C}}
\newcommand{\TDendr}{\mathsf{TDendr}}
\newcommand{\RatFct}{\mathsf{RatFct}}
\newcommand{\Mould}{\mathsf{Mould}}

\newcommand{\MT}{\mathsf{MT}}
\newcommand{\DMT}{\mathsf{DMT}}
\newcommand{\Grav}{\mathsf{Grav}}
\newcommand{\Lie}{\mathsf{Lie}}
\newcommand{\Lab}{\mathsf{Lab}}
\newcommand{\Bub}{\mathsf{Bub}}
\newcommand{\Deg}{\mathsf{Deg}}
\newcommand{\Cro}{\mathsf{Cro}}
\newcommand{\Acy}{\mathsf{Acy}}
\newcommand{\Whi}{\mathrm{Whi}}

\newcommand{\Nes}{\mathrm{Nes}}
\newcommand{\Paths}{\mathrm{Pat}}
\newcommand{\Forests}{\mathrm{For}}
\newcommand{\Motzkin}{\mathrm{Mot}}
\newcommand{\Diss}{\mathrm{Dis}}
\newcommand{\WNC}{\mathrm{WNC}}
\newcommand{\Luc}{\mathsf{Luc}}

\newcommand{\Angle}[1]{\left\langle#1\right\rangle}
\newcommand{\Unit}{\mathds{1}}
\newcommand{\Hilbert}{\mathcal{H}}
\newcommand{\Op}{\star}

\newcommand{\Nar}{\mathrm{nar}}

\newcommand{\Rel}{\mathfrak{R}}
\newcommand{\Arcs}{\mathcal{A}}
\newcommand{\Diagonals}{\mathcal{D}}
\newcommand{\Edges}{\mathcal{E}}

\newcommand{\Bubbles}{\mathcal{B}}
\newcommand{\Triangles}{\mathcal{T}}
\newcommand{\Primes}{\mathcal{P}}
\newcommand{\Cliques}{\mathcal{C}}

\newcommand{\Returned}{\mathrm{ref}}
\newcommand{\Shift}{\mathrm{sh}}

\newcommand{\Id}{\mathrm{Id}}

\newcommand{\Hamming}{\mathrm{h}}

\newcommand{\Cros}{\mathrm{cros}}
\newcommand{\Degr}{\mathrm{degr}}
\newcommand{\OrdBE}{\preceq_{\mathrm{be}}}
\newcommand{\OrdD}{\preceq_{\mathrm{d}}}
\newcommand{\Del}{\mathrm{d}}
\newcommand{\Skel}{\mathrm{skel}}
\newcommand{\Frac}{\mathrm{F}}

\DeclareMathOperator{\BinRel}{\mathfrak{R}}

\newcommand{\UnitClique}{
\begin{tikzpicture}[scale=.6,Centering]
    \node[CliquePoint](1)at(0,0){};
    \node[CliquePoint](2)at(.75,0){};
    \draw[CliqueEmptyEdge](1)edge[]node[CliqueLabel]{}(2);
\end{tikzpicture}}

\newcommand{\Triangle}[3]{
\begin{tikzpicture}[scale=.42,Centering]
    \node[CliquePoint](1)at(0,1){};
    \node[CliquePoint](2)at(0.87,-0.5){};
    \node[CliquePoint](3)at(-0.87,-0.5){};
    \draw[CliqueEdge](1)edge[]node[CliqueLabel]
        {\begin{math}#3\end{math}}(2);
    \draw[CliqueEdge](1)edge[]node[CliqueLabel]
        {\begin{math}#2\end{math}}(3);
    \draw[CliqueEdge](2)edge[]node[CliqueLabel]
        {\begin{math}#1\end{math}}(3);
\end{tikzpicture}}

\newcommand{\TriangleEXX}[3]{
\begin{tikzpicture}[scale=.42,Centering]
    \node[CliquePoint](1)at(0,1){};
    \node[CliquePoint](2)at(0.87,-0.5){};
    \node[CliquePoint](3)at(-0.87,-0.5){};
    \draw[CliqueEdge](1)edge[]node[CliqueLabel]
        {\begin{math}#3\end{math}}(2);
    \draw[CliqueEdge](1)edge[]node[CliqueLabel]
        {\begin{math}#2\end{math}}(3);
    \draw[CliqueEmptyEdge](2)edge[]node[CliqueLabel]{}(3);
\end{tikzpicture}}

\newcommand{\TriangleXEX}[3]{
\begin{tikzpicture}[scale=.42,Centering]
    \node[CliquePoint](1)at(0,1){};
    \node[CliquePoint](2)at(0.87,-0.5){};
    \node[CliquePoint](3)at(-0.87,-0.5){};
    \draw[CliqueEdge](1)edge[]node[CliqueLabel]
        {\begin{math}#3\end{math}}(2);
    \draw[CliqueEmptyEdge](1)edge[]node[CliqueLabel]{}(3);
    \draw[CliqueEdge](2)edge[]node[CliqueLabel]
        {\begin{math}#1\end{math}}(3);
\end{tikzpicture}}

\newcommand{\TriangleXXE}[3]{
\begin{tikzpicture}[scale=.42,Centering]
    \node[CliquePoint](1)at(0,1){};
    \node[CliquePoint](2)at(0.87,-0.5){};
    \node[CliquePoint](3)at(-0.87,-0.5){};
    \draw[CliqueEmptyEdge](1)edge[]node[CliqueLabel]{}(2);
    \draw[CliqueEdge](1)edge[]node[CliqueLabel]
        {\begin{math}#2\end{math}}(3);
    \draw[CliqueEdge](2)edge[]node[CliqueLabel]
        {\begin{math}#1\end{math}}(3);
\end{tikzpicture}}

\newcommand{\TriangleXEE}[3]{
\begin{tikzpicture}[scale=.42,Centering]
    \node[CliquePoint](1)at(0,1){};
    \node[CliquePoint](2)at(0.87,-0.5){};
    \node[CliquePoint](3)at(-0.87,-0.5){};
    \draw[CliqueEmptyEdge](1)edge[]node[CliqueLabel]{}(2);
    \draw[CliqueEmptyEdge](1)edge[]node[CliqueLabel]{}(3);
    \draw[CliqueEdge](2)edge[]node[CliqueLabel]
        {\begin{math}#1\end{math}}(3);
\end{tikzpicture}}

\newcommand{\TriangleEEX}[3]{
\begin{tikzpicture}[scale=.42,Centering]
    \node[CliquePoint](1)at(0,1){};
    \node[CliquePoint](2)at(0.87,-0.5){};
    \node[CliquePoint](3)at(-0.87,-0.5){};
    \draw[CliqueEdge](1)edge[]node[CliqueLabel]
        {\begin{math}#3\end{math}}(2);
    \draw[CliqueEmptyEdge](1)edge[]node[CliqueLabel]{}(3);
    \draw[CliqueEmptyEdge](2)edge[]node[CliqueLabel]{}(3);
\end{tikzpicture}}

\newcommand{\TriangleEXE}[3]{
\begin{tikzpicture}[scale=.42,Centering]
    \node[CliquePoint](1)at(0,1){};
    \node[CliquePoint](2)at(0.87,-0.5){};
    \node[CliquePoint](3)at(-0.87,-0.5){};
    \draw[CliqueEmptyEdge](1)edge[]node[CliqueLabel]{}(2);
    \draw[CliqueEdge](1)edge[]node[CliqueLabel]
        {\begin{math}#2\end{math}}(3);
    \draw[CliqueEmptyEdge](2)edge[]node[CliqueLabel]{}(3);
\end{tikzpicture}}

\newcommand{\TriangleEEE}[3]{
\begin{tikzpicture}[scale=.42,Centering]
    \node[CliquePoint](1)at(0,1){};
    \node[CliquePoint](2)at(0.87,-0.5){};
    \node[CliquePoint](3)at(-0.87,-0.5){};
    \draw[CliqueEmptyEdge](1)edge[]node[CliqueLabel]{}(2);
    \draw[CliqueEmptyEdge](1)edge[]node[CliqueLabel]{}(3);
    \draw[CliqueEmptyEdge](2)edge[]node[CliqueLabel]{}(3);
\end{tikzpicture}}

\newcommand{\SquareN}[4]{
\begin{tikzpicture}[scale=.6,Centering]
    \node[CliquePoint](1)at(-0.71,0.71){};
    \node[CliquePoint](2)at(0.71,0.71){};
    \node[CliquePoint](3)at(0.71,-0.71){};
    \node[CliquePoint](4)at(-0.71,-0.71){};
    \draw[CliqueEdge](1)edge[]node[CliqueLabel]
        {\begin{math}#2\end{math}}(2);
    \draw[CliqueEdge](1)edge[]node[CliqueLabel]
        {\begin{math}#1\end{math}}(4);
    \draw[CliqueEdge](2)edge[]node[CliqueLabel]
        {\begin{math}#3\end{math}}(3);
    \draw[CliqueEdge](3)edge[]node[CliqueLabel]
        {\begin{math}#4\end{math}}(4);
\end{tikzpicture}}

\newcommand{\SquareLeft}[5]{
\begin{tikzpicture}[scale=.6,Centering]
    \node[CliquePoint](1)at(-0.71,0.71){};
    \node[CliquePoint](2)at(0.71,0.71){};
    \node[CliquePoint](3)at(0.71,-0.71){};
    \node[CliquePoint](4)at(-0.71,-0.71){};
    \draw[CliqueEdge](1)edge[]node[CliqueLabel]
        {\begin{math}#2\end{math}}(2);
    \draw[CliqueEdge](1)edge[]node[CliqueLabel]
        {\begin{math}#1\end{math}}(4);
    \draw[CliqueEdge](2)edge[]node[CliqueLabel]
        {\begin{math}#3\end{math}}(3);
    \draw[CliqueEdge](3)edge[]node[CliqueLabel]
        {\begin{math}#4\end{math}}(4);
    \draw[CliqueEdge](1)edge[]node[CliqueLabel]
        {\begin{math}#5\end{math}}(3);
\end{tikzpicture}}

\newcommand{\SquareRight}[5]{
\begin{tikzpicture}[scale=.6,Centering]
    \node[CliquePoint](1)at(-0.71,0.71){};
    \node[CliquePoint](2)at(0.71,0.71){};
    \node[CliquePoint](3)at(0.71,-0.71){};
    \node[CliquePoint](4)at(-0.71,-0.71){};
    \draw[CliqueEdge](1)edge[]node[CliqueLabel]
        {\begin{math}#2\end{math}}(2);
    \draw[CliqueEdge](1)edge[]node[CliqueLabel]
        {\begin{math}#1\end{math}}(4);
    \draw[CliqueEdge](2)edge[]node[CliqueLabel]
        {\begin{math}#3\end{math}}(3);
    \draw[CliqueEdge](3)edge[]node[CliqueLabel]
        {\begin{math}#4\end{math}}(4);
    \draw[CliqueEdge](2)edge[]node[CliqueLabel]
        {\begin{math}#5\end{math}}(4);
\end{tikzpicture}}

\tikzstyle{CliqueEdge}=[draw=Col2!90,thick]
\tikzstyle{CliqueEdgeColorA}=[CliqueEdge,draw=Col4!80,fill=Col4!8]
\tikzstyle{CliqueEmptyEdge}=[draw=Col4!90,thick,densely dashed]
\tikzstyle{CliqueLabel}=[midway,inner sep=1pt,fill=ColWhite!0,
font=\scriptsize]
\tikzstyle{CliquePoint}=[circle,inner sep=1pt,fill=Col2!25,
draw=Col2!70]
\tikzstyle{CliqueEdgeGray}=[ColBlack!30,draw,cap=round]
\tikzstyle{CliqueEdgeBlue}=[Col1!80,thick,draw,cap=round]
\tikzstyle{CliqueEdgeRed}=[Col2!80,thick,draw,cap=round,dotted]
\tikzstyle{Node}=[circle,draw=Col1!80,fill=Col1!8,inner sep=1pt,
minimum size=2mm,thick,font=\scriptsize]
\tikzstyle{NodeST}=[font=\footnotesize]
\tikzstyle{Edge}=[draw=Col2!80,cap=round,thick]
\tikzstyle{Leaf}=[rectangle,draw=ColBlack!70,fill=ColBlack!16,
inner sep=0pt,minimum size=1mm,thick]
\tikzstyle{EdgeLabel}=[midway,inner sep=1pt,fill=ColWhite!0,
font=\scriptsize]
\tikzstyle{Subtree}=[regular polygon,regular polygon sides=3,
draw=Col1!80,fill=Col1!20,thick,minimum size=5mm,font=\scriptsize]
\tikzstyle{Injection}=[ColBlack!100,draw,>->]
\tikzstyle{Surjection}=[ColBlack!100,draw,->>]
\tikzstyle{Bijection}=[ColBlack!100,draw,<->]

\begin{document}

%%%%%%%%%%%%%%%%%%%%%%%%%%%%%%%%%%%%%%%%%%%%%%%%%%%%%%%%%%%%%%%%%%%%%%%%
%%%%%%%%%%%%%%%%%%%%%%%%%%%%%%%%%%%%%%%%%%%%%%%%%%%%%%%%%%%%%%%%%%%%%%%%
%%%%%%%%%%%%%%%%%%%%%%%%%%%%%%%%%%%%%%%%%%%%%%%%%%%%%%%%%%%%%%%%%%%%%%%%
\begin{abstract}
    We introduce a functorial construction $\Cli$ which takes unitary
    magmas $\Mca$ as input and produces operads. The obtained operads
    involve configurations of chords labeled by elements of $\Mca$, called
    $\Mca$-decorated cliques and generalizing usual configurations of
    chords. By considering combinatorial subfamilies of
    $\Mca$-decorated cliques defined, for instance, by limiting the
    maximal number of crossing diagonals or the maximal degree of the
    vertices, we obtain suboperads and quotients of $\Cli\Mca$. This
    leads to a new hierarchy of operads containing, among others,
    operads on noncrossing configurations, Motzkin configurations,
    forests, dissections of polygons, and involutions. Besides, the
    construction $\Cli$ leads to alternative definitions of the operads
    of simple and double multi-tildes, and of the gravity operad.
\end{abstract}

\maketitle

\tableofcontents

%%%%%%%%%%%%%%%%%%%%%%%%%%%%%%%%%%%%%%%%%%%%%%%%%%%%%%%%%%%%%%%%%%%%%%%%
%%%%%%%%%%%%%%%%%%%%%%%%%%%%%%%%%%%%%%%%%%%%%%%%%%%%%%%%%%%%%%%%%%%%%%%%
%%%%%%%%%%%%%%%%%%%%%%%%%%%%%%%%%%%%%%%%%%%%%%%%%%%%%%%%%%%%%%%%%%%%%%%%
\section*{Introduction}
Configurations of chords on regular polygons are very classical
combinatorial objects. Up to some restrictions or enrichments, sets of
these objects can be put in bijection with several combinatorial
families. For instance, it is well-known that
triangulations~\cite{DRS10}, forming a particular subset of the set
of all configurations of chords, are in one-to-one correspondence with
binary trees, and a lot of structures and operations on binary trees
translate nicely on triangulations. Indeed, among others, the rotation
operation on binary trees~\cite{Knu98} is the covering relation of the
Tamari order~\cite{HT72} and this operation translates as a diagonal
flip in triangulations. Also, noncrossing configurations~\cite{FN99}
form another interesting subfamily of such chord configurations. Natural
generalizations of noncrossing configurations consist in allowing, with
more or less restrictions, some crossing diagonals. One of these
families is formed by the multi-triangulations~\cite{CP92}
wherein the number of mutually crossing diagonals is bounded. In particular,
the class of combinatorial objects in bijection with
some configurations of chords is large enough in order to contain,
among others, dissections of polygons, noncrossing partitions,
permutations, and involutions.
\smallbreak

On the other hand, coming historically from algebraic
topology~\cite{May72,BV73}, operads provide an abstraction of the notion
of operators (of any arities) and their compositions. In more concrete
terms, operads are algebraic structures abstracting the notion of planar
rooted trees and their grafting operations (see~\cite{LV12} for a
complete exposition of the theory and~\cite{Men15} for an exposition
focused on symmetric set-operads). The modern treatment of operads in
algebraic combinatorics consists in regarding combinatorial objects like
operators endowed with gluing operations mimicking the composition of
operators. In the last years, a lot of combinatorial sets and
combinatorial spaces have been endowed fruitfully with the structure of an
operad (see for instance~\cite{Cha08} for an exposition of known
interactions between operads and combinatorics, focused on trees,
\cite{LMN13,GLMN16} where operads abstracting operations in language
theory are introduced, \cite{CG14} for the study of an operad involving
particular noncrossing configurations, \cite{Gir15} for a general
construction of operads on many combinatorial sets, \cite{Gir16b} where
operads are constructed from posets, and \cite{CHN16} where operads on
various species of trees are introduced). In most of the cases, this
approach brings results about enumeration, helps to discover new
statistics, and leads to establish new links (by morphisms) between
different combinatorial sets or spaces. We can observe that most of the
subfamilies of polygons endowed with configurations of chords discussed
above are stable for several natural composition operations. Even
better, some of these can be described as the closure with respect
to these composition operations of small sets of polygons. For this
reason, operads are very promising candidates, among the modern
algebraic structures, to study such objects under an algebraic and
combinatorial flavor.
\smallbreak

The purpose of this work is twofold. First, we are concerned with endowing
the linear span of the configurations of chords with the structure of an
operad. This leads to seeing these objects under a new light, stressing
some of their combinatorial and algebraic properties. Second, we would
provide a general construction of operads of configurations of chords
rich enough so that it includes some already known operads. As a
consequence, we obtain alternative definitions of existing operads and
new interpretations of these. For this aim, we work here with
$\Mca$-decorated cliques (or $\Mca$-cliques for short), that are
complete graphs whose arcs are labeled by elements of $\Mca$, where $\Mca$ is a
unitary magma. These objects are natural generalizations of
configurations of chords since the arcs of any $\Mca$-clique labeled
by the unit of $\Mca$ are considered as missing. The elements of $\Mca$
different from the unit allow moreover to handle chords of different
colors. For instance, each usual noncrossing configuration $\Cfr$ can
be encoded by an $\N_2$-clique $\Pfr$, where $\N_2$ is the cyclic
additive unitary magma $\Z / 2\Z$, wherein each arc labeled by
$1 \in \N_2$ in $\Pfr$ denotes the presence of the same arc in $\Cfr$,
and each arc labeled by $0 \in \N_2$ in $\Pfr$ denotes its absence
in~$\Cfr$. Our construction is materialized by a functor $\Cli$ from
the category of unitary magmas to the category of operads. It builds,
from any unitary magma $\Mca$, an operad $\Cli\Mca$ on $\Mca$-cliques.
The partial composition $\Pfr \circ_i \Qfr$ of two $\Mca$-cliques $\Pfr$
and $\Qfr$ of $\Cli\Mca$ consists in gluing the $i$th edge of $\Pfr$
(with respect to a precise indexation) and a special arc of $\Qfr$,
called the base, together to form a new $\Mca$-clique. The magmatic
operation of $\Mca$ explains how to relabel the two overlapping arcs.
\smallbreak

This operad $\Cli\Mca$ has a lot of properties, which can be apprehended
both under a combinatorial and an algebraic point of view. First, many
families of particular configurations of chords form quotients or
suboperads of $\Cli\Mca$. We can for instance control the degrees of the
vertices or the crossings between diagonals to obtain new operads. We
can also forbid all diagonals, or some labels for the diagonals or the
edges, or all nestings of diagonals, or even all cycles formed by arcs.
All these combinatorial particularities and restrictions on
$\Mca$-cliques behave well algebraically. Moreover, by using the fact
that the direct sum of two ideals of an operad $\Oca$ is still an ideal
of $\Oca$, these constructions can be mixed to get even more operads.
For instance, it is well-known that Motzkin configurations, that are
polygons with disjoint noncrossing diagonals, are enumerated by Motzkin
numbers~\cite{Mot48}. Since a Motzkin configuration can be encoded by
an $\Mca$-clique where all vertices are of degree at most $1$ and no
diagonal crosses another one, we obtain an operad $\Motzkin\Mca$ on
colored Motzkin configurations which is both a quotient of $\Deg_1\Mca$,
the quotient of $\Cli\Mca$ consisting in all $\Mca$-cliques such that
all vertices are of degree at most $1$, and of $\Cro_0\Mca$, the
quotient (and suboperad) of $\Cli\Mca$ consisting in all noncrossing
$\Mca$-cliques. We also get quotients of $\Cli\Mca$ involving, among
others, Schröder trees, forests of paths, forests of trees, dissections
of polygons, Lucas configurations, with colored versions for each of
these. This leads to a new hierarchy of operads, wherein links between
its components appear as surjective or injective operad morphisms. One
of the most notable of these is built by considering the
$\Dbb_0$-cliques that have vertices of degree at most $1$, where
$\Dbb_0$ is the multiplicative unitary magma on $\{0, 1\}$. This is in
fact the quotient $\Deg_1\Dbb_0$ of $\Cli\Dbb_0$ and involves
involutions (or equivalently, standard Young tableaux by the
Robinson-Schensted correspondence~\cite{Lot02}). To the best of our
knowledge, $\Deg_1\Dbb_0$ is the first nontrivial operad on these
objects.
\smallbreak

As an important remark at this stage, let us highlight that when $\Mca$
is nontrivial, $\Cli\Mca$ is not a binary operad. Indeed, all its
minimal generating sets are infinite and its generators have arbitrarily
high arities. Furthermore, the construction $\Cli$ maintains some links
with the operad $\RatFct$ of rational functions introduced by
Loday~\cite{Lod10}. In fact, provided that $\Mca$ satisfies some
conditions, each $\Mca$-clique encodes a rational function. This
defines an operad morphism from $\Cli\Mca$ to $\RatFct$. Moreover, the
construction $\Cli$ allows to construct already known operads in
original ways. For instance, for well-chosen unitary magmas $\Mca$, the
operads $\Cli\Mca$ contain $\MT$ and $\DMT$, two operads respectively
defined in~\cite{LMN13} and~\cite{GLMN16} that involve multi-tildes and
double multi-tildes, operators coming from formal language
theory~\cite{CCM11}. The operads $\Cli\Mca$ also contains $\Grav$, the
gravity operad, a symmetric operad introduced by Getzler~\cite{Get94},
seen here as a nonsymmetric one~\cite{AP15}.
\smallbreak

This text is organized as follows. Section~\ref{sec:definitions_tools}
sets our notations, general definitions, and tools about nonsymmetric
operads (since we deal only with nonsymmetric operads here, we call
these simply operads) and configurations of chords. In
Section~\ref{sec:construction_Cli}, we introduce $\Mca$-cliques, the
construction $\Cli$, and study some of its properties. Then,
Section~\ref{sec:quotients_suboperads} is devoted to define several
suboperads and quotients of $\Cli\Mca$. This leads to plenty of new
operads on particular $\Mca$-cliques. Finally, in
Section~\ref{sec:concrete_constructions}, we use the construction
$\Cli$ to provide alternative definitions of some known operads.
\medbreak

This paper is an extended version of~\cite{Gir17}, containing the
proofs of the presented results.
\medbreak

\subsubsection*{Acknowledgements}
The author would like to thank warmly Dan Petersen for introducing him to the gravity operad
and highlighting links between this operad and the current work.  The author also thanks the
anonymous reviewer for his time and his suggestions, which have greatly contributed to
improving the article.
\medbreak

\subsubsection*{General notations and conventions}
All the algebraic structures of this article have a field of
characteristic zero $\K$ as ground field. For any set $S$, $\K \Angle{S}$
denotes the linear span of the elements of $S$. For any integers $a$ and
$c$, $[a, c]$ denotes the set $\{b \in \N : a \leq b \leq c\}$ and $[n]$,
the set $[1, n]$. The cardinality of a finite set $S$ is denoted
by~$\# S$. If $u$ is a word, its letters are indexed from left to right
from $1$ to its length $|u|$. If $a$ is a letter, $|u|_a$ denotes the
number of occurrences of $a$ in~$u$.
\medbreak

%%%%%%%%%%%%%%%%%%%%%%%%%%%%%%%%%%%%%%%%%%%%%%%%%%%%%%%%%%%%%%%%%%%%%%%%
%%%%%%%%%%%%%%%%%%%%%%%%%%%%%%%%%%%%%%%%%%%%%%%%%%%%%%%%%%%%%%%%%%%%%%%%
%%%%%%%%%%%%%%%%%%%%%%%%%%%%%%%%%%%%%%%%%%%%%%%%%%%%%%%%%%%%%%%%%%%%%%%%
\section{Elementary definitions and tools}
\label{sec:definitions_tools}
We set here our notations and recall some definitions about operads and
related structures. We also introduce some notations and definitions
about configurations of chords in polygons.
\medbreak

%%%%%%%%%%%%%%%%%%%%%%%%%%%%%%%%%%%%%%%%%%%%%%%%%%%%%%%%%%%%%%%%%%%%%%%%
%%%%%%%%%%%%%%%%%%%%%%%%%%%%%%%%%%%%%%%%%%%%%%%%%%%%%%%%%%%%%%%%%%%%%%%%
\subsection{Nonsymmetric operads} \label{subsec:ns_operads}
We adopt most of notations and conventions of~\cite{LV12} about operads.
For the sake of completeness, we recall here the elementary notions
about operads employed thereafter.
\medbreak

A \Def{nonsymmetric operad in the category of vector spaces}, or a
\Def{nonsymmetric operad} for short, is a graded vector space
\begin{equation}
    \Oca := \bigoplus_{n \geq 1} \Oca(n)
\end{equation}
together with linear maps
\begin{equation}
    \circ_i : \Oca(n) \otimes \Oca(m) \to \Oca(n + m - 1),
    \qquad n, m \geq 1, i \in [n],
\end{equation}
called \Def{partial compositions}, and a distinguished element
$\Unit \in \Oca(1)$, the \Def{unit} of $\Oca$. This data has to satisfy
the three relations
\begin{subequations}
\begin{equation} \label{equ:operad_axiom_1}
    (x \circ_i y) \circ_{i + j - 1} z = x \circ_i (y \circ_j z),
    \qquad x \in \Oca(n), y \in \Oca(m),
    z \in \Oca(k), i \in [n], j \in [m],
\end{equation}
\begin{equation} \label{equ:operad_axiom_2}
    (x \circ_i y) \circ_{j + m - 1} z = (x \circ_j z) \circ_i y,
    \qquad x \in \Oca(n), y \in \Oca(m),
    z \in \Oca(k), i < j \in [n],
\end{equation}
\begin{equation} \label{equ:operad_axiom_3}
    \Unit \circ_1 x = x = x \circ_i \Unit,
    \qquad x \in \Oca(n), i \in [n].
\end{equation}
\end{subequations}
Since we consider in this paper only nonsymmetric operads, we shall call
these simply \Def{operads}. Moreover, in this work, we shall only
consider operads $\Oca$ for which $\Oca(1)$ has dimension~$1$.
\medbreak

When $\Oca$ is such that all $\Oca(n)$ have finite dimensions for all
$n \geq 1$, the \Def{Hilbert series} of~$\Oca$ is the series
$\Hilbert_\Oca(t)$ defined by
\begin{equation}
    \Hilbert_\Oca(t) := \sum_{n \geq 1} \dim \Oca(n)\, t^n.
\end{equation}
If $x$ is an element of $\Oca$ such that $x \in \Oca(n)$ for an
$n \geq 1$, we say that $n$ is the \Def{arity} of $x$ and we denote it
by $|x|$. If $\Oca_1$ and $\Oca_2$ are two operads, a linear map
$\phi : \Oca_1 \to \Oca_2$ is an \Def{operad morphism} if it respects
arities, sends the unit of $\Oca_1$ to the unit of $\Oca_2$, and
commutes with partial composition maps. We say that $\Oca_2$ is a
\Def{suboperad} of $\Oca_1$ if $\Oca_2$ is a graded subspace of $\Oca_1$,
$\Oca_1$ and $\Oca_2$ have the same unit, and the partial compositions
of $\Oca_2$ are the ones of $\Oca_1$ restricted on $\Oca_2$. For any
subset $G$ of $\Oca$, the \Def{operad generated} by $G$ is the smallest
suboperad $\Oca^G$ of $\Oca$ containing $G$. When $\Oca^G = \Oca$ and
$G$ is minimal with respect to the inclusion among the subsets of $G$
satisfying this property, $G$ is a \Def{minimal generating set} of
$\Oca$ and its elements are \Def{generators} of $\Oca$. An \Def{operad
ideal} of $\Oca$ is a graded subspace $I$ of $\Oca$ such that, for any
$x \in \Oca$ and $y \in I$, $x \circ_i y$ and $y \circ_j x$ are in $I$
for all valid integers $i$ and $j$. Given an operad ideal $I$ of $\Oca$,
one can define the \Def{quotient operad} $\Oca / I$ of $\Oca$ by $I$ in
the usual way.
\medbreak

Let us recall and set some more definitions about operads. The
\Def{Hadamard product} between the two operads $\Oca_1$ and $\Oca_2$ is
the operad $\Oca_1 * \Oca_2$ satisfying
$(\Oca_1 * \Oca_2)(n) = \Oca_1(n) \otimes \Oca_2(n)$, and its partial
composition is defined component-wise from the partial compositions of
$\Oca_1$ and~$\Oca_2$. An element $x$ of $\Oca(2)$ is \Def{associative}
if $x \circ_1 x = x \circ_2 x$. An \Def{antiautomorphism} of $\Oca$ is
a graded vector space automorphism $\phi$ of $\Oca$ sending the unit of
$\Oca$ to the unit of $\Oca$ and such that for any $x \in \Oca(n)$,
$y \in \Oca$, and $i \in [n]$,
$\phi(x \circ_i y) = \phi(x) \circ_{n - i + 1} \phi(y)$. A
\Def{symmetry} of $\Oca$ is either an automorphism or an
antiautomorphism of $\Oca$. The set of all symmetries of $\Oca$ forms a
group for the map composition, called the \Def{group of symmetries}
of~$\Oca$. A basis $B := \sqcup_{n \geq 1} B(n)$ of $\Oca$ is a
\Def{set-operad basis} if all partial compositions of elements of $B$
belong to $B$. In this case, we say that $\Oca$ is a \Def{set-operad}
with respect to the basis $B$. Moreover, when all the maps
\begin{equation}
    \circ_i^y : B(n) \to B(n + m - 1),
    \qquad n, m \geq 1, i \in [n], y \in B(m),
\end{equation}
defined by
\begin{equation}
    \circ_i^y(x) = x \circ_i y,
    \qquad x \in B(n),
\end{equation}
are injective, we say that $B$ is a \Def{basic set-operad basis} of
$\Oca$. This notion is a slightly modified version of the original
notion of basic set-operads introduced by Vallette\cite{Val07}. Finally,
$\Oca$ is \Def{cyclic} (see~\cite{GK95}) if there is a map
\begin{equation} \label{equ:rotation_map_0}
    \rho : \Oca(n) \to \Oca(n), \qquad n \geq 1,
\end{equation}
satisfying, for all $x \in \Oca(n)$, $y \in \Oca(m)$, and $i \in [n]$,
\begin{subequations}
\begin{equation} \label{equ:rotation_map_1}
    \rho(\Unit) = \Unit,
\end{equation}
\begin{equation} \label{equ:rotation_map_2}
    \rho^{n + 1}(x) = x,
\end{equation}
\begin{equation} \label{equ:rotation_map_3}
    \rho(x \circ_i y) =
    \begin{cases}
        \rho(y) \circ_m \rho(x) & \mbox{if } i = 1, \\
        \rho(x) \circ_{i - 1} y & \mbox{otherwise}.
    \end{cases}
\end{equation}
\end{subequations}
We call such a map $\rho$ a \Def{rotation map}.
\medbreak

%%%%%%%%%%%%%%%%%%%%%%%%%%%%%%%%%%%%%%%%%%%%%%%%%%%%%%%%%%%%%%%%%%%%%%%%
%%%%%%%%%%%%%%%%%%%%%%%%%%%%%%%%%%%%%%%%%%%%%%%%%%%%%%%%%%%%%%%%%%%%%%%%
\subsection{Configurations of chords} \label{subsec:configurations}
Configurations of chords are very classical combinatorial objects defined as collections of
diagonals and edges in regular polygons. The literature abounds of studies of various kinds
of configurations. One can cite for instance~\cite{DRS10} about triangulations, \cite{FN99}
about noncrossing configurations, and~\cite{CP92} about multi-trian\-gulations.
Combinatorial properties related with crossings and nestings in configurations of chords
appear in~\cite{Jon05,CDDSY07,RS10,SS12}. We provide here definitions about these objects
and consider a generalization of configurations wherein the edges and diagonals are labeled
by a set.
\medbreak

%%%%%%%%%%%%%%%%%%%%%%%%%%%%%%%%%%%%%%%%%%%%%%%%%%%%%%%%%%%%%%%%%%%%%%%%
\subsubsection{Polygons}
A \Def{polygon} of \Def{size} $n \geq 1$ is a directed graph $\Pfr$ on
the set of vertices $[n + 1]$. An \Def{arc} of $\Pfr$ is a pair of
integers $(x, y)$ with $1 \leq x < y \leq n + 1$, a \Def{diagonal} is
an arc $(x, y)$ different from $(x, x + 1)$ and $(1, n + 1)$, and an
\Def{edge} is an arc of the form $(x, x + 1)$ and different from
$(1, n + 1)$. We denote by $\Arcs_\Pfr$ (resp. $\Diagonals_\Pfr$,
$\Edges_\Pfr$) the set of all arcs (resp. diagonals, edges) of $\Pfr$.
For any $i \in [n]$, the \Def{$i$th edge} of $\Pfr$ is the edge
$(i, i + 1)$, and the arc $(1, n + 1)$ is the \Def{base} of~$\Pfr$.
\medbreak

In our graphical representations, each polygon is drawn so that its
base is the bottommost segment, vertices are implicitly numbered from
$1$ to $n + 1$ in clockwise direction, and the diagonals are not
drawn. For example,
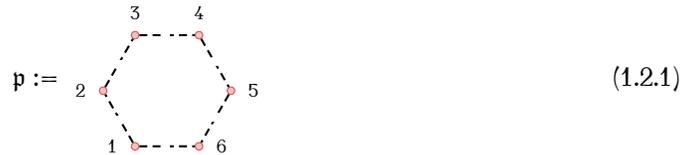
\begin{equation}
    \Pfr :=
    \begin{tikzpicture}[scale=.85,Centering]
        \node[CliquePoint](1)at(-0.50,-0.87){};
        \node[CliquePoint](2)at(-1.00,-0.00){};
        \node[CliquePoint](3)at(-0.50,0.87){};
        \node[CliquePoint](4)at(0.50,0.87){};
        \node[CliquePoint](5)at(1.00,0.00){};
        \node[CliquePoint](6)at(0.50,-0.87){};
        \draw[CliqueEmptyEdge](1)edge[]node[CliqueLabel]{}(2);
        \draw[CliqueEmptyEdge](1)edge[]node[CliqueLabel]{}(6);
        \draw[CliqueEmptyEdge](2)edge[]node[CliqueLabel]{}(3);
        \draw[CliqueEmptyEdge](3)edge[]node[CliqueLabel]{}(4);
        \draw[CliqueEmptyEdge](4)edge[]node[CliqueLabel]{}(5);
        \draw[CliqueEmptyEdge](5)edge[]node[CliqueLabel]{}(6);
        \node[left of=1,node distance=3mm,font=\scriptsize]
            {\begin{math}1\end{math}};
        \node[left of=2,node distance=3mm,font=\scriptsize]
            {\begin{math}2\end{math}};
        \node[above of=3,node distance=3mm,font=\scriptsize]
            {\begin{math}3\end{math}};
        \node[above of=4,node distance=3mm,font=\scriptsize]
            {\begin{math}4\end{math}};
        \node[right of=5,node distance=3mm,font=\scriptsize]
            {\begin{math}5\end{math}};
        \node[right of=6,node distance=3mm,font=\scriptsize]
            {\begin{math}6\end{math}};
    \end{tikzpicture}
\end{equation}
is a polygon of size $5$. Its set of all diagonals is
\begin{equation}
    \Diagonals_\Pfr =
    \{(1, 3), (1, 4), (1, 5), (2, 4), (2, 5), (2, 6),
    (3, 5), (3, 6), (4, 6)\},
\end{equation}
its set of all edges is
\begin{equation}
    \Edges_\Pfr = \{(1, 2), (2, 3), (3, 4), (4, 5), (5, 6)\},
\end{equation}
and its set of all arcs is
\begin{equation}
    \Arcs_\Pfr = \Diagonals_\Pfr \sqcup \Edges_\Pfr \sqcup \{(1, 6)\}.
\end{equation}
\medbreak

%%%%%%%%%%%%%%%%%%%%%%%%%%%%%%%%%%%%%%%%%%%%%%%%%%%%%%%%%%%%%%%%%%%%%%%%
\subsubsection{Configurations} \label{subsubsec:configurations}
For any set $S$, an \Def{$S$-configuration} (or a \Def{configuration}
when $S$ is known without ambiguity) is a polygon $\Pfr$ endowed with a
partial function
\begin{equation}
    \phi_\Pfr : \Arcs_\Pfr \to S.
\end{equation}
When $\phi_\Pfr((x, y))$ is defined, we say that the arc $(x, y)$ is \Def{labeled} and we
write simply $\Pfr(x, y)$ instead of $\phi_\Pfr((x, y))$.  When the base of $\Pfr$ is
labeled, we write simply $\Pfr_0$ for $\Pfr(1, n + 1)$, where $n$ is the size of $\Pfr$.
Finally, when the $i$th edge of $\Pfr$ is labeled, we write simply $\Pfr_i$ for $\Pfr(i, i +
1)$.
\medbreak

In our graphical representations, we shall represent any
$S$-configuration $\Pfr$ by drawing a polygon of the same size as the
one of $\Pfr$ following the conventions explained before, and by
labeling its arcs accordingly. For instance
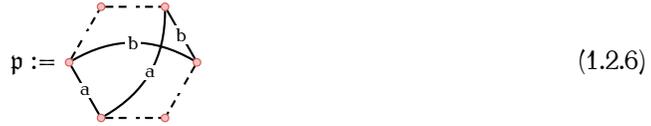
\begin{equation}
    \Pfr :=
    \begin{tikzpicture}[scale=.85,Centering]
        \node[CliquePoint](1)at(-0.50,-0.87){};
        \node[CliquePoint](2)at(-1.00,-0.00){};
        \node[CliquePoint](3)at(-0.50,0.87){};
        \node[CliquePoint](4)at(0.50,0.87){};
        \node[CliquePoint](5)at(1.00,0.00){};
        \node[CliquePoint](6)at(0.50,-0.87){};
        \draw[CliqueEdge](1)edge[]node[CliqueLabel]
            {\begin{math}\Att\end{math}}(2);
        \draw[CliqueEmptyEdge](1)edge[]node[CliqueLabel]{}(6);
        \draw[CliqueEmptyEdge](2)edge[]node[CliqueLabel]{}(3);
        \draw[CliqueEmptyEdge](3)edge[]node[CliqueLabel]{}(4);
        \draw[CliqueEdge](4)edge[]node[CliqueLabel]
            {\begin{math}\Btt\end{math}}(5);
        \draw[CliqueEmptyEdge](5)edge[]node[CliqueLabel]{}(6);
        \draw[CliqueEdge](1)edge[bend right=30]node[CliqueLabel]
            {\begin{math}\Att\end{math}}(4);
        \draw[CliqueEdge](2)edge[bend left=30]node[CliqueLabel]
            {\begin{math}\Btt\end{math}}(5);
    \end{tikzpicture}
\end{equation}
is an $\{\Att, \Btt\}$-configuration. The arcs $(1, 2)$ and
$(1, 4)$ of $\Pfr$ are labeled by $\Att$, the arcs $(2, 5)$ and
$(4, 5)$ are labeled by $\Btt$, and the other arcs are unlabeled.
\medbreak

%%%%%%%%%%%%%%%%%%%%%%%%%%%%%%%%%%%%%%%%%%%%%%%%%%%%%%%%%%%%%%%%%%%%%%%%
\subsubsection{Additional definitions}
Let us now provide some definitions and statistics on configurations.
Let $\Pfr$ be a configuration of size $n$. The \Def{skeleton} of $\Pfr$
is the undirected graph $\Skel(\Pfr)$ on the set of vertices $[n + 1]$
such that for any $x < y \in [n + 1]$, there is an arc $\{x, y\}$
in $\Skel(\Pfr)$ if $(x, y)$ is labeled in $\Pfr$. The \Def{degree} of
a vertex $x$ of $\Pfr$ is the number of vertices adjacent to $x$ in
$\Skel(\Pfr)$. The \Def{degree} $\Degr(\Pfr)$ of $\Pfr$ is the maximal
degree among its vertices. Two (non-necessarily labeled) diagonals
$(x, y)$ and $(x', y')$ of $\Pfr$ are \Def{crossing} if
$x < x' < y < y'$ or $x' < x < y' < y$. The \Def{crossing number} of a
labeled diagonal $(x, y)$ of $\Pfr$ is the number of labeled diagonals
$(x', y')$ such that $(x, y)$ and $(x', y')$ are crossing. The
\Def{crossing number} $\Cros(\Pfr)$ of $\Pfr$ is the maximal crossing
among its labeled diagonals. When $\Cros(\Pfr) = 0$, there are no
crossing diagonals in $\Pfr$ and in this case, $\Pfr$ is
\Def{noncrossing}. A (non-necessarily labeled) arc $(x', y')$ is
\Def{nested} in a (non-necessarily labeled) arc $(x, y)$ of $\Pfr$ if
$x \leq x' < y' \leq y$. We say that $\Pfr$ is \Def{nesting-free} if for
any labeled arcs $(x, y)$ and $(x', y')$ of $\Pfr$ such that $(x', y')$
is nested in $(x, y)$, $(x, y) = (x', y')$. Besides, $\Pfr$ is
\Def{acyclic} if $\Skel(\Pfr)$ is acyclic, that is  there is no subset
$\{x_1, \dots, x_k\}$ of $[n + 1]$ of cardinality $k \geq 3$ such that
$\{x_i, x_{i + 1}\}$ and $\{x_k, x_1\}$ are arcs in $\Skel(\Pfr)$ for
all $i \in [k - 1]$. When $\Pfr$ has no labeled edges nor labeled base,
$\Pfr$ is \Def{white}. If $\Pfr$ has no labeled diagonals, $\Pfr$ is a
\Def{bubble}. A \Def{triangle} is a configuration of size~$2$.
Obviously, all triangles are bubbles, and all bubbles are noncrossing.
\medbreak

%%%%%%%%%%%%%%%%%%%%%%%%%%%%%%%%%%%%%%%%%%%%%%%%%%%%%%%%%%%%%%%%%%%%%%%%
%%%%%%%%%%%%%%%%%%%%%%%%%%%%%%%%%%%%%%%%%%%%%%%%%%%%%%%%%%%%%%%%%%%%%%%%
%%%%%%%%%%%%%%%%%%%%%%%%%%%%%%%%%%%%%%%%%%%%%%%%%%%%%%%%%%%%%%%%%%%%%%%%
\section{From unitary magmas to operads}\label{sec:construction_Cli}
We describe in this section our construction from unitary magmas to
operads and study its main algebraic and combinatorial properties.
\medbreak

%%%%%%%%%%%%%%%%%%%%%%%%%%%%%%%%%%%%%%%%%%%%%%%%%%%%%%%%%%%%%%%%%%%%%%%%
%%%%%%%%%%%%%%%%%%%%%%%%%%%%%%%%%%%%%%%%%%%%%%%%%%%%%%%%%%%%%%%%%%%%%%%%
\subsection{Operads of decorated cliques}%
\label{subsec:decorated_cliques}
We present here our main combinatorial objects, the decorated cliques.
The construction $\Cli$, which takes a unitary magma as input and
produces an operad, is defined.
\medbreak

%%%%%%%%%%%%%%%%%%%%%%%%%%%%%%%%%%%%%%%%%%%%%%%%%%%%%%%%%%%%%%%%%%%%%%%%
\subsubsection{Unitary magmas}
Recall first that a unitary magma is a set endowed with a binary
operation $\Op$ admitting a left and right unit $\Unit_\Mca$. For
convenience, we denote by $\bar{\Mca}$ the set
$\Mca \setminus \{\Unit_\Mca\}$. To explore some examples in this
article, we shall mostly consider four sorts of unitary magmas: the
additive unitary magma on all integers denoted by $\Z$, the cyclic
additive unitary magma on $\Z / \ell \Z$ denoted by $\N_\ell$, the
unitary magma
\begin{equation}
    \Dbb_\ell := \{\Unit, 0, \Dtt_1, \dots, \Dtt_\ell\}
\end{equation}
where $\Unit$ is the unit of $\Dbb_\ell$, $0$ is absorbing, and
$\Dtt_i \Op \Dtt_j = 0$ for all $i, j \in [\ell]$, and the unitary magma
\begin{equation}
    \Ebb_\ell := \{\Unit, \Ett_1, \dots, \Ett_\ell\}
\end{equation}
where $\Unit$ is the unit of $\Ebb_\ell$ and $\Ett_i \Op \Ett_j = \Unit$
for all $i, j \in [\ell]$. Observe that since
\begin{equation}
    \Ett_1 \Op (\Ett_1 \Op \Ett_2) = \Ett_1 \Op \Unit = \Ett_1
    \ne
    \Ett_2 = \Unit \Op \Ett_2 = (\Ett_1 \Op \Ett_1) \Op \Ett_2,
\end{equation}
all unitary magmas $\Ebb_\ell$, $\ell \geq 2$, are not monoids.
\medbreak

%%%%%%%%%%%%%%%%%%%%%%%%%%%%%%%%%%%%%%%%%%%%%%%%%%%%%%%%%%%%%%%%%%%%%%%%
\subsubsection{Decorated cliques}
An \Def{$\Mca$-decorated clique} (or an \Def{$\Mca$-clique} for short)
is an $\Mca$-configuration $\Pfr$ such that all arcs of $\Pfr$ have
labels. When the arc $(x, y)$ of $\Pfr$ is labeled by an element
different from $\Unit_\Mca$, we say that the arc $(x, y)$ is
\Def{solid}. By convention, we require that the $\Mca$-clique
$\UnitClique$ of size $1$ having its base labeled by $\Unit_\Mca$ is the
only such object of size $1$. The set of all $\Mca$-cliques is denoted
by~$\Cliques_\Mca$.
\medbreak

In our graphical representations, we shall represent any $\Mca$-clique
$\Pfr$ by following the drawing conventions of configurations explained
in Section~\ref{subsubsec:configurations} with the difference that
non-solid diagonals are not drawn. For instance,
\begin{equation}
    \Pfr :=
    \begin{tikzpicture}[scale=.85,Centering]
        \node[CliquePoint](1)at(-0.43,-0.90){};
        \node[CliquePoint](2)at(-0.97,-0.22){};
        \node[CliquePoint](3)at(-0.78,0.62){};
        \node[CliquePoint](4)at(-0.00,1.00){};
        \node[CliquePoint](5)at(0.78,0.62){};
        \node[CliquePoint](6)at(0.97,-0.22){};
        \node[CliquePoint](7)at(0.43,-0.90){};
        \draw[CliqueEdge](1)edge[]node[CliqueLabel]
            {\begin{math}-1\end{math}}(2);
        \draw[CliqueEdge](1)
            edge[bend left=30]node[CliqueLabel,near start]
            {\begin{math}2\end{math}}(5);
        \draw[CliqueEdge](1)edge[]node[CliqueLabel]
            {\begin{math}1\end{math}}(7);
        \draw[CliqueEmptyEdge](2)edge[]node[CliqueLabel]{}(3);
        \draw[CliqueEmptyEdge](3)edge[]node[CliqueLabel]{}(4);
        \draw[CliqueEdge](3)
            edge[bend left=30]node[CliqueLabel,near start]
            {\begin{math}-1\end{math}}(7);
        \draw[CliqueEdge](4)edge[]node[CliqueLabel]
            {\begin{math}3\end{math}}(5);
        \draw[CliqueEdge](5)edge[]node[CliqueLabel]
            {\begin{math}2\end{math}}(6);
        \draw[CliqueEdge](5)edge[]node[CliqueLabel]
            {\begin{math}1\end{math}}(7);
        \draw[CliqueEmptyEdge](6)edge[]node[CliqueLabel]{}(7);
    \end{tikzpicture}
\end{equation}
is a $\Z$-clique such that, among others $\Pfr(1, 2) = -1$,
$\Pfr(1, 5) = 2$, $\Pfr(3, 7) = -1$, $\Pfr(5, 7) = 1$, $\Pfr(2, 3) = 0$
(because $0$ is the unit of $\Z$), and $\Pfr(2, 6) = 0$ (for the same
reason).
\medbreak

Let us now provide some definitions and statistics on $\Mca$-cliques.
The \Def{underlying configuration} of $\Pfr$ is the
$\bar{\Mca}$-configuration $\bar{\Pfr}$ of the same size as the one of
$\Pfr$ and such that $\bar{\Pfr}(x, y) := \Pfr(x, y)$ for all solid arcs
$(x, y)$ of $\Pfr$, and all other arcs of $\bar{\Pfr}$ are unlabeled.
The \Def{skeleton}, (resp. \Def{degree}, \Def{crossing number}) of
$\Pfr$ is the skeleton (resp. the degree, the crossing number) of
$\bar{\Pfr}$. Moreover, $\Pfr$ is \Def{nesting-free}, (resp.
\Def{acyclic}, \Def{white}, an \Def{$\Mca$-bubble}, an
\Def{$\Mca$-triangle}), if $\bar{\Pfr}$ is nesting-free (resp. acyclic,
white, a bubble, a triangle). The set of all $\Mca$-bubbles (resp.
$\Mca$-triangles) is denoted by $\Bubbles_\Mca$ (resp.
$\Triangles_\Mca$).
\medbreak

%%%%%%%%%%%%%%%%%%%%%%%%%%%%%%%%%%%%%%%%%%%%%%%%%%%%%%%%%%%%%%%%%%%%%%%%
\subsubsection{Partial composition of $\Mca$-cliques}
From now, the \Def{arity} of an $\Mca$-clique $\Pfr$ is its size and
is denoted by $|\Pfr|$. For any unitary magma $\Mca$, we define the
vector space
\begin{equation}
    \Cli\Mca := \bigoplus_{n \geq 1} \Cli\Mca(n)
    = \K \Angle{\Cliques_\Mca},
\end{equation}
where $\Cli\Mca(n)$ is the linear span of all $\Mca$-cliques of arity
$n$, $n \geq 1$. The set $\Cliques_\Mca$ forms hence a basis of
$\Cli\Mca$ called \Def{fundamental basis}. Observe that the space
$\Cli\Mca(1)$ has dimension $1$ since it is the linear span of the
$\Mca$-clique $\UnitClique$. We endow $\Cli\Mca$ with partial
composition maps
\begin{equation}
    \circ_i : \Cli\Mca(n) \otimes \Cli\Mca(m) \to
    \Cli\Mca(n + m - 1),
    \qquad n, m \geq 1, i \in [n],
\end{equation}
defined linearly, in the fundamental basis, in the following way. Let
$\Pfr$ and $\Qfr$ be two $\Mca$-cliques of respective arities $n$ and
$m$, and $i \in [n]$ be an integer. We set $\Pfr \circ_i \Qfr$ as the
$\Mca$-clique of arity $n + m - 1$ such that, for any arc $(x, y)$ where
$1 \leq x < y \leq n + m$,
\begin{equation} \label{equ:partial_composition_Cli_M}
    (\Pfr \circ_i \Qfr)(x, y) :=
    \begin{cases}
        \Pfr(x, y)
            & \mbox{if } y \leq i, \\
        \Pfr(x, y - m + 1)
            & \mbox{if } x \leq i < i + m \leq y
            \mbox{ and } (x, y) \ne (i, i + m), \\
        \Pfr(x - m + 1, y - m + 1)
            & \mbox{if } i + m \leq x, \\
        \Qfr(x - i + 1, y - i + 1)
            & \mbox{if } i \leq x < y \leq i + m
              \mbox{ and } (x, y) \ne (i, i + m), \\
        \Pfr_i \Op \Qfr_0
            & \mbox{if } (x, y) = (i, i + m), \\
        \Unit_\Mca
            & \mbox{otherwise}.
    \end{cases}
\end{equation}
We recall that $\Op$ denotes the operation of $\Mca$ and $\Unit_\Mca$
its unit. Graphically, $\Pfr \circ_i \Qfr$ is obtained by gluing
the base of $\Qfr$ onto the $i$th edge of $\Pfr$ and by labeling this arc by
$\Pfr_i \Op \Qfr_0$, and by adding all
required non solid diagonals on the graph thus obtained to become a
clique (see Figure~\ref{fig:composition_Cli_M}).
\begin{figure}[ht]
    \centering
    \begin{equation*}
        \begin{tikzpicture}[scale=.55,Centering]
            \node[shape=coordinate](1)at(-0.50,-0.87){};
            \node[shape=coordinate](2)at(-1.00,-0.00){};
            \node[CliquePoint](3)at(-0.50,0.87){};
            \node[CliquePoint](4)at(0.50,0.87){};
            \node[shape=coordinate](5)at(1.00,0.00){};
            \node[shape=coordinate](6)at(0.50,-0.87){};
            \draw[CliqueEdge](1)edge[]node[CliqueLabel]{}(2);
            \draw[CliqueEdge](1)edge[]node[CliqueLabel]{}(6);
            \draw[CliqueEdge](2)edge[]node[CliqueLabel]{}(3);
            \draw[CliqueEdge](3)edge[]node[CliqueLabel]
                {\begin{math}\Pfr_i\end{math}}(4);
            \draw[CliqueEdge](4)edge[]node[CliqueLabel]{}(5);
            \draw[CliqueEdge](5)edge[]node[CliqueLabel]{}(6);
            \node[left of=3,node distance=2mm,font=\scriptsize]
                {\begin{math}i\end{math}};
            \node[right of=4,node distance=4mm,font=\scriptsize]
                {\begin{math}i\!+\!1\end{math}};
            \node[font=\footnotesize](name)at(0,0)
                {\begin{math}\Pfr\end{math}};
        \end{tikzpicture}
        \enspace \circ_i \enspace\enspace
        \begin{tikzpicture}[scale=.55,Centering]
            \node[CliquePoint](1)at(-0.50,-0.87){};
            \node[shape=coordinate](2)at(-1.00,-0.00){};
            \node[shape=coordinate](3)at(-0.50,0.87){};
            \node[shape=coordinate](4)at(0.50,0.87){};
            \node[shape=coordinate](5)at(1.00,0.00){};
            \node[CliquePoint](6)at(0.50,-0.87){};
            \draw[CliqueEdge,Col6!80](1)edge[]node[CliqueLabel]{}(2);
            \draw[CliqueEdge,draw=Col6!80](1)edge[]node[CliqueLabel]
                {\begin{math}\Qfr_0\end{math}}(6);
            \draw[CliqueEdge,Col6!80](2)edge[]node[CliqueLabel]{}(3);
            \draw[CliqueEdge,Col6!80](3)edge[]node[CliqueLabel]{}(4);
            \draw[CliqueEdge,Col6!80](4)edge[]node[CliqueLabel]{}(5);
            \draw[CliqueEdge,Col6!80](5)edge[]node[CliqueLabel]{}(6);
            \node[font=\footnotesize](name)at(0,0)
                {\begin{math}\Qfr\end{math}};
        \end{tikzpicture}
        \quad = \quad
        \begin{tikzpicture}[scale=.55,Centering]
            \begin{scope}
            \node[shape=coordinate](1)at(-0.50,-0.87){};
            \node[shape=coordinate](2)at(-1.00,-0.00){};
            \node[CliquePoint](3)at(-0.50,0.87){};
            \node[CliquePoint](4)at(0.50,0.87){};
            \node[shape=coordinate](5)at(1.00,0.00){};
            \node[shape=coordinate](6)at(0.50,-0.87){};
            \draw[CliqueEdge](1)edge[]node[CliqueLabel]{}(2);
            \draw[CliqueEdge](1)edge[]node[CliqueLabel]{}(6);
            \draw[CliqueEdge](2)edge[]node[CliqueLabel]{}(3);
            \draw[CliqueEdge](3)edge[]node[CliqueLabel]
                {\begin{math}\Pfr_i\end{math}}(4);
            \draw[CliqueEdge](4)edge[]node[CliqueLabel]{}(5);
            \draw[CliqueEdge](5)edge[]node[CliqueLabel]{}(6);
            \node[left of=3,node distance=2mm,font=\scriptsize]
                {\begin{math}i\end{math}};
            \node[right of=4,node distance=4mm,font=\scriptsize]
                {\begin{math}i\!+\!1\end{math}};
            \node[font=\footnotesize](name)at(0,0)
                {\begin{math}\Pfr\end{math}};
            \end{scope}
            \begin{scope}[yshift=2.1cm]
            \node[CliquePoint](1)at(-0.50,-0.87){};
            \node[shape=coordinate](2)at(-1.00,-0.00){};
            \node[shape=coordinate](3)at(-0.50,0.87){};
            \node[shape=coordinate](4)at(0.50,0.87){};
            \node[shape=coordinate](5)at(1.00,0.00){};
            \node[CliquePoint](6)at(0.50,-0.87){};
            \draw[CliqueEdge,Col6!80](1)edge[]node[CliqueLabel]{}(2);
            \draw[CliqueEdge,draw=Col6!80](1)edge[]node[CliqueLabel]
                {\begin{math}\Qfr_0\end{math}}(6);
            \draw[CliqueEdge,Col6!80](2)edge[]node[CliqueLabel]{}(3);
            \draw[CliqueEdge,Col6!80](3)edge[]node[CliqueLabel]{}(4);
            \draw[CliqueEdge,Col6!80](4)edge[]node[CliqueLabel]{}(5);
            \draw[CliqueEdge,Col6!80](5)edge[]node[CliqueLabel]{}(6);
            \node[font=\footnotesize](name)at(0,0)
                {\begin{math}\Qfr\end{math}};
            \end{scope}
        \end{tikzpicture}
        \quad = \quad
        \begin{tikzpicture}[scale=.75,Centering]
            \node[shape=coordinate](1)at(-0.31,-0.95){};
            \node[shape=coordinate](2)at(-0.81,-0.59){};
            \node[CliquePoint](3)at(-1.00,-0.00){};
            \node[shape=coordinate](4)at(-0.81,0.59){};
            \node[shape=coordinate](5)at(-0.31,0.95){};
            \node[shape=coordinate](6)at(0.31,0.95){};
            \node[shape=coordinate](7)at(0.81,0.59){};
            \node[CliquePoint](8)at(1.00,0.00){};
            \node[shape=coordinate](9)at(0.81,-0.59){};
            \node[shape=coordinate](10)at(0.31,-0.95){};
            \draw[CliqueEdge](1)edge[]node[]{}(2);
            \draw[CliqueEdge](1)edge[]node[]{}(10);
            \draw[CliqueEdge](2)edge[]node[]{}(3);
            \draw[CliqueEdge,Col6!80](3)edge[]node[]{}(4);
            \draw[CliqueEdge,Col6!80](4)edge[]node[]{}(5);
            \draw[CliqueEdge,Col6!80](5)edge[]node[]{}(6);
            \draw[CliqueEdge,Col6!80](6)edge[]node[]{}(7);
            \draw[CliqueEdge,Col6!80](7)edge[]node[]{}(8);
            \draw[CliqueEdge](8)edge[]node[]{}(9);
            \draw[CliqueEdge](9)edge[]node[]{}(10);
            \node[left of=3,node distance=2mm,font=\scriptsize]
                {\begin{math}i\end{math}};
            \node[right of=8,node distance=5mm,font=\scriptsize]
                {\begin{math}i\!+\!m\end{math}};
            \draw[CliqueEdge,draw=Col6!90]
                (3)edge[]node[CliqueLabel]
                {\begin{math}\Pfr_i \Op \Qfr_0\end{math}}(8);
        \end{tikzpicture}
    \end{equation*}
    \caption{\footnotesize
    The partial composition of $\Cli\Mca$, described in graphical terms.
    Here, $\Pfr$ and $\Qfr$ are two $\Mca$-cliques. The arity of $\Qfr$
    is~$m$ and $i$ is an integer between $1$ and~$|\Pfr|$.}
    \label{fig:composition_Cli_M}
\end{figure}
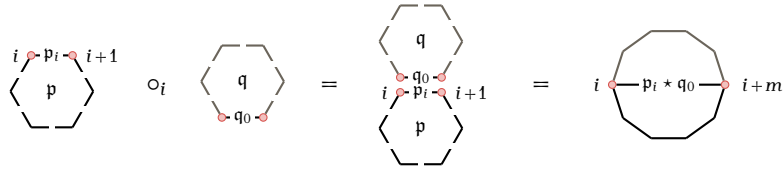
For example, in $\Cli\Z$, one has the two partial compositions
\begin{subequations}
\begin{equation}
    \begin{tikzpicture}[scale=.85,Centering]
        \node[CliquePoint](1)at(-0.50,-0.87){};
        \node[CliquePoint](2)at(-1.00,-0.00){};
        \node[CliquePoint](3)at(-0.50,0.87){};
        \node[CliquePoint](4)at(0.50,0.87){};
        \node[CliquePoint](5)at(1.00,0.00){};
        \node[CliquePoint](6)at(0.50,-0.87){};
        \draw[CliqueEdge](1)edge[]node[CliqueLabel]
            {\begin{math}1\end{math}}(2);
        \draw[CliqueEdge](1)edge[bend left=30]node[CliqueLabel]
            {\begin{math}-2\end{math}}(5);
        \draw[CliqueEmptyEdge](1)edge[]node[CliqueLabel]{}(6);
        \draw[CliqueEdge](2)edge[]node[CliqueLabel]
            {\begin{math}-2\end{math}}(3);
        \draw[CliqueEmptyEdge](3)edge[]node[CliqueLabel]{}(4);
        \draw[CliqueEdge](3)edge[bend right=30]node[CliqueLabel]
            {\begin{math}1\end{math}}(5);
        \draw[CliqueEmptyEdge](4)edge[]node[CliqueLabel]{}(5);
        \draw[CliqueEmptyEdge](5)edge[]node[CliqueLabel]{}(6);
    \end{tikzpicture}
    \enspace \circ_2 \enspace
    \begin{tikzpicture}[scale=.65,Centering]
        \node[CliquePoint](1)at(-0.71,-0.71){};
        \node[CliquePoint](2)at(-0.71,0.71){};
        \node[CliquePoint](3)at(0.71,0.71){};
        \node[CliquePoint](4)at(0.71,-0.71){};
        \draw[CliqueEmptyEdge](1)edge[]node[CliqueLabel]{}(2);
        \draw[CliqueEdge](1)edge[]node[CliqueLabel,near end]
            {\begin{math}1\end{math}}(3);
        \draw[CliqueEdge](1)edge[]node[CliqueLabel]
            {\begin{math}3\end{math}}(4);
        \draw[CliqueEmptyEdge](2)edge[]node[CliqueLabel]{}(3);
        \draw[CliqueEdge](2)edge[]node[CliqueLabel,near start]
            {\begin{math}1\end{math}}(4);
        \draw[CliqueEdge](3)edge[]node[CliqueLabel]
            {\begin{math}2\end{math}}(4);
    \end{tikzpicture}
    \enspace = \enspace
    \begin{tikzpicture}[scale=1.1,Centering]
        \node[CliquePoint](1)at(-0.38,-0.92){};
        \node[CliquePoint](2)at(-0.92,-0.38){};
        \node[CliquePoint](3)at(-0.92,0.38){};
        \node[CliquePoint](4)at(-0.38,0.92){};
        \node[CliquePoint](5)at(0.38,0.92){};
        \node[CliquePoint](6)at(0.92,0.38){};
        \node[CliquePoint](7)at(0.92,-0.38){};
        \node[CliquePoint](8)at(0.38,-0.92){};
        \draw[CliqueEdge](1)edge[]node[CliqueLabel]
            {\begin{math}1\end{math}}(2);
        \draw[CliqueEdge](1)edge[bend left=30]node[CliqueLabel]
            {\begin{math}-2\end{math}}(7);
        \draw[CliqueEmptyEdge](1)edge[]node[CliqueLabel]{}(8);
        \draw[CliqueEmptyEdge](2)edge[]node[CliqueLabel]{}(3);
        \draw[CliqueEdge](2)edge[bend right=30]node[CliqueLabel]
            {\begin{math}1\end{math}}(4);
        \draw[CliqueEdge](2)edge[bend right=30]node[CliqueLabel]
            {\begin{math}1\end{math}}(5);
        \draw[CliqueEmptyEdge](3)edge[]node[CliqueLabel]{}(4);
        \draw[CliqueEdge](3)edge[bend right=30]node[CliqueLabel]
            {\begin{math}1\end{math}}(5);
        \draw[CliqueEdge](4)edge[]node[CliqueLabel]
            {\begin{math}2\end{math}}(5);
        \draw[CliqueEmptyEdge](5)edge[]node[CliqueLabel]{}(6);
        \draw[CliqueEdge](5)edge[bend right=30]node[CliqueLabel]
            {\begin{math}1\end{math}}(7);
        \draw[CliqueEmptyEdge](6)edge[]node[CliqueLabel]{}(7);
        \draw[CliqueEmptyEdge](7)edge[]node[CliqueLabel]{}(8);
    \end{tikzpicture}\,,
\end{equation}
\begin{equation}
    \begin{tikzpicture}[scale=.85,Centering]
        \node[CliquePoint](1)at(-0.50,-0.87){};
        \node[CliquePoint](2)at(-1.00,-0.00){};
        \node[CliquePoint](3)at(-0.50,0.87){};
        \node[CliquePoint](4)at(0.50,0.87){};
        \node[CliquePoint](5)at(1.00,0.00){};
        \node[CliquePoint](6)at(0.50,-0.87){};
        \draw[CliqueEdge](1)edge[]node[CliqueLabel]
            {\begin{math}1\end{math}}(2);
        \draw[CliqueEdge](1)edge[bend left=30]node[CliqueLabel]
            {\begin{math}-2\end{math}}(5);
        \draw[CliqueEmptyEdge](1)edge[]node[CliqueLabel]{}(6);
        \draw[CliqueEdge](2)edge[]node[CliqueLabel]
            {\begin{math}-2\end{math}}(3);
        \draw[CliqueEmptyEdge](3)edge[]node[CliqueLabel]{}(4);
        \draw[CliqueEdge](3)edge[bend right=30]node[CliqueLabel]
            {\begin{math}1\end{math}}(5);
        \draw[CliqueEmptyEdge](4)edge[]node[CliqueLabel]{}(5);
        \draw[CliqueEmptyEdge](5)edge[]node[CliqueLabel]{}(6);
    \end{tikzpicture}
    \enspace \circ_2 \enspace
    \begin{tikzpicture}[scale=.65,Centering]
        \node[CliquePoint](1)at(-0.71,-0.71){};
        \node[CliquePoint](2)at(-0.71,0.71){};
        \node[CliquePoint](3)at(0.71,0.71){};
        \node[CliquePoint](4)at(0.71,-0.71){};
        \draw[CliqueEmptyEdge](1)edge[]node[CliqueLabel]{}(2);
        \draw[CliqueEdge](1)edge[]node[CliqueLabel,near end]
            {\begin{math}1\end{math}}(3);
        \draw[CliqueEdge](1)edge[]node[CliqueLabel]
            {\begin{math}2\end{math}}(4);
        \draw[CliqueEmptyEdge](2)edge[]node[CliqueLabel]{}(3);
        \draw[CliqueEdge](2)edge[]node[CliqueLabel,near start]
            {\begin{math}1\end{math}}(4);
        \draw[CliqueEdge](3)edge[]node[CliqueLabel]
            {\begin{math}2\end{math}}(4);
    \end{tikzpicture}
    \enspace = \enspace
    \begin{tikzpicture}[scale=1.1,Centering]
        \node[CliquePoint](1)at(-0.38,-0.92){};
        \node[CliquePoint](2)at(-0.92,-0.38){};
        \node[CliquePoint](3)at(-0.92,0.38){};
        \node[CliquePoint](4)at(-0.38,0.92){};
        \node[CliquePoint](5)at(0.38,0.92){};
        \node[CliquePoint](6)at(0.92,0.38){};
        \node[CliquePoint](7)at(0.92,-0.38){};
        \node[CliquePoint](8)at(0.38,-0.92){};
        \draw[CliqueEdge](1)edge[]node[CliqueLabel]
            {\begin{math}1\end{math}}(2);
        \draw[CliqueEdge](1)edge[bend left=30]node[CliqueLabel]
            {\begin{math}-2\end{math}}(7);
        \draw[CliqueEmptyEdge](1)edge[]node[CliqueLabel]{}(8);
        \draw[CliqueEmptyEdge](2)edge[]node[CliqueLabel]{}(3);
        \draw[CliqueEdge](2)edge[bend right=30]node[CliqueLabel]
            {\begin{math}1\end{math}}(4);
        \draw[CliqueEmptyEdge](3)edge[]node[CliqueLabel]{}(4);
        \draw[CliqueEdge](3)edge[bend right=30]node[CliqueLabel]
            {\begin{math}1\end{math}}(5);
        \draw[CliqueEdge](4)edge[]node[CliqueLabel]
            {\begin{math}2\end{math}}(5);
        \draw[CliqueEmptyEdge](5)edge[]node[CliqueLabel]{}(6);
        \draw[CliqueEdge](5)edge[bend right=30]node[CliqueLabel]
            {\begin{math}1\end{math}}(7);
        \draw[CliqueEmptyEdge](6)edge[]node[CliqueLabel]{}(7);
        \draw[CliqueEmptyEdge](7)edge[]node[CliqueLabel]{}(8);
    \end{tikzpicture}\,.
\end{equation}
\end{subequations}
\medbreak

%%%%%%%%%%%%%%%%%%%%%%%%%%%%%%%%%%%%%%%%%%%%%%%%%%%%%%%%%%%%%%%%%%%%%%%%
\subsubsection{Functorial construction from unitary magmas to operads}
If $\Mca_1$ and $\Mca_2$ are two unitary magmas and
$\theta : \Mca_1 \to \Mca_2$ is a unitary magma morphism, we define
\begin{equation}
    \Cli\theta : \Cli\Mca_1 \to \Cli\Mca_2
\end{equation}
as the linear map sending any $\Mca_1$-clique $\Pfr$ of arity $n$ to the
$\Mca_2$-clique $(\Cli\theta)(\Pfr)$ of the same arity such that, for
any arc $(x, y)$ where $1 \leq x < y \leq  n + 1$,
\begin{equation} \label{equ:morphism_Cli_M}
    ((\Cli\theta)(\Pfr))(x, y) := \theta(\Pfr(x, y)).
\end{equation}
Graphically, $(\Cli\theta)(\Pfr)$ is the $\Mca_2$-clique obtained
by relabeling each arc of $\Pfr$ by the image of its label by~$\theta$.
\medbreak

\begin{Theorem} \label{thm:clique_construction}
    The construction $\Cli$ is a functor from the category of unitary
    magmas to the category of operads. Moreover, $\Cli$ respects
    injections and surjections.
\end{Theorem}
\begin{proof}
    Let $\Mca$ be a unitary magma. The fact that $\Cli\Mca$ endowed with
    the partial composition~\eqref{equ:partial_composition_Cli_M} is an
    operad can be established by showing that the two associativity
    relations~\eqref{equ:operad_axiom_1} and~\eqref{equ:operad_axiom_2}
    of operads are satisfied. This is a technical but a simple
    verification. Since $\Cli\Mca(1)$ contains $\UnitClique$ and this
    element is the unit for this partial composition,
    \eqref{equ:operad_axiom_3} holds. Moreover, let $\Mca_1$ and
    $\Mca_2$ be two unitary magmas and $\theta : \Mca_1 \to \Mca_2$ be a
    unitary magma morphism. The fact that the map $\Cli\theta$ defined
    in~\eqref{equ:morphism_Cli_M} is an operad morphism is
    straightforward to check. All this implies that $\Cli$ is a functor.
    Finally, the fact that $\Cli$ respects injections and surjections is
    also straightforward to verify.
\end{proof}
\medbreak

We name the construction $\Cli$ as the \Def{clique construction} and
$\Cli\Mca$ as the \Def{$\Mca$-clique operad}. Observe that the
fundamental basis of $\Cli\Mca$ is a set-operad basis of $\Cli\Mca$.
Besides, when $\Mca$ is the trivial unitary magma $\{\Unit_\Mca\}$,
$\Cli\Mca$ is the linear span of all decorated cliques having only
non-solid arcs. Thus, each space $\Cli\Mca(n)$, $n \geq 1$, is of
dimension $1$ and it follows from the definition of the partial
composition of $\Cli\Mca$ that this operad is isomorphic to the
associative operad $\As$. The next result shows that the clique
construction is compatible with the Cartesian product of unitary magmas.
\medbreak

\begin{Proposition} \label{prop:Cli_M_Cartesian_product}
    Let $\Mca_1$ and $\Mca_2$ be two unitary magmas. Then,
    $\Cli(\Mca_1 \times \Mca_2)$ is isomorphic to the Hadamard product
    of operads $(\Cli\Mca_1) * (\Cli\Mca_2)$.
\end{Proposition}
\begin{proof}
    Let
    $\phi : (\Cli\Mca_1) * (\Cli\Mca_2) \to \Cli(\Mca_1 \times \Mca_2)$
    be the linear map defined as follows. For any $\Mca_1$-clique
    $\Pfr$ of $\Cli\Mca_1$ and any $\Mca_2$-clique $\Qfr$ of $\Cli\Mca_2$
    both of arity $n$, $\phi(\Pfr \otimes \Qfr)$ is the
    $\Mca_1 \times \Mca_2$-clique defined, for any
    $1 \leq x < y \leq n + 1$, by
    \begin{equation}
        \left(\phi(\Pfr \otimes \Qfr)\right)(x, y)
        := (\Pfr(x, y), \Qfr(x, y)).
    \end{equation}
    Let the linear map
    $\psi : \Cli(\Mca_1 \times \Mca_2) \to (\Cli\Mca_1) * (\Cli\Mca_2)$
    defined, for any $\Mca_1 \times \Mca_2$-clique $\Rfr$ of
    $\Cli(\Mca_1 \times \Mca_2)$ of arity $n$, as  follows. The
    $\Mca_1$-clique $\Pfr$ and the $\Mca_2$-clique $\Qfr$ of arity $n$
    of the tensor $\Pfr \otimes \Qfr := \psi(\Rfr)$ are defined, for any
    $1 \leq x < y \leq n + 1$, by $\Pfr(x, y) := a$ and
    $\Qfr(x, y) := b$ where $(a, b) = \Rfr(x, y)$. Since we observe
    immediately that $\psi$ is the inverse of $\phi$, $\phi$ is a
    bijection. Moreover, it follows from the definition of the partial
    composition of clique operads that $\phi$ is an operad morphism. The
    statement of the proposition follows.
\end{proof}
\medbreak

%%%%%%%%%%%%%%%%%%%%%%%%%%%%%%%%%%%%%%%%%%%%%%%%%%%%%%%%%%%%%%%%%%%%%%%%
%%%%%%%%%%%%%%%%%%%%%%%%%%%%%%%%%%%%%%%%%%%%%%%%%%%%%%%%%%%%%%%%%%%%%%%%
\subsection{General properties}
We investigate here some properties of clique operads, as their
dimensions, their minimal generating sets, the fact that they admit a
cyclic operad structure, and describe their partial compositions over
two alternative bases.
\medbreak

%%%%%%%%%%%%%%%%%%%%%%%%%%%%%%%%%%%%%%%%%%%%%%%%%%%%%%%%%%%%%%%%%%%%%%%%
\subsubsection{Binary relations}
Let us start by remarking that, depending on the cardinality $m$ of
$\Mca$, the set of all $\Mca$-cliques can be interpreted as particular
binary relations. When $m \geq 4$, let us set
$\Mca = \{\Unit_\Mca, \Att, \Btt, \Ctt, \dots\}$ so that $\Att$, $\Btt$,
and $\Ctt$ are distinguished pairwise distinct elements of $\Mca$
different from $\Unit_\Mca$. Given an $\Mca$-clique $\Pfr$ of arity
$n \geq 2$, we build a binary relation $\BinRel$ on $[n + 1]$
satisfying, for all $x < y \in [n + 1]$,
\begin{equation}\begin{split}
    x \BinRel y & \quad \mbox{ if } \Pfr(x, y) = \Att, \\
    y \BinRel x & \quad \mbox{ if } \Pfr(x, y) = \Btt, \\
    x \BinRel y \mbox{ and } y \BinRel x
        & \quad \mbox{ if } \Pfr(x, y) = \Ctt.
\end{split}\end{equation}
In particular, when $m = 2$ (resp. $m = 3$, $m = 4$),
$\Mca = \{\Unit, \Ctt\}$ (resp. $\Mca = \{\Unit, \Att, \Btt\}$,
$\Mca = \{\Unit, \Att, \Btt, \Ctt\}$) and the set of all $\Mca$-cliques
of arities $n \geq 2$ is in one-to-one correspondence with the set of all
irreflexive and symmetric (resp. irreflexive and antisymmetric,
irreflexive) binary relations on $[n + 1]$. Therefore, the operads
$\Cli\Mca$ can be interpreted as operads involving binary relations with
more or less properties.
\medbreak

%%%%%%%%%%%%%%%%%%%%%%%%%%%%%%%%%%%%%%%%%%%%%%%%%%%%%%%%%%%%%%%%%%%%%%%%
\subsubsection{Dimensions and minimal generating set}

\begin{Proposition} \label{prop:dimensions_Cli_M}
    Let $\Mca$ be a finite unitary magma. For all $n \geq 2$,
    \begin{equation} \label{equ:dimensions_Cli_M}
        \dim \Cli\Mca(n) = m^{\binom{n + 1}{2}},
    \end{equation}
    where $m := \# \Mca$.
\end{Proposition}
\begin{proof}
    By definition of the clique construction and of $\Mca$-cliques, the
    dimension of $\Cli\Mca(n)$ is the number of maps from the set
    $\left\{(x, y) \in [n + 1]^2 : x < y\right\}$ to $\Mca$. Therefore,
    when $n \geq 2$, this implies~\eqref{equ:dimensions_Cli_M}.
\end{proof}
\medbreak

From Proposition~\ref{prop:dimensions_Cli_M}, the first dimensions of
$\Cli\Mca$ depending on $m := \# \Mca$ are
\begin{subequations}
\begin{equation}
    1, 1, 1, 1, 1, 1, 1, 1,
    \qquad m = 1,
\end{equation}
\begin{equation}
    1, 8, 64, 1024, 32768, 2097152, 268435456, 68719476736,
    \qquad m = 2,
\end{equation}
\begin{multline}
    1, 27, 729, 59049, 14348907, 10460353203, 22876792454961, \\
    150094635296999121,
    \qquad m = 3,
\end{multline}
\begin{multline}
    1, 64, 4096, 1048576, 1073741824, 4398046511104,
    72057594037927936, \\
    4722366482869645213696,
    \qquad m = 4.
\end{multline}
\end{subequations}
Except for the first terms, the second one forms
Sequence~\OEIS{A006125}, the third one forms Sequence~\OEIS{A047656},
and the last one forms Sequence~\OEIS{A053763} of~\cite{Slo}.
\medbreak

\begin{Lemma} \label{lem:decomposition_clique_diagonal}
    Let $\Mca$ be a unitary magma, $\Pfr$ be an $\Mca$-clique of arity
    $n \geq 2$, and $(x, y)$ be a diagonal of $\Pfr$. Then, the
    following two assertions are equivalent:
    \begin{enumerate}[fullwidth,label={(\it\roman*)}]
        \item \label{item:decomposition_clique_diagonal_1}
        the diagonal $(x, y)$ is solid and its crossing number is $0$,
        or $(x, y)$ is not solid;
        \item \label{item:decomposition_clique_diagonal_2}
        the $\Mca$-clique $\Pfr$ can be written as $\Pfr = \Qfr \circ_x \Rfr$,
        where $\Qfr$ is an $\Mca$-clique of arity $n + x - y + 1$ and
        $\Rfr$ is an $\Mca$-clique of arity $y - x$.
    \end{enumerate}
\end{Lemma}
\begin{proof}
    Assume first that~\ref{item:decomposition_clique_diagonal_1} holds.
    Set $\Qfr$ as the $\Mca$-clique of arity $n + x - y + 1$ defined,
    for any arc $(z, t)$ where $1 \leq z < t \leq n + x - y + 2$, by
    \begin{equation}
        \Qfr(z, t) :=
        \begin{cases}
            \Pfr(z, t) & \mbox{if } t \leq x, \\
            \Pfr(z, t + y - x - 1) & \mbox{if } x + 1 \leq t, \\
            \Pfr(z + y - x - 1, t + y - x - 1) & \mbox{otherwise},
        \end{cases}
    \end{equation}
    and $\Rfr$ as the $\Mca$-clique of arity $y - x$ defined, for any
    arc $(z, t)$ where $1 \leq z < t \leq y - x + 1$, by
    \begin{equation}
        \Rfr(z, t) :=
        \begin{cases}
            \Pfr(z + x - 1, t + x - 1)
                & \mbox{if } (z, t) \ne (1, y - x + 1), \\
            \Unit_\Mca & \mbox{otherwise}.
        \end{cases}
    \end{equation}
    By following the definition of the partial composition of
    $\Cli\Mca$, one obtains $\Pfr = \Qfr \circ_x \Rfr$,
    hence~\ref{item:decomposition_clique_diagonal_2} holds.
    \smallbreak

    Assume conversely that~\ref{item:decomposition_clique_diagonal_2}
    holds. By definition of the partial composition of $\Cli\Mca$, the
    fact that $\Pfr = \Qfr \circ_x \Rfr$ implies that
    $\Pfr(x', y') = \Unit_\Mca$ for any arc $(x', y')$ such that
    $(x, y)$ and $(x', y')$ are crossing. Therefore,
    \ref{item:decomposition_clique_diagonal_1} holds.
\end{proof}
\medbreak

Let $\Primes_\Mca$ be the set of all $\Mca$-cliques $\Pfr$ of arity
$n \geq 2$ that do not satisfy the property of the statement of
Lemma~\ref{lem:decomposition_clique_diagonal}. In other words,
$\Primes_\Mca$ is the set of all $\Mca$-cliques such that, for any
(non-necessarily solid) diagonal $(x, y)$ of $\Pfr$, there is at least
one solid diagonal $(x', y')$ of $\Pfr$ such that $(x, y)$ and
$(x', y')$ are crossing. We call $\Primes_\Mca$ the set of all
\Def{prime $\Mca$-cliques}. Observe that, according to this description,
all $\Mca$-triangles are prime.
\medbreak

\begin{Proposition} \label{prop:generating_set_Cli_M}
    Let $\Mca$ be a unitary magma. The set $\Primes_\Mca$ is a minimal
    generating set of~$\Cli\Mca$.
\end{Proposition}
\begin{proof}
    We show by induction on the arity that $\Primes_\Mca$ is a
    generating set of $\Cli\Mca$. Let $\Pfr$ be an $\Mca$-clique. If
    $\Pfr$ is of arity $1$, $\Pfr = \UnitClique$ and hence $\Pfr$
    trivially belongs to $(\Cli\Mca)^{\Primes_\Mca}$ (recall that this notation stands for
    the suboperad of $\Cli\Mca$ generated by $\Primes_\Mca$). Let us assume that
    $\Pfr$ is of arity $n \geq 2$. First, if $\Pfr \in \Primes_\Mca$,
    then $\Pfr \in (\Cli\Mca)^{\Primes_\Mca}$. Otherwise, $\Pfr$ is an
    $\Mca$-clique which satisfies the description of the statement of
    Lemma~\ref{lem:decomposition_clique_diagonal}. Therefore, by this
    lemma, there are two $\Mca$-cliques $\Qfr$ and $\Rfr$ and an integer
    $x \in [|\Pfr|]$ such that $|\Qfr| < |\Pfr|$, $|\Rfr| < |\Pfr|$, and
    $\Pfr = \Qfr \circ_x \Rfr$. By induction hypothesis, $\Qfr$ and
    $\Rfr$ belong to $(\Cli\Mca)^{\Primes_\Mca}$ and hence, $\Pfr$
    also belongs to~$(\Cli\Mca)^{\Primes_\Mca}$.
    \smallbreak

    Finally, by Lemma~\ref{lem:decomposition_clique_diagonal}, if $\Pfr$
    is a prime $\Mca$-clique, $\Pfr$ cannot be expressed as a partial
    composition of prime $\Mca$-cliques. Moreover, since the space
    $\Cli\Mca(1)$ is trivial, these arguments imply that $\Primes_\Mca$
    is a minimal generating set of~$\Cli\Mca$.
\end{proof}
\medbreak

Computer experiments tell us that, when $m := \# \Mca = 2$, the first
numbers of prime $\Mca$-cliques are, size by size,
\begin{equation} \label{equ:prime_cliques_numbers_2}
    0, 8, 16, 352, 16448, 1380224.
\end{equation}
Moreover, remark that the $n$th term of this sequence is divisible
by $m^{n + 1}$ since the labels of the base and the edges of an
$\Mca$-clique $\Pfr$ have no influence on the fact that $\Pfr$ is prime.
This gives the sequence
\begin{equation} \label{equ:white_prime_cliques_numbers}
    0, 1, 1, 11, 257, 10783
\end{equation}
enumerating the first of them size by size. Besides, a prime
$\Mca$-clique $\Pfr$ is \Def{minimal} if any $\Mca$-clique obtained from
$\Pfr$ by replacing a solid arc by a non-solid one is not prime. Of
course, all minimal prime $\Mca$-cliques are white. Computer
experiments show us that when $m := \# \Mca = 2$, the numbers of
minimal prime $\Mca$-cliques begin by
\begin{equation} \label{equ:minimal_prime_cliques_numbers}
    0, 1, 1, 5, 22, 119.
\end{equation}
None of these sequences appear in~\cite{Slo} at this time.
\medbreak

%%%%%%%%%%%%%%%%%%%%%%%%%%%%%%%%%%%%%%%%%%%%%%%%%%%%%%%%%%%%%%%%%%%%%%%%
\subsubsection{Associative elements}

\begin{Proposition} \label{prop:associative_elements_Cli_M}
    Let $\Mca$ be a unitary magma and $f$ be an element of $\Cli\Mca(2)$
    of the form
    \begin{equation} \label{equ:associative_elements_Cli_M_0}
        f :=
        \sum_{\Pfr \in \Triangles_\Mca}
        \lambda_\Pfr \Pfr,
    \end{equation}
    where the $\lambda_\Pfr$, $\Pfr \in \Triangles_\Mca$, are
    coefficients of $\K$. Then, $f$ is associative if and only if
    \begin{subequations}
    \begin{equation} \label{equ:associative_elements_Cli_M_1}
        \sum_{\substack{
            \Pfr_1, \Qfr_0 \in \Mca \\
            \delta =  \Pfr_1 \Op \Qfr_0
        }}
        \lambda_{\Triangle{\Pfr_0}{\Pfr_1}{\Pfr_2}}
        \lambda_{\Triangle{\Qfr_0}{\Qfr_1}{\Qfr_2}}
        = 0,
        \qquad
        \Pfr_0, \Pfr_2, \Qfr_1, \Qfr_2 \in \Mca,
        \delta \in \bar{\Mca},
    \end{equation}
    \begin{equation} \label{equ:associative_elements_Cli_M_2}
        \sum_{\substack{
            \Pfr_1, \Qfr_0 \in \Mca \\
            \Pfr_1 \Op \Qfr_0 = \Unit_\Mca
        }}
        \lambda_{\Triangle{\Pfr_0}{\Pfr_1}{\Pfr_2}}
        \lambda_{\Triangle{\Qfr_0}{\Qfr_1}{\Qfr_2}}
        -
        \lambda_{\Triangle{\Pfr_0}{\Qfr_1}{\Pfr_1}}
        \lambda_{\Triangle{\Qfr_0}{\Qfr_2}{\Pfr_2}}
        = 0,
        \qquad
        \Pfr_0, \Pfr_2, \Qfr_1, \Qfr_2 \in \Mca,
    \end{equation}
    \begin{equation} \label{equ:associative_elements_Cli_M_3}
        \sum_{\substack{
            \Pfr_2, \Qfr_0 \in \Mca \\
            \delta =  \Pfr_2 \Op \Qfr_0
        }}
        \lambda_{\Triangle{\Pfr_0}{\Pfr_1}{\Pfr_2}}
        \lambda_{\Triangle{\Qfr_0}{\Qfr_1}{\Qfr_2}}
        = 0,
        \qquad
        \Pfr_0, \Pfr_1, \Qfr_1, \Qfr_2 \in \Mca,
        \delta \in \bar{\Mca}.
    \end{equation}
    \end{subequations}
\end{Proposition}
\begin{proof}
    The element $f$ defined in~\eqref{equ:associative_elements_Cli_M_0}
    is associative if and only if $f \circ_1 f - f \circ_2 f = 0$.
    Therefore, this property is equivalent to the fact that
    \begin{equation}\begin{split}
        \label{equ:associative_elements_Cli_M_demo}
        f \circ_1 f - f \circ_2 f
        & =
        \left(
        \sum_{\substack{
            \Pfr, \Qfr \in \Triangles_\Mca \\
            \delta := \Pfr_1 \Op \Qfr_0 \ne \Unit_\Mca
        }}
        \lambda_\Pfr \lambda_\Qfr
        \SquareRight{\Qfr_1}{\Qfr_2}{\Pfr_2}{\Pfr_0}{\delta}
        \right)
        +
        \left(
        \sum_{\substack{
            \Pfr, \Qfr \in \Triangles_\Mca \\
            \Pfr_1 \Op \Qfr_0 = \Unit_\Mca
        }}
        \lambda_\Pfr \lambda_\Qfr
        \SquareN{\Qfr_1}{\Qfr_2}{\Pfr_2}{\Pfr_0}
        \right) \\
        & \quad -
        \left(
        \sum_{\substack{
            \Pfr, \Qfr \in \Triangles_\Mca \\
            \delta := \Pfr_2 \Op \Qfr_0 \ne \Unit_\Mca
        }}
        \lambda_\Pfr \lambda_\Qfr
        \SquareLeft{\Pfr_1}{\Qfr_1}{\Qfr_2}{\Pfr_0}{\delta}
        \right)
        -
        \left(
        \sum_{\substack{
            \Pfr, \Qfr \in \Triangles_\Mca \\
            \Pfr_2 \Op \Qfr_0 = \Unit_\Mca
        }}
        \lambda_\Pfr \lambda_\Qfr
        \SquareN{\Pfr_1}{\Qfr_1}{\Qfr_2}{\Pfr_0}
        \right) \\
        & =
        \left(
        \sum_{\substack{
            \Pfr_0, \Pfr_2, \Qfr_1, \Qfr_2 \in \Mca \\
            \delta \in \bar{\Mca}
        }}
        \left(
        \sum_{\substack{
            \Pfr_1, \Qfr_0 \in \Mca \\
            \delta =  \Pfr_1 \Op \Qfr_0
        }}
        \lambda_{\Triangle{\Pfr_0}{\Pfr_1}{\Pfr_2}}
        \lambda_{\Triangle{\Qfr_0}{\Qfr_1}{\Qfr_2}}
        \right)
        \SquareRight{\Qfr_1}{\Qfr_2}{\Pfr_2}{\Pfr_0}{\delta}
        \right) \\
        & \quad +
        \left(
        \sum_{
            \Pfr_0, \Pfr_2, \Qfr_1, \Qfr_2 \in \Mca
        }
        \left(
        \sum_{\substack{
            \Pfr_1, \Qfr_0 \in \Mca \\
            \Pfr_1 \Op \Qfr_0 = \Unit_\Mca
        }}
        \lambda_{\Triangle{\Pfr_0}{\Pfr_1}{\Pfr_2}}
        \lambda_{\Triangle{\Qfr_0}{\Qfr_1}{\Qfr_2}}
        -
        \lambda_{\Triangle{\Pfr_0}{\Qfr_1}{\Pfr_1}}
        \lambda_{\Triangle{\Qfr_0}{\Qfr_2}{\Pfr_2}}
        \right)
        \SquareN{\Qfr_1}{\Qfr_2}{\Pfr_2}{\Pfr_0}
        \right) \\
        & \quad -
        \left(
        \sum_{\substack{
            \Pfr_0, \Pfr_1, \Qfr_1, \Qfr_2 \in \Mca \\
            \delta \in \bar{\Mca}
        }}
        \left(
        \sum_{\substack{
            \Pfr_2, \Qfr_0 \in \Mca \\
            \delta =  \Pfr_2 \Op \Qfr_0
        }}
        \lambda_{\Triangle{\Pfr_0}{\Pfr_1}{\Pfr_2}}
        \lambda_{\Triangle{\Qfr_0}{\Qfr_1}{\Qfr_2}}
        \right)
        \SquareLeft{\Pfr_1}{\Qfr_1}{\Qfr_2}{\Pfr_0}{\delta}
        \right) \\
        & = 0,
    \end{split}\end{equation}
    and hence, is equivalent to the fact
    that~\eqref{equ:associative_elements_Cli_M_1},
    \eqref{equ:associative_elements_Cli_M_2},
    and~\eqref{equ:associative_elements_Cli_M_3} hold.
\end{proof}
\medbreak

For instance, by Proposition~\ref{prop:associative_elements_Cli_M}, the
binary elements
\begin{subequations}
\begin{equation}
    \Triangle{1}{1}{1},
\end{equation}
\begin{equation}
    \TriangleEEE{}{}{}
    +
    \TriangleEXE{}{1}{}
    -
    \TriangleXEE{1}{}{}
    +
    \TriangleEEX{}{}{1}
    -
    \TriangleXXE{1}{1}{}
    +
    \TriangleEXX{}{1}{1}
    -
    \TriangleXEX{1}{}{1}
    -
    \Triangle{1}{1}{1}
\end{equation}
\end{subequations}
of $\Cli\N_2$ are associative, and the binary elements
\begin{subequations}
\begin{equation}
    \TriangleEXX{}{0}{0}
    -
    \Triangle{0}{0}{0},
\end{equation}
\begin{equation}
    \TriangleXEE{0}{}{}
    -
    \TriangleXXE{0}{0}{}
    -
    \TriangleXEX{0}{}{0}
    +
    \Triangle{0}{0}{0}
\end{equation}
\end{subequations}
of $\Cli\Dbb_0$ are associative.
\medbreak

%%%%%%%%%%%%%%%%%%%%%%%%%%%%%%%%%%%%%%%%%%%%%%%%%%%%%%%%%%%%%%%%%%%%%%%%
\subsubsection{Symmetries}
Let $\Returned : \Cli\Mca \to \Cli\Mca$ be the linear map sending any
$\Mca$-clique $\Pfr$ of arity $n$ to the $\Mca$-clique $\Returned(\Pfr)$
of the same arity such that, for any arc $(x, y)$ where
$1 \leq x < y \leq n + 1$,
\begin{equation} \label{equ:returned_map_Cli_M}
    \left(\Returned(\Pfr)\right)(x, y) := \Pfr(n - y + 2, n - x + 2).
\end{equation}
Graphically, $\Returned(\Pfr)$ is the $\Mca$-clique obtained by
applying on $\Pfr$ a reflection trough the vertical line passing by its
base. For instance, one has in $\Cli\Z$,
\begin{equation}
    \Returned\left(
    \begin{tikzpicture}[scale=.85,Centering]
        \node[CliquePoint](1)at(-0.50,-0.87){};
        \node[CliquePoint](2)at(-1.00,-0.00){};
        \node[CliquePoint](3)at(-0.50,0.87){};
        \node[CliquePoint](4)at(0.50,0.87){};
        \node[CliquePoint](5)at(1.00,0.00){};
        \node[CliquePoint](6)at(0.50,-0.87){};
        \draw[CliqueEdge](1)edge[]node[CliqueLabel]
            {\begin{math}1\end{math}}(2);
        \draw[CliqueEdge](1)edge[bend left=30]node[CliqueLabel]
            {\begin{math}-2\end{math}}(5);
        \draw[CliqueEmptyEdge](1)edge[]node[CliqueLabel]{}(6);
        \draw[CliqueEdge](2)edge[]node[CliqueLabel]
            {\begin{math}-2\end{math}}(3);
        \draw[CliqueEmptyEdge](3)edge[]node[CliqueLabel]{}(4);
        \draw[CliqueEdge](3)edge[bend right=30]node[CliqueLabel]
            {\begin{math}1\end{math}}(5);
        \draw[CliqueEmptyEdge](4)edge[]node[CliqueLabel]{}(5);
        \draw[CliqueEmptyEdge](5)edge[]node[CliqueLabel]{}(6);
    \end{tikzpicture}
    \right)
    \enspace = \enspace
    \begin{tikzpicture}[scale=.85,Centering]
        \node[CliquePoint](1)at(-0.50,-0.87){};
        \node[CliquePoint](2)at(-1.00,-0.00){};
        \node[CliquePoint](3)at(-0.50,0.87){};
        \node[CliquePoint](4)at(0.50,0.87){};
        \node[CliquePoint](5)at(1.00,0.00){};
        \node[CliquePoint](6)at(0.50,-0.87){};
        \draw[CliqueEmptyEdge](1)edge[]node[CliqueLabel]{}(2);
        \draw[CliqueEmptyEdge](1)edge[]node[CliqueLabel]{}(6);
        \draw[CliqueEmptyEdge](2)edge[]node[CliqueLabel]{}(3);
        \draw[CliqueEdge](2)edge[bend right=30]node[CliqueLabel]
            {\begin{math}1\end{math}}(4);
        \draw[CliqueEdge](2)edge[bend left=30]node[CliqueLabel]
            {\begin{math}-2\end{math}}(6);
        \draw[CliqueEmptyEdge](3)edge[]node[CliqueLabel]{}(4);
        \draw[CliqueEdge](4)edge[]node[CliqueLabel]
            {\begin{math}-2\end{math}}(5);
        \draw[CliqueEdge](5)edge[]node[CliqueLabel]
            {\begin{math}1\end{math}}(6);
    \end{tikzpicture}\,.
\end{equation}
\medbreak

\begin{Proposition} \label{prop:symmetries_Cli_M}
    Let $\Mca$ be a unitary magma. Then, the group of symmetries of
    $\Cli\Mca$ contains the map $\Returned$ and all the maps
    $\Cli\theta$ where $\theta$ is a unitary magma automorphisms
    of~$\Mca$.
\end{Proposition}
\begin{proof}
    When $\theta$ is a unitary magma automorphism of $\Mca$, since by
    Theorem~\ref{thm:clique_construction} $\Cli$ is a functor
    respecting bijections, $\Cli\theta$ is an operad automorphism of
    $\Cli\Mca$. Hence, $\Cli\theta$ belongs to the group of symmetries
    of $\Cli\Mca$. Moreover, the fact that $\Returned$ belongs to the
    group of symmetries of $\Cli\Mca$ can be established by showing that
    this map is an antiautomorphism of $\Cli\Mca$, directly from the
    definition of the partial composition of $\Cli\Mca$ and that
    of~$\Returned$.
\end{proof}
\medbreak

%%%%%%%%%%%%%%%%%%%%%%%%%%%%%%%%%%%%%%%%%%%%%%%%%%%%%%%%%%%%%%%%%%%%%%%%
\subsubsection{Basic set-operad basis}
A unitary magma $\Mca$ is \Def{right cancelable} if for any
$x, y, z \in \Mca$, $y \Op x = z \Op x$ implies $y = z$.
\medbreak

\begin{Proposition} \label{prop:basic_Cli_M}
    Let $\Mca$ be a unitary magma. The fundamental basis of $\Cli\Mca$
    is a basic set-operad basis if and only if $\Mca$ is right
    cancelable.
\end{Proposition}
\begin{proof}
    Assume first that $\Mca$ is right cancelable. Let $n \geq 1$, $i \in [n]$, and $\Pfr$,
    $\Pfr'$, and $\Qfr$ be three $\Mca$-cliques such that $\Pfr$ and $\Pfr'$ are of arity
    $n$. If $\circ_i^\Qfr(\Pfr) = \circ_i^\Qfr(\Pfr')$, we have $\Pfr \circ_i \Qfr = \Pfr'
    \circ_i \Qfr$. By definition of the partial composition map of $\Cli\Mca$, we have
    $\Pfr(x, y) = \Pfr'(x, y)$ for all arcs $(x, y)$ where $1 \leq x < y \leq n + 1$ and
    $(x, y) \ne (i, i + 1)$.  Moreover, we have
    \begin{math}
        \Pfr_i \Op \Qfr_0 = \Pfr'_i \Op \Qfr_0.
    \end{math}
    Since $\Mca$ is right cancelable, this implies that
    $\Pfr_i = \Pfr'_i$, and hence, $\Pfr = \Pfr'$. This shows that the
    maps $\circ_i^\Qfr$ are injective and thus, that the fundamental
    basis of $\Cli\Mca$ is a basic set-operad basis.
    \smallbreak

    Conversely, assume that the fundamental basis of $\Cli\Mca$ is a
    basic set-operad basis. Then, in particular, for all $n \geq 1$ and
    all $\Mca$-cliques $\Pfr$, $\Pfr'$, and $\Qfr$ such that $\Pfr$ and
    $\Pfr'$ are of arity $n$, $\circ_1^\Qfr(\Pfr) = \circ_1^\Qfr(\Pfr')$
    implies $\Pfr = \Pfr'$. This is equivalent to the statement that
    \begin{math}
        \Pfr_1 \Op \Qfr_0 = \Pfr'_1 \Op \Qfr_0
    \end{math}
    implies $\Pfr_1 = \Pfr'_1$. This amount exactly to the statement that $\Mca$
    is right cancelable.
\end{proof}
\medbreak

%%%%%%%%%%%%%%%%%%%%%%%%%%%%%%%%%%%%%%%%%%%%%%%%%%%%%%%%%%%%%%%%%%%%%%%%
\subsubsection{Cyclic operad structure}
Let $\rho : \Cli\Mca \to \Cli\Mca$ be the linear map sending any
$\Mca$-clique $\Pfr$ of arity $n$ to the $\Mca$-clique $\rho(\Pfr)$ of
the same arity such that, for any arc $(x, y)$ where
$1 \leq x < y \leq n + 1$,
\begin{equation} \label{equ:rotation_map_Cli_M}
    (\rho(\Pfr))(x, y) :=
    \begin{cases}
        \Pfr(x + 1, y + 1) & \mbox{if } y \leq n, \\
        \Pfr(1, x + 1) & \mbox{otherwise (} y = n + 1 \mbox{)}.
    \end{cases}
\end{equation}
Graphically, $\rho(\Pfr)$ is the $\Mca$-clique obtained by
applying a rotation of one step of $\Pfr$ in counterclockwise
direction. For instance, one has in $\Cli\Z$,
\begin{equation}
    \rho\left(
    \begin{tikzpicture}[scale=.85,Centering]
        \node[CliquePoint](1)at(-0.50,-0.87){};
        \node[CliquePoint](2)at(-1.00,-0.00){};
        \node[CliquePoint](3)at(-0.50,0.87){};
        \node[CliquePoint](4)at(0.50,0.87){};
        \node[CliquePoint](5)at(1.00,0.00){};
        \node[CliquePoint](6)at(0.50,-0.87){};
        \draw[CliqueEdge](1)edge[]node[CliqueLabel]
            {\begin{math}1\end{math}}(2);
        \draw[CliqueEdge](1)edge[bend left=30]node[CliqueLabel]
            {\begin{math}-2\end{math}}(5);
        \draw[CliqueEmptyEdge](1)edge[]node[CliqueLabel]{}(6);
        \draw[CliqueEdge](2)edge[]node[CliqueLabel]
            {\begin{math}-2\end{math}}(3);
        \draw[CliqueEmptyEdge](3)edge[]node[CliqueLabel]{}(4);
        \draw[CliqueEdge](3)edge[bend right=30]node[CliqueLabel]
            {\begin{math}1\end{math}}(5);
        \draw[CliqueEmptyEdge](4)edge[]node[CliqueLabel]{}(5);
        \draw[CliqueEmptyEdge](5)edge[]node[CliqueLabel]{}(6);
    \end{tikzpicture}
    \right)
    \enspace = \enspace
    \begin{tikzpicture}[scale=.85,Centering]
        \node[CliquePoint](1)at(-0.50,-0.87){};
        \node[CliquePoint](2)at(-1.00,-0.00){};
        \node[CliquePoint](3)at(-0.50,0.87){};
        \node[CliquePoint](4)at(0.50,0.87){};
        \node[CliquePoint](5)at(1.00,0.00){};
        \node[CliquePoint](6)at(0.50,-0.87){};
        \draw[CliqueEdge](1)edge[]node[CliqueLabel]
            {\begin{math}-2\end{math}}(2);
        \draw[CliqueEdge](1)edge[]node[CliqueLabel]
            {\begin{math}1\end{math}}(6);
        \draw[CliqueEmptyEdge](2)edge[]node[CliqueLabel]{}(3);
        \draw[CliqueEdge](2)edge[bend right=30]node[CliqueLabel]
            {\begin{math}1\end{math}}(4);
        \draw[CliqueEmptyEdge](3)edge[]node[CliqueLabel]{}(4);
        \draw[CliqueEmptyEdge](4)edge[]node[CliqueLabel]{}(5);
        \draw[CliqueEdge](4)edge[bend right=30]node[CliqueLabel]
            {\begin{math}-2\end{math}}(6);
        \draw[CliqueEmptyEdge](5)edge[]node[CliqueLabel]{}(6);
    \end{tikzpicture}\,.
\end{equation}
\medbreak

\begin{Proposition} \label{prop:cyclic_Cli_M}
    Let $\Mca$ be a unitary magma. The map $\rho$ is a rotation map
    of $\Cli\Mca$, endowing this operad with a cyclic operad structure.
\end{Proposition}
\begin{proof}
    The fact that $\rho$ is a rotation map for $\Cli\Mca$ follows from
    a technical but straightforward verification of the fact that
    Relations~\eqref{equ:rotation_map_1}, \eqref{equ:rotation_map_2},
    and~\eqref{equ:rotation_map_3} hold.
\end{proof}
\medbreak

%%%%%%%%%%%%%%%%%%%%%%%%%%%%%%%%%%%%%%%%%%%%%%%%%%%%%%%%%%%%%%%%%%%%%%%%
\subsubsection{Alternative bases}
If $\Pfr$ and $\Qfr$ are two $\Mca$-cliques of the same arity, the
\Def{Hamming distance} $\Hamming(\Pfr, \Qfr)$ between $\Pfr$ and $\Qfr$
is the number of arcs $(x, y)$ such that $\Pfr(x, y) \ne \Qfr(x, y)$.
Let $\OrdBE$ be the partial order relation on the set of all
$\Mca$-cliques, where, for any $\Mca$-cliques $\Pfr$ and $\Qfr$, one
has $\Pfr \OrdBE \Qfr$ if $\Qfr$ can be obtained from $\Pfr$ by
replacing some labels $\Unit_\Mca$ of its edges or its base by other
labels of $\Mca$. In the same way, let $\OrdD$ be the partial order
on the same set where $\Pfr \OrdD \Qfr$ if $\Qfr$ can be obtained from
$\Pfr$ by replacing some labels $\Unit_\Mca$ of its diagonals by other
labels of $\Mca$.
\medbreak

For all $\Mca$-cliques $\Pfr$, let us introduce the elements of $\Cli\Mca$ defined by
\begin{subequations}
\begin{equation}
    \Hsf_\Pfr :=
    \sum_{\substack{
        \Pfr' \in \Cliques_\Mca \\
        \Pfr' \OrdBE \Pfr
    }}
    \Pfr',
\end{equation}
and
\begin{equation}
    \Ksf_\Pfr :=
    \sum_{\substack{
        \Pfr' \in \Cliques_\Mca \\
        \Pfr' \OrdD \Pfr
    }}
    (-1)^{\Hamming(\Pfr', \Pfr)}
    \Pfr'.
\end{equation}
\end{subequations}
For instance, in~$\Cli\Z$,
\begin{subequations}
\begin{equation}
    \Hsf_{
    \begin{tikzpicture}[scale=.6,Centering]
        \node[CliquePoint](1)at(-0.59,-0.81){};
        \node[CliquePoint](2)at(-0.95,0.31){};
        \node[CliquePoint](3)at(-0.00,1.00){};
        \node[CliquePoint](4)at(0.95,0.31){};
        \node[CliquePoint](5)at(0.59,-0.81){};
        \draw[CliqueEmptyEdge](1)edge[]node[]{}(2);
        \draw[CliqueEmptyEdge](1)edge[]node[]{}(5);
        \draw[CliqueEmptyEdge](2)edge[]node[]{}(3);
        \draw[CliqueEdge](2)edge[bend left=30]node[CliqueLabel,near start]
            {\begin{math}1\end{math}}(5);
        \draw[CliqueEdge](3)edge[]node[CliqueLabel]
            {\begin{math}1\end{math}}(4);
        \draw[CliqueEdge](4)edge[]node[CliqueLabel]
            {\begin{math}2\end{math}}(5);
        \draw[CliqueEdge](1)edge[bend right=30]node[CliqueLabel,near end]
            {\begin{math}2\end{math}}(3);
    \end{tikzpicture}}
    =
    \begin{tikzpicture}[scale=.6,Centering]
        \node[CliquePoint](1)at(-0.59,-0.81){};
        \node[CliquePoint](2)at(-0.95,0.31){};
        \node[CliquePoint](3)at(-0.00,1.00){};
        \node[CliquePoint](4)at(0.95,0.31){};
        \node[CliquePoint](5)at(0.59,-0.81){};
        \draw[CliqueEmptyEdge](1)edge[]node[]{}(2);
        \draw[CliqueEmptyEdge](1)edge[]node[]{}(5);
        \draw[CliqueEmptyEdge](2)edge[]node[]{}(3);
        \draw[CliqueEdge](2)edge[bend left=30]node[CliqueLabel,near start]
            {\begin{math}1\end{math}}(5);
        \draw[CliqueEmptyEdge](3)edge[]node[]{}(4);
        \draw[CliqueEmptyEdge](4)edge[]node[]{}(5);
        \draw[CliqueEdge](1)edge[bend right=30]node[CliqueLabel,near end]
            {\begin{math}2\end{math}}(3);
    \end{tikzpicture}
    +
    \begin{tikzpicture}[scale=.6,Centering]
        \node[CliquePoint](1)at(-0.59,-0.81){};
        \node[CliquePoint](2)at(-0.95,0.31){};
        \node[CliquePoint](3)at(-0.00,1.00){};
        \node[CliquePoint](4)at(0.95,0.31){};
        \node[CliquePoint](5)at(0.59,-0.81){};
        \draw[CliqueEmptyEdge](1)edge[]node[]{}(2);
        \draw[CliqueEmptyEdge](1)edge[]node[]{}(5);
        \draw[CliqueEmptyEdge](2)edge[]node[]{}(3);
        \draw[CliqueEdge](2)edge[bend left=30]node[CliqueLabel,near start]
            {\begin{math}1\end{math}}(5);
        \draw[CliqueEmptyEdge](3)edge[]node[]{}(4);
        \draw[CliqueEdge](4)edge[]node[CliqueLabel]
            {\begin{math}2\end{math}}(5);
        \draw[CliqueEdge](1)edge[bend right=30]node[CliqueLabel,near end]
            {\begin{math}2\end{math}}(3);
    \end{tikzpicture}
    +
    \begin{tikzpicture}[scale=.6,Centering]
        \node[CliquePoint](1)at(-0.59,-0.81){};
        \node[CliquePoint](2)at(-0.95,0.31){};
        \node[CliquePoint](3)at(-0.00,1.00){};
        \node[CliquePoint](4)at(0.95,0.31){};
        \node[CliquePoint](5)at(0.59,-0.81){};
        \draw[CliqueEmptyEdge](1)edge[]node[]{}(2);
        \draw[CliqueEmptyEdge](1)edge[]node[]{}(5);
        \draw[CliqueEmptyEdge](2)edge[]node[]{}(3);
        \draw[CliqueEdge](2)edge[bend left=30]node[CliqueLabel,near start]
            {\begin{math}1\end{math}}(5);
        \draw[CliqueEdge](3)edge[]node[CliqueLabel]
            {\begin{math}1\end{math}}(4);
        \draw[CliqueEmptyEdge](4)edge[]node[]{}(5);
        \draw[CliqueEdge](1)edge[bend right=30]node[CliqueLabel,near end]
            {\begin{math}2\end{math}}(3);
    \end{tikzpicture}
    +
    \begin{tikzpicture}[scale=.6,Centering]
        \node[CliquePoint](1)at(-0.59,-0.81){};
        \node[CliquePoint](2)at(-0.95,0.31){};
        \node[CliquePoint](3)at(-0.00,1.00){};
        \node[CliquePoint](4)at(0.95,0.31){};
        \node[CliquePoint](5)at(0.59,-0.81){};
        \draw[CliqueEmptyEdge](1)edge[]node[]{}(2);
        \draw[CliqueEmptyEdge](1)edge[]node[]{}(5);
        \draw[CliqueEmptyEdge](2)edge[]node[]{}(3);
        \draw[CliqueEdge](2)edge[bend left=30]node[CliqueLabel,near start]
            {\begin{math}1\end{math}}(5);
        \draw[CliqueEdge](3)edge[]node[CliqueLabel]
            {\begin{math}1\end{math}}(4);
        \draw[CliqueEdge](4)edge[]node[CliqueLabel]
            {\begin{math}2\end{math}}(5);
        \draw[CliqueEdge](1)edge[bend right=30]node[CliqueLabel,near end]
            {\begin{math}2\end{math}}(3);
    \end{tikzpicture}\,,
\end{equation}
\begin{equation}
    \Ksf_{
    \begin{tikzpicture}[scale=.6,Centering]
        \node[CliquePoint](1)at(-0.59,-0.81){};
        \node[CliquePoint](2)at(-0.95,0.31){};
        \node[CliquePoint](3)at(-0.00,1.00){};
        \node[CliquePoint](4)at(0.95,0.31){};
        \node[CliquePoint](5)at(0.59,-0.81){};
        \draw[CliqueEmptyEdge](1)edge[]node[]{}(2);
        \draw[CliqueEmptyEdge](1)edge[]node[]{}(5);
        \draw[CliqueEmptyEdge](2)edge[]node[]{}(3);
        \draw[CliqueEdge](2)edge[bend left=30]node[CliqueLabel,near start]
            {\begin{math}1\end{math}}(5);
        \draw[CliqueEdge](3)edge[]node[CliqueLabel]
            {\begin{math}1\end{math}}(4);
        \draw[CliqueEdge](4)edge[]node[CliqueLabel]
            {\begin{math}2\end{math}}(5);
        \draw[CliqueEdge](1)edge[bend right=30]node[CliqueLabel,near end]
            {\begin{math}2\end{math}}(3);
    \end{tikzpicture}}
    =
    \begin{tikzpicture}[scale=.6,Centering]
        \node[CliquePoint](1)at(-0.59,-0.81){};
        \node[CliquePoint](2)at(-0.95,0.31){};
        \node[CliquePoint](3)at(-0.00,1.00){};
        \node[CliquePoint](4)at(0.95,0.31){};
        \node[CliquePoint](5)at(0.59,-0.81){};
        \draw[CliqueEmptyEdge](1)edge[]node[]{}(2);
        \draw[CliqueEmptyEdge](1)edge[]node[]{}(5);
        \draw[CliqueEmptyEdge](2)edge[]node[]{}(3);
        \draw[CliqueEdge](2)edge[bend left=30]node[CliqueLabel,near start]
            {\begin{math}1\end{math}}(5);
        \draw[CliqueEdge](3)edge[]node[CliqueLabel]
            {\begin{math}1\end{math}}(4);
        \draw[CliqueEdge](4)edge[]node[CliqueLabel]
            {\begin{math}2\end{math}}(5);
        \draw[CliqueEdge](1)edge[bend right=30]node[CliqueLabel,near end]
            {\begin{math}2\end{math}}(3);
    \end{tikzpicture}
    -
    \begin{tikzpicture}[scale=.6,Centering]
        \node[CliquePoint](1)at(-0.59,-0.81){};
        \node[CliquePoint](2)at(-0.95,0.31){};
        \node[CliquePoint](3)at(-0.00,1.00){};
        \node[CliquePoint](4)at(0.95,0.31){};
        \node[CliquePoint](5)at(0.59,-0.81){};
        \draw[CliqueEmptyEdge](1)edge[]node[]{}(2);
        \draw[CliqueEmptyEdge](1)edge[]node[]{}(5);
        \draw[CliqueEmptyEdge](2)edge[]node[]{}(3);
        \draw[CliqueEdge](3)edge[]node[CliqueLabel]
            {\begin{math}1\end{math}}(4);
        \draw[CliqueEdge](4)edge[]node[CliqueLabel]
            {\begin{math}2\end{math}}(5);
        \draw[CliqueEdge](1)edge[bend right=30]node[CliqueLabel]
            {\begin{math}2\end{math}}(3);
    \end{tikzpicture}
    -
    \begin{tikzpicture}[scale=.6,Centering]
        \node[CliquePoint](1)at(-0.59,-0.81){};
        \node[CliquePoint](2)at(-0.95,0.31){};
        \node[CliquePoint](3)at(-0.00,1.00){};
        \node[CliquePoint](4)at(0.95,0.31){};
        \node[CliquePoint](5)at(0.59,-0.81){};
        \draw[CliqueEmptyEdge](1)edge[]node[]{}(2);
        \draw[CliqueEmptyEdge](1)edge[]node[]{}(5);
        \draw[CliqueEmptyEdge](2)edge[]node[]{}(3);
        \draw[CliqueEdge](2)edge[bend left=30]node[CliqueLabel]
            {\begin{math}1\end{math}}(5);
        \draw[CliqueEdge](3)edge[]node[CliqueLabel]
            {\begin{math}1\end{math}}(4);
        \draw[CliqueEdge](4)edge[]node[CliqueLabel]
            {\begin{math}2\end{math}}(5);
    \end{tikzpicture}
    +
    \begin{tikzpicture}[scale=.6,Centering]
        \node[CliquePoint](1)at(-0.59,-0.81){};
        \node[CliquePoint](2)at(-0.95,0.31){};
        \node[CliquePoint](3)at(-0.00,1.00){};
        \node[CliquePoint](4)at(0.95,0.31){};
        \node[CliquePoint](5)at(0.59,-0.81){};
        \draw[CliqueEmptyEdge](1)edge[]node[]{}(2);
        \draw[CliqueEmptyEdge](1)edge[]node[]{}(5);
        \draw[CliqueEmptyEdge](2)edge[]node[]{}(3);
        \draw[CliqueEdge](3)edge[]node[CliqueLabel]
            {\begin{math}1\end{math}}(4);
        \draw[CliqueEdge](4)edge[]node[CliqueLabel]
            {\begin{math}2\end{math}}(5);
    \end{tikzpicture}\,.
\end{equation}
\end{subequations}

Since by Möbius inversion, one has for any $\Mca$-clique $\Pfr$,
\begin{equation}
    \sum_{\substack{
        \Pfr' \in \Cliques_\Mca \\
        \Pfr' \OrdBE \Pfr
    }}
    (-1)^{\Hamming(\Pfr', \Pfr)}
    \Hsf_{\Pfr'}
    =
    \Pfr
    =
    \sum_{\substack{
        \Pfr' \in \Cliques_\Mca \\
        \Pfr' \OrdD \Pfr
    }}
    \Ksf_{\Pfr'},
\end{equation}
by triangularity, the family of all the $\Hsf_\Pfr$ (resp. $\Ksf_\Pfr$)
forms a  basis of $\Cli\Mca$ called the \Def{$\Hsf$-basis} (resp. the
\Def{$\Ksf$-basis}).
\medbreak

If $\Pfr$ is an $\Mca$-clique, $\Del_0(\Pfr)$ (resp. $\Del_i(\Pfr)$) is
the $\Mca$-clique obtained by replacing the label of the base
(resp. $i$th edge) of $\Pfr$ by $\Unit_\Mca$.
\medbreak

\begin{Proposition} \label{prop:composition_Cli_M_basis_H}
    Let $\Mca$ be a unitary magma. The partial composition of $\Cli\Mca$ can be
    expressed in terms of the $\Hsf$-basis, for any $\Mca$-cliques $\Pfr$ and
    $\Qfr$ different from $\UnitClique$ and any valid integer $i$, by
    \begin{small}
    \begin{equation}
        \Hsf_\Pfr \circ_i \Hsf_\Qfr
        =
        \begin{cases}
            \Hsf_{\Pfr \circ_i \Qfr}
            + \Hsf_{\Del_i(\Pfr) \circ_i \Qfr}
            + \Hsf_{\Pfr \circ_i \Del_0(\Qfr)}
            + \Hsf_{\Del_i(\Pfr) \circ_i \Del_0(\Qfr)}
                & \mbox{if } \Pfr_i \ne \Unit_\Mca \mbox{ and }
                    \Qfr_0 \ne \Unit_\Mca, \\
            \Hsf_{\Pfr \circ_i \Qfr}
            + \Hsf_{\Del_i(\Pfr) \circ_i \Qfr}
                & \mbox{if } \Pfr_i \ne \Unit_\Mca \mbox{ and }
                    \Qfr_0 = \Unit_\Mca, \\
            \Hsf_{\Pfr \circ_i \Qfr}
            + \Hsf_{\Pfr \circ_i \Del_0(\Qfr)}
                & \mbox{if } \Pfr_i = \Unit_\Mca \mbox{ and }
                \Qfr_0 \ne \Unit_\Mca, \\
            \Hsf_{\Pfr \circ_i \Qfr} & \mbox{otherwise}.
        \end{cases}
    \end{equation}
    \end{small}
\end{Proposition}
\begin{proof}
    From the definition of the $\Hsf$-basis, we have
    \begin{equation}\begin{split}
        \label{equ:composition_Cli_M_basis_H_demo}
        \Hsf_\Pfr \circ_i \Hsf_\Qfr
        & =
        \sum_{\substack{
            \Pfr', \Qfr' \in \Cliques_\Mca \\
            \Pfr' \OrdBE \Pfr \\
            \Qfr' \OrdBE \Qfr
        }}
        \Pfr' \circ_i \Qfr' \\
        & =
        \sum_{\substack{
            \Pfr', \Qfr' \in \Cliques_\Mca \\
            \Pfr' \OrdBE \Pfr \\
            \Qfr' \OrdBE \Qfr \\
            \Pfr'_i \ne \Unit_\Mca \\
            \Qfr'_0 \ne \Unit_\Mca
        }}
        \Pfr' \circ_i \Qfr'
        +
        \sum_{\substack{
            \Pfr', \Qfr' \in \Cliques_\Mca \\
            \Pfr' \OrdBE \Pfr \\
            \Qfr' \OrdBE \Qfr \\
            \Pfr'_i \ne \Unit_\Mca \\
            \Qfr'_0 = \Unit_\Mca
        }}
        \Pfr' \circ_i \Qfr'
        +
        \sum_{\substack{
            \Pfr', \Qfr' \in \Cliques_\Mca \\
            \Pfr' \OrdBE \Pfr \\
            \Qfr' \OrdBE \Qfr \\
            \Pfr'_i = \Unit_\Mca \\
            \Qfr'_0 \ne \Unit_\Mca
        }}
        \Pfr' \circ_i \Qfr'
        +
        \sum_{\substack{
            \Pfr', \Qfr' \in \Cliques_\Mca \\
            \Pfr' \OrdBE \Pfr \\
            \Qfr' \OrdBE \Qfr \\
            \Pfr'_i = \Unit_\Mca \\
            \Qfr'_0 = \Unit_\Mca
        }}
        \Pfr' \circ_i \Qfr'.
    \end{split}\end{equation}
    Let $s_1$ (resp. $s_2$, $s_3$, $s_4$) be the first (resp. second,
    third, fourth) summand of the right-hand side
    of~\eqref{equ:composition_Cli_M_basis_H_demo}. There are four cases
    to explore depending on whether the $i$th edge of $\Pfr$ and the
    base of $\Qfr$ are solid or not. From the definition of the
    $\Hsf$-basis and of the partial order relation $\OrdBE$, we have
    that
    \begin{enumerate}[fullwidth,label=(\alph*)]
        \item when $\Pfr_i \ne \Unit_\Mca$ and $\Qfr_0 \ne \Unit_\Mca$,
        $s_1 = \Hsf_{\Pfr \circ_i \Qfr}$,
        $s_2 = \Hsf_{\Pfr \circ_i \Del_0(\Qfr)}$,
        $s_3 = \Hsf_{\Del_i(\Pfr) \circ_i \Qfr}$, and
        $s_4 = \Hsf_{\Del_i(\Pfr) \circ_i \Del_0(\Qfr)}$;
        \item when $\Pfr_i \ne \Unit_\Mca$ and $\Qfr_0 = \Unit_\Mca$,
        $s_1 = 0$, $s_2 = \Hsf_{\Pfr \circ_i \Qfr}$,
        $s_3 = 0$, and $s_4 = \Hsf_{\Del_i(\Pfr) \circ_i \Qfr}$;
        \item when $\Pfr_i = \Unit_\Mca$ and $\Qfr_0 \ne \Unit_\Mca$,
        $s_1 = 0$, $s_2 = 0$, $s_3 = \Hsf_{\Pfr \circ_i \Qfr}$, and
        $s_4 = \Hsf_{\Pfr \circ_i \Del_0(\Qfr)}$;
        \item and when $\Pfr_i = \Unit_\Mca$ and $\Qfr_0 = \Unit_\Mca$,
        $s_1 = 0$, $s_2 = 0$, $s_3 = 0$, and
        $s_4 = \Hsf_{\Pfr \circ_i \Qfr}$.
    \end{enumerate}
    By assembling these cases together, we obtain the stated result.
\end{proof}
\medbreak

\begin{Proposition} \label{prop:composition_Cli_M_basis_K}
    Let $\Mca$ be a unitary magma. The partial composition of $\Cli\Mca$ can be
    expressed in terms of the $\Ksf$-basis, for any $\Mca$-cliques $\Pfr$ and
    $\Qfr$ different from $\UnitClique$ and any valid integer $i$, by
    \begin{small}
    \begin{equation}
        \Ksf_\Pfr \circ_i \Ksf_\Qfr
        =
        \begin{cases}
            \Ksf_{\Pfr \circ_i \Qfr}
                & \mbox{if }
                \Pfr_i \Op \Qfr_0 = \Unit_\Mca, \\
            \Ksf_{\Pfr \circ_i \Qfr} +
            \Ksf_{\Del_i(\Pfr) \circ_i \Del_0(\Qfr)}
                & \mbox{otherwise}.
        \end{cases}
    \end{equation}
    \end{small}
\end{Proposition}
\begin{proof}
    Let $m$ be the arity of $\Qfr$. From the definition of the
    $\Ksf$-basis and of the partial order relation $\OrdD$, we have
    \begin{equation}\begin{split}
        \label{equ:composition_Cli_M_basis_K_demo}
        \Ksf_\Pfr \circ_i \Ksf_\Qfr
        & =
        \sum_{\substack{
            \Pfr', \Qfr' \in \Cliques_\Mca \\
            \Pfr' \OrdD \Pfr \\
            \Qfr' \OrdD \Qfr
        }}
        (-1)^{\Hamming(\Pfr', \Pfr) + \Hamming(\Qfr', \Qfr)}
        \Pfr' \circ_i \Qfr' \\
        & =
        \sum_{\substack{
            \Pfr', \Qfr' \in \Cliques_\Mca \\
            \Pfr' \circ_i \Qfr' \OrdD \Pfr \circ_i \Qfr \\
            \Pfr'_i = \Pfr_i \\
            \Qfr'_0 = \Qfr_0
        }}
        (-1)^{\Hamming(\Pfr', \Pfr) + \Hamming(\Qfr', \Qfr)}
        \Pfr' \circ_i \Qfr' \\
        & =
        \sum_{\substack{
            \Rfr \in \Cliques_\Mca \\
            \Rfr \OrdD \Pfr \circ_i \Qfr \\
            \Rfr(i, i + m - 1) = \Pfr_i \Op \Qfr_0
        }}
        (-1)^{\Hamming(\Rfr, \Pfr \circ_i \Qfr)}
        \Rfr.
    \end{split}\end{equation}
    When $\Pfr_i \Op \Qfr_0 = \Unit_\Mca$,
    \eqref{equ:composition_Cli_M_basis_K_demo} is equal to
    $\Ksf_{\Pfr \circ_i \Qfr}$. Otherwise, when
    $\Pfr_i \Op \Qfr_0 \ne \Unit_\Mca$, we have
    \begin{equation}\begin{split}
        \sum_{\substack{
            \Rfr \in \Cliques_\Mca \\
            \Rfr \OrdD \Pfr \circ_i \Qfr \\
            \Rfr(i, i + m - 1) = \Pfr_i \Op \Qfr_0
        }}
        (-1)^{\Hamming(\Rfr, \Pfr \circ_i \Qfr)}
        \Rfr
        & =
         \sum_{\substack{
            \Rfr \in \Cliques_\Mca \\
            \Rfr \OrdD \Pfr \circ_i \Qfr
        }}
        (-1)^{\Hamming(\Rfr, \Pfr \circ_i \Qfr)}
        \Rfr
        \enspace -
        \sum_{\substack{
            \Rfr \in \Cliques_\Mca \\
            \Rfr \OrdD \Pfr \circ_i \Qfr \\
            \Rfr(i, i + m - 1) \ne \Pfr_i \Op \Qfr_0
        }}
        (-1)^{\Hamming(\Rfr, \Pfr \circ_i \Qfr)}
        \Rfr \\
        & =
        \Ksf_{\Pfr \circ_i \Qfr}
        \enspace -
        \sum_{\substack{
            \Rfr \in \Cliques_\Mca \\
            \Rfr \OrdD \Del_i(\Pfr) \circ_i \Del_0(\Qfr) \\
        }}
        (-1)^{\Hamming(\Rfr, \Pfr \circ_i \Qfr)}
        \Rfr \\
        & =
        \Ksf_{\Pfr \circ_i \Qfr}
        \enspace -
        \sum_{\substack{
            \Rfr \in \Cliques_\Mca \\
            \Rfr \OrdD \Del_i(\Pfr) \circ_i \Del_0(\Qfr) \\
        }}
        (-1)^{1 + \Hamming(\Rfr, \Del_i(\Pfr) \circ_i \Del_0(\Qfr))}
        \Rfr \\
        & =
        \Ksf_{\Pfr \circ_i \Qfr}
        + \Ksf_{\Del_i(\Pfr) \circ_i \Del_0(\Qfr)}.
    \end{split}\end{equation}
    This proves the claimed formula for the partial composition of
    $\Cli\Mca$ over the $\Ksf$-basis.
\end{proof}
\medbreak

For instance, in $\Cli\Z$,
\begin{subequations}
\begin{equation}
    \Hsf_{
    \begin{tikzpicture}[scale=0.4,Centering]
        \node[CliquePoint](1)at(-0.87,-0.50){};
        \node[CliquePoint](2)at(-0.00,1.00){};
        \node[CliquePoint](3)at(0.87,-0.50){};
        \draw[CliqueEmptyEdge](1)edge[]node[]{}(2);
        \draw[CliqueEmptyEdge](1)edge[]node[]{}(3);
        \draw[CliqueEdge](2)edge[]node[CliqueLabel]
            {\begin{math}1\end{math}}(3);
    \end{tikzpicture}}
    \circ_2
    \Hsf_{
    \begin{tikzpicture}[scale=0.4,Centering]
        \node[CliquePoint](1)at(-0.87,-0.50){};
        \node[CliquePoint](2)at(-0.00,1.00){};
        \node[CliquePoint](3)at(0.87,-0.50){};
        \draw[CliqueEmptyEdge](1)edge[]node[]{}(2);
        \draw[CliqueEdge](1)edge[]node[CliqueLabel]
            {\begin{math}1\end{math}}(3);
        \draw[CliqueEmptyEdge](2)edge[]node[]{}(3);
    \end{tikzpicture}}
    =
    \Hsf_{
    \begin{tikzpicture}[scale=0.5,Centering]
        \node[CliquePoint](1)at(-0.71,-0.71){};
        \node[CliquePoint](2)at(-0.71,0.71){};
        \node[CliquePoint](3)at(0.71,0.71){};
        \node[CliquePoint](4)at(0.71,-0.71){};
        \draw[CliqueEmptyEdge](1)edge[]node[]{}(2);
        \draw[CliqueEmptyEdge](1)edge[]node[]{}(4);
        \draw[CliqueEmptyEdge](2)edge[]node[]{}(3);
        \draw[CliqueEmptyEdge](3)edge[]node[]{}(4);
    \end{tikzpicture}}
    +
    2\;
    \Hsf_{
    \begin{tikzpicture}[scale=0.5,Centering]
        \node[CliquePoint](1)at(-0.71,-0.71){};
        \node[CliquePoint](2)at(-0.71,0.71){};
        \node[CliquePoint](3)at(0.71,0.71){};
        \node[CliquePoint](4)at(0.71,-0.71){};
        \draw[CliqueEmptyEdge](1)edge[]node[]{}(2);
        \draw[CliqueEmptyEdge](1)edge[]node[]{}(4);
        \draw[CliqueEmptyEdge](2)edge[]node[]{}(3);
        \draw[CliqueEdge](2)edge[]node[CliqueLabel]
            {\begin{math}1\end{math}}(4);
        \draw[CliqueEmptyEdge](3)edge[]node[]{}(4);
    \end{tikzpicture}}
    +
    \Hsf_{
    \begin{tikzpicture}[scale=0.5,Centering]
        \node[CliquePoint](1)at(-0.71,-0.71){};
        \node[CliquePoint](2)at(-0.71,0.71){};
        \node[CliquePoint](3)at(0.71,0.71){};
        \node[CliquePoint](4)at(0.71,-0.71){};
        \draw[CliqueEmptyEdge](1)edge[]node[]{}(2);
        \draw[CliqueEmptyEdge](1)edge[]node[]{}(4);
        \draw[CliqueEmptyEdge](2)edge[]node[]{}(3);
        \draw[CliqueEdge](2)edge[]node[CliqueLabel]
            {\begin{math}2\end{math}}(4);
        \draw[CliqueEmptyEdge](3)edge[]node[]{}(4);
    \end{tikzpicture}}\,,
\end{equation}
\begin{equation}
    \Ksf_{
    \begin{tikzpicture}[scale=0.4,Centering]
        \node[CliquePoint](1)at(-0.87,-0.50){};
        \node[CliquePoint](2)at(-0.00,1.00){};
        \node[CliquePoint](3)at(0.87,-0.50){};
        \draw[CliqueEmptyEdge](1)edge[]node[]{}(2);
        \draw[CliqueEmptyEdge](1)edge[]node[]{}(3);
        \draw[CliqueEdge](2)edge[]node[CliqueLabel]
            {\begin{math}1\end{math}}(3);
    \end{tikzpicture}}
    \circ_2
    \Ksf_{
    \begin{tikzpicture}[scale=0.4,Centering]
        \node[CliquePoint](1)at(-0.87,-0.50){};
        \node[CliquePoint](2)at(-0.00,1.00){};
        \node[CliquePoint](3)at(0.87,-0.50){};
        \draw[CliqueEmptyEdge](1)edge[]node[]{}(2);
        \draw[CliqueEdge](1)edge[]node[CliqueLabel]
            {\begin{math}1\end{math}}(3);
        \draw[CliqueEmptyEdge](2)edge[]node[]{}(3);
    \end{tikzpicture}}
    =
    \Ksf_{
    \begin{tikzpicture}[scale=0.5,Centering]
        \node[CliquePoint](1)at(-0.71,-0.71){};
        \node[CliquePoint](2)at(-0.71,0.71){};
        \node[CliquePoint](3)at(0.71,0.71){};
        \node[CliquePoint](4)at(0.71,-0.71){};
        \draw[CliqueEmptyEdge](1)edge[]node[]{}(2);
        \draw[CliqueEmptyEdge](1)edge[]node[]{}(4);
        \draw[CliqueEmptyEdge](2)edge[]node[]{}(3);
        \draw[CliqueEmptyEdge](3)edge[]node[]{}(4);
    \end{tikzpicture}}
    +
    \Ksf_{
    \begin{tikzpicture}[scale=0.5,Centering]
        \node[CliquePoint](1)at(-0.71,-0.71){};
        \node[CliquePoint](2)at(-0.71,0.71){};
        \node[CliquePoint](3)at(0.71,0.71){};
        \node[CliquePoint](4)at(0.71,-0.71){};
        \draw[CliqueEmptyEdge](1)edge[]node[]{}(2);
        \draw[CliqueEmptyEdge](1)edge[]node[]{}(4);
        \draw[CliqueEmptyEdge](2)edge[]node[]{}(3);
        \draw[CliqueEdge](2)edge[]node[CliqueLabel]
            {\begin{math}2\end{math}}(4);
        \draw[CliqueEmptyEdge](3)edge[]node[]{}(4);
    \end{tikzpicture}}\,,
\end{equation}
\begin{equation}
    \Hsf_{
    \begin{tikzpicture}[scale=0.5,Centering]
        \node[CliquePoint](1)at(-0.71,-0.71){};
        \node[CliquePoint](2)at(-0.71,0.71){};
        \node[CliquePoint](3)at(0.71,0.71){};
        \node[CliquePoint](4)at(0.71,-0.71){};
        \draw[CliqueEmptyEdge](1)edge[]node[]{}(2);
        \draw[CliqueEdge](1)edge[]node[CliqueLabel]
            {\begin{math}2\end{math}}(3);
        \draw[CliqueEmptyEdge](1)edge[]node[]{}(4);
        \draw[CliqueEmptyEdge](2)edge[]node[]{}(3);
        \draw[CliqueEdge](3)edge[]node[CliqueLabel]
            {\begin{math}1\end{math}}(4);
    \end{tikzpicture}}
    \circ_3
    \Hsf_{
    \begin{tikzpicture}[scale=0.4,Centering]
        \node[CliquePoint](1)at(-0.87,-0.50){};
        \node[CliquePoint](2)at(-0.00,1.00){};
        \node[CliquePoint](3)at(0.87,-0.50){};
        \draw[CliqueEdge](1)edge[]node[CliqueLabel]
            {\begin{math}1\end{math}}(2);
        \draw[CliqueEdge](1)edge[]node[CliqueLabel]
            {\begin{math}2\end{math}}(3);
        \draw[CliqueEdge](2)edge[]node[CliqueLabel]
            {\begin{math}2\end{math}}(3);
    \end{tikzpicture}}
    =
    \Hsf_{
    \begin{tikzpicture}[scale=0.6,Centering]
        \node[CliquePoint](1)at(-0.59,-0.81){};
        \node[CliquePoint](2)at(-0.95,0.31){};
        \node[CliquePoint](3)at(-0.00,1.00){};
        \node[CliquePoint](4)at(0.95,0.31){};
        \node[CliquePoint](5)at(0.59,-0.81){};
        \draw[CliqueEmptyEdge](1)edge[]node[]{}(2);
        \draw[CliqueEdge](1)edge[bend right=30]node[CliqueLabel]
            {\begin{math}2\end{math}}(3);
        \draw[CliqueEmptyEdge](1)edge[]node[]{}(5);
        \draw[CliqueEmptyEdge](2)edge[]node[]{}(3);
        \draw[CliqueEdge](3)edge[]node[CliqueLabel]
            {\begin{math}1\end{math}}(4);
        \draw[CliqueEdge](4)edge[]node[CliqueLabel]
            {\begin{math}2\end{math}}(5);
    \end{tikzpicture}}
    +
    \Hsf_{
    \begin{tikzpicture}[scale=0.6,Centering]
        \node[CliquePoint](1)at(-0.59,-0.81){};
        \node[CliquePoint](2)at(-0.95,0.31){};
        \node[CliquePoint](3)at(-0.00,1.00){};
        \node[CliquePoint](4)at(0.95,0.31){};
        \node[CliquePoint](5)at(0.59,-0.81){};
        \draw[CliqueEmptyEdge](1)edge[]node[]{}(2);
        \draw[CliqueEdge](1)edge[]node[CliqueLabel]
            {\begin{math}2\end{math}}(3);
        \draw[CliqueEmptyEdge](1)edge[]node[]{}(5);
        \draw[CliqueEmptyEdge](2)edge[]node[]{}(3);
        \draw[CliqueEdge](3)edge[]node[CliqueLabel]
            {\begin{math}1\end{math}}(4);
        \draw[CliqueEdge](3)edge[]node[CliqueLabel]
            {\begin{math}1\end{math}}(5);
        \draw[CliqueEdge](4)edge[]node[CliqueLabel]
            {\begin{math}2\end{math}}(5);
    \end{tikzpicture}}
    +
    \Hsf_{
    \begin{tikzpicture}[scale=0.6,Centering]
        \node[CliquePoint](1)at(-0.59,-0.81){};
        \node[CliquePoint](2)at(-0.95,0.31){};
        \node[CliquePoint](3)at(-0.00,1.00){};
        \node[CliquePoint](4)at(0.95,0.31){};
        \node[CliquePoint](5)at(0.59,-0.81){};
        \draw[CliqueEmptyEdge](1)edge[]node[]{}(2);
        \draw[CliqueEdge](1)edge[]node[CliqueLabel]
            {\begin{math}2\end{math}}(3);
        \draw[CliqueEmptyEdge](1)edge[]node[]{}(5);
        \draw[CliqueEmptyEdge](2)edge[]node[]{}(3);
        \draw[CliqueEdge](3)edge[]node[CliqueLabel]
            {\begin{math}1\end{math}}(4);
        \draw[CliqueEdge](3)edge[]node[CliqueLabel]
            {\begin{math}2\end{math}}(5);
        \draw[CliqueEdge](4)edge[]node[CliqueLabel]
            {\begin{math}2\end{math}}(5);
    \end{tikzpicture}}
    +
    \Hsf_{
    \begin{tikzpicture}[scale=0.6,Centering]
        \node[CliquePoint](1)at(-0.59,-0.81){};
        \node[CliquePoint](2)at(-0.95,0.31){};
        \node[CliquePoint](3)at(-0.00,1.00){};
        \node[CliquePoint](4)at(0.95,0.31){};
        \node[CliquePoint](5)at(0.59,-0.81){};
        \draw[CliqueEmptyEdge](1)edge[]node[]{}(2);
        \draw[CliqueEdge](1)edge[]node[CliqueLabel]
            {\begin{math}2\end{math}}(3);
        \draw[CliqueEmptyEdge](1)edge[]node[]{}(5);
        \draw[CliqueEmptyEdge](2)edge[]node[]{}(3);
        \draw[CliqueEdge](3)edge[]node[CliqueLabel]
            {\begin{math}1\end{math}}(4);
        \draw[CliqueEdge](3)edge[]node[CliqueLabel]
            {\begin{math}3\end{math}}(5);
        \draw[CliqueEdge](4)edge[]node[CliqueLabel]
            {\begin{math}2\end{math}}(5);
    \end{tikzpicture}}\,,
\end{equation}
\begin{equation}
    \Ksf_{
    \begin{tikzpicture}[scale=0.5,Centering]
        \node[CliquePoint](1)at(-0.71,-0.71){};
        \node[CliquePoint](2)at(-0.71,0.71){};
        \node[CliquePoint](3)at(0.71,0.71){};
        \node[CliquePoint](4)at(0.71,-0.71){};
        \draw[CliqueEmptyEdge](1)edge[]node[]{}(2);
        \draw[CliqueEdge](1)edge[]node[CliqueLabel]
            {\begin{math}2\end{math}}(3);
        \draw[CliqueEmptyEdge](1)edge[]node[]{}(4);
        \draw[CliqueEmptyEdge](2)edge[]node[]{}(3);
        \draw[CliqueEdge](3)edge[]node[CliqueLabel]
            {\begin{math}1\end{math}}(4);
    \end{tikzpicture}}
    \circ_3
    \Ksf_{
    \begin{tikzpicture}[scale=0.4,Centering]
        \node[CliquePoint](1)at(-0.87,-0.50){};
        \node[CliquePoint](2)at(-0.00,1.00){};
        \node[CliquePoint](3)at(0.87,-0.50){};
        \draw[CliqueEdge](1)edge[]node[CliqueLabel]
            {\begin{math}1\end{math}}(2);
        \draw[CliqueEdge](1)edge[]node[CliqueLabel]
            {\begin{math}2\end{math}}(3);
        \draw[CliqueEdge](2)edge[]node[CliqueLabel]
            {\begin{math}2\end{math}}(3);
    \end{tikzpicture}}
    =
    \Ksf_{
    \begin{tikzpicture}[scale=0.6,Centering]
        \node[CliquePoint](1)at(-0.59,-0.81){};
        \node[CliquePoint](2)at(-0.95,0.31){};
        \node[CliquePoint](3)at(-0.00,1.00){};
        \node[CliquePoint](4)at(0.95,0.31){};
        \node[CliquePoint](5)at(0.59,-0.81){};
        \draw[CliqueEmptyEdge](1)edge[]node[]{}(2);
        \draw[CliqueEdge](1)edge[bend right=30]node[CliqueLabel]
            {\begin{math}2\end{math}}(3);
        \draw[CliqueEmptyEdge](1)edge[]node[]{}(5);
        \draw[CliqueEmptyEdge](2)edge[]node[]{}(3);
        \draw[CliqueEdge](3)edge[]node[CliqueLabel]
            {\begin{math}1\end{math}}(4);
        \draw[CliqueEdge](4)edge[]node[CliqueLabel]
            {\begin{math}2\end{math}}(5);
    \end{tikzpicture}}
    +
    \Ksf_{
    \begin{tikzpicture}[scale=0.6,Centering]
        \node[CliquePoint](1)at(-0.59,-0.81){};
        \node[CliquePoint](2)at(-0.95,0.31){};
        \node[CliquePoint](3)at(-0.00,1.00){};
        \node[CliquePoint](4)at(0.95,0.31){};
        \node[CliquePoint](5)at(0.59,-0.81){};
        \draw[CliqueEmptyEdge](1)edge[]node[]{}(2);
        \draw[CliqueEdge](1)edge[]node[CliqueLabel]
            {\begin{math}2\end{math}}(3);
        \draw[CliqueEmptyEdge](1)edge[]node[]{}(5);
        \draw[CliqueEmptyEdge](2)edge[]node[]{}(3);
        \draw[CliqueEdge](3)edge[]node[CliqueLabel]
            {\begin{math}1\end{math}}(4);
        \draw[CliqueEdge](3)edge[]node[CliqueLabel]
            {\begin{math}3\end{math}}(5);
        \draw[CliqueEdge](4)edge[]node[CliqueLabel]
            {\begin{math}2\end{math}}(5);
    \end{tikzpicture}}\,,
\end{equation}
\begin{equation}
    \Hsf_{
    \begin{tikzpicture}[scale=0.6,Centering]
        \node[CliquePoint](1)at(-0.71,-0.71){};
        \node[CliquePoint](2)at(-0.71,0.71){};
        \node[CliquePoint](3)at(0.71,0.71){};
        \node[CliquePoint](4)at(0.71,-0.71){};
        \draw[CliqueEmptyEdge](1)edge[]node[]{}(2);
        \draw[CliqueEmptyEdge](1)edge[]node[]{}(4);
        \draw[CliqueEdge](2)edge[]node[CliqueLabel]
            {\begin{math}-1\end{math}}(3);
        \draw[CliqueEdge](2)edge[]node[CliqueLabel]
            {\begin{math}2\end{math}}(4);
        \draw[CliqueEdge](3)edge[]node[CliqueLabel]
            {\begin{math}1\end{math}}(4);
    \end{tikzpicture}}
    \circ_2
    \Hsf_{
    \begin{tikzpicture}[scale=0.6,Centering]
        \node[CliquePoint](1)at(-0.71,-0.71){};
        \node[CliquePoint](2)at(-0.71,0.71){};
        \node[CliquePoint](3)at(0.71,0.71){};
        \node[CliquePoint](4)at(0.71,-0.71){};
        \draw[CliqueEmptyEdge](1)edge[]node[]{}(2);
        \draw[CliqueEdge](1)edge[]node[CliqueLabel]
            {\begin{math}-1\end{math}}(3);
        \draw[CliqueEdge](1)edge[]node[CliqueLabel]
            {\begin{math}1\end{math}}(4);
        \draw[CliqueEdge](2)edge[]node[CliqueLabel]
            {\begin{math}1\end{math}}(3);
        \draw[CliqueEmptyEdge](3)edge[]node[]{}(4);
    \end{tikzpicture}}
    =
    \Hsf_{
    \begin{tikzpicture}[scale=0.8,Centering]
        \node[CliquePoint](1)at(-0.50,-0.87){};
        \node[CliquePoint](2)at(-1.00,-0.00){};
        \node[CliquePoint](3)at(-0.50,0.87){};
        \node[CliquePoint](4)at(0.50,0.87){};
        \node[CliquePoint](5)at(1.00,0.00){};
        \node[CliquePoint](6)at(0.50,-0.87){};
        \draw[CliqueEmptyEdge](1)edge[]node[]{}(2);
        \draw[CliqueEmptyEdge](1)edge[]node[]{}(6);
        \draw[CliqueEmptyEdge](2)edge[]node[]{}(3);
        \draw[CliqueEdge](2)edge[]node[CliqueLabel]
            {\begin{math}-1\end{math}}(4);
        \draw[CliqueEdge](2)edge[]node[CliqueLabel]
            {\begin{math}-1\end{math}}(5);
        \draw[CliqueEdge](2)edge[]node[CliqueLabel]
            {\begin{math}2\end{math}}(6);
        \draw[CliqueEdge](3)edge[]node[CliqueLabel]
            {\begin{math}1\end{math}}(4);
        \draw[CliqueEmptyEdge](4)edge[]node[]{}(5);
        \draw[CliqueEdge](5)edge[]node[CliqueLabel]
            {\begin{math}1\end{math}}(6);
    \end{tikzpicture}}
     + 2 \;
    \Hsf_{
    \begin{tikzpicture}[scale=0.8,Centering]
        \node[CliquePoint](1)at(-0.50,-0.87){};
        \node[CliquePoint](2)at(-1.00,-0.00){};
        \node[CliquePoint](3)at(-0.50,0.87){};
        \node[CliquePoint](4)at(0.50,0.87){};
        \node[CliquePoint](5)at(1.00,0.00){};
        \node[CliquePoint](6)at(0.50,-0.87){};
        \draw[CliqueEmptyEdge](1)edge[]node[]{}(2);
        \draw[CliqueEmptyEdge](1)edge[]node[]{}(6);
        \draw[CliqueEmptyEdge](2)edge[]node[]{}(3);
        \draw[CliqueEdge](2)edge[bend right=30]node[CliqueLabel]
            {\begin{math}-1\end{math}}(4);
        \draw[CliqueEdge](2)edge[bend left=30]node[CliqueLabel]
            {\begin{math}2\end{math}}(6);
        \draw[CliqueEdge](3)edge[]node[CliqueLabel]
            {\begin{math}1\end{math}}(4);
        \draw[CliqueEmptyEdge](4)edge[]node[]{}(5);
        \draw[CliqueEdge](5)edge[]node[CliqueLabel]
            {\begin{math}1\end{math}}(6);
    \end{tikzpicture}}
     +
    \Hsf_{
    \begin{tikzpicture}[scale=0.8,Centering]
        \node[CliquePoint](1)at(-0.50,-0.87){};
        \node[CliquePoint](2)at(-1.00,-0.00){};
        \node[CliquePoint](3)at(-0.50,0.87){};
        \node[CliquePoint](4)at(0.50,0.87){};
        \node[CliquePoint](5)at(1.00,0.00){};
        \node[CliquePoint](6)at(0.50,-0.87){};
        \draw[CliqueEmptyEdge](1)edge[]node[]{}(2);
        \draw[CliqueEmptyEdge](1)edge[]node[]{}(6);
        \draw[CliqueEmptyEdge](2)edge[]node[]{}(3);
        \draw[CliqueEdge](2)edge[]node[CliqueLabel]
            {\begin{math}-1\end{math}}(4);
        \draw[CliqueEdge](2)edge[]node[CliqueLabel]
            {\begin{math}1\end{math}}(5);
        \draw[CliqueEdge](2)edge[]node[CliqueLabel]
            {\begin{math}2\end{math}}(6);
        \draw[CliqueEdge](3)edge[]node[CliqueLabel]
            {\begin{math}1\end{math}}(4);
        \draw[CliqueEmptyEdge](4)edge[]node[]{}(5);
        \draw[CliqueEdge](5)edge[]node[CliqueLabel]
            {\begin{math}1\end{math}}(6);
    \end{tikzpicture}}\,,
\end{equation}
\begin{equation}
    \Ksf_{
    \begin{tikzpicture}[scale=0.6,Centering]
        \node[CliquePoint](1)at(-0.71,-0.71){};
        \node[CliquePoint](2)at(-0.71,0.71){};
        \node[CliquePoint](3)at(0.71,0.71){};
        \node[CliquePoint](4)at(0.71,-0.71){};
        \draw[CliqueEmptyEdge](1)edge[]node[]{}(2);
        \draw[CliqueEmptyEdge](1)edge[]node[]{}(4);
        \draw[CliqueEdge](2)edge[]node[CliqueLabel]
            {\begin{math}-1\end{math}}(3);
        \draw[CliqueEdge](2)edge[]node[CliqueLabel]
            {\begin{math}2\end{math}}(4);
        \draw[CliqueEdge](3)edge[]node[CliqueLabel]
            {\begin{math}1\end{math}}(4);
    \end{tikzpicture}}
    \circ_2
    \Ksf_{
    \begin{tikzpicture}[scale=0.6,Centering]
        \node[CliquePoint](1)at(-0.71,-0.71){};
        \node[CliquePoint](2)at(-0.71,0.71){};
        \node[CliquePoint](3)at(0.71,0.71){};
        \node[CliquePoint](4)at(0.71,-0.71){};
        \draw[CliqueEmptyEdge](1)edge[]node[]{}(2);
        \draw[CliqueEdge](1)edge[]node[CliqueLabel]
            {\begin{math}-1\end{math}}(3);
        \draw[CliqueEdge](1)edge[]node[CliqueLabel]
            {\begin{math}1\end{math}}(4);
        \draw[CliqueEdge](2)edge[]node[CliqueLabel]
            {\begin{math}1\end{math}}(3);
        \draw[CliqueEmptyEdge](3)edge[]node[]{}(4);
    \end{tikzpicture}}
    =
    \Ksf_{
    \begin{tikzpicture}[scale=0.8,Centering]
        \node[CliquePoint](1)at(-0.50,-0.87){};
        \node[CliquePoint](2)at(-1.00,-0.00){};
        \node[CliquePoint](3)at(-0.50,0.87){};
        \node[CliquePoint](4)at(0.50,0.87){};
        \node[CliquePoint](5)at(1.00,0.00){};
        \node[CliquePoint](6)at(0.50,-0.87){};
        \draw[CliqueEmptyEdge](1)edge[]node[]{}(2);
        \draw[CliqueEmptyEdge](1)edge[]node[]{}(6);
        \draw[CliqueEmptyEdge](2)edge[]node[]{}(3);
        \draw[CliqueEdge](2)edge[bend right=30]node[CliqueLabel]
            {\begin{math}-1\end{math}}(4);
        \draw[CliqueEdge](2)edge[bend left=30]node[CliqueLabel]
            {\begin{math}2\end{math}}(6);
        \draw[CliqueEdge](3)edge[]node[CliqueLabel]
            {\begin{math}1\end{math}}(4);
        \draw[CliqueEmptyEdge](4)edge[]node[]{}(5);
        \draw[CliqueEdge](5)edge[]node[CliqueLabel]
            {\begin{math}1\end{math}}(6);
    \end{tikzpicture}}\,,
\end{equation}
\end{subequations}
and in $\Dbb_1$,
\begin{subequations}
\begin{equation}
    \Hsf_{
    \begin{tikzpicture}[scale=0.6,Centering]
        \node[CliquePoint](1)at(-0.71,-0.71){};
        \node[CliquePoint](2)at(-0.71,0.71){};
        \node[CliquePoint](3)at(0.71,0.71){};
        \node[CliquePoint](4)at(0.71,-0.71){};
        \draw[CliqueEmptyEdge](1)edge[]node[]{}(2);
        \draw[CliqueEmptyEdge](1)edge[]node[]{}(4);
        \draw[CliqueEdge](2)edge[]node[CliqueLabel]
            {\begin{math}0\end{math}}(3);
        \draw[CliqueEdge](2)edge[]node[CliqueLabel]
            {\begin{math}\Dtt_1\end{math}}(4);
        \draw[CliqueEdge](3)edge[]node[CliqueLabel]
            {\begin{math}0\end{math}}(4);
    \end{tikzpicture}}
    \circ_2
    \Hsf_{
    \begin{tikzpicture}[scale=0.6,Centering]
        \node[CliquePoint](1)at(-0.71,-0.71){};
        \node[CliquePoint](2)at(-0.71,0.71){};
        \node[CliquePoint](3)at(0.71,0.71){};
        \node[CliquePoint](4)at(0.71,-0.71){};
        \draw[CliqueEmptyEdge](1)edge[]node[]{}(2);
        \draw[CliqueEdge](1)edge[]node[CliqueLabel]
            {\begin{math}0\end{math}}(3);
        \draw[CliqueEdge](1)edge[]node[CliqueLabel]
            {\begin{math}0\end{math}}(4);
        \draw[CliqueEdge](2)edge[]node[CliqueLabel]
            {\begin{math}0\end{math}}(3);
        \draw[CliqueEmptyEdge](3)edge[]node[]{}(4);
    \end{tikzpicture}}
    =
    3 \;
    \Hsf_{
    \begin{tikzpicture}[scale=0.8,Centering]
        \node[CliquePoint](1)at(-0.50,-0.87){};
        \node[CliquePoint](2)at(-1.00,-0.00){};
        \node[CliquePoint](3)at(-0.50,0.87){};
        \node[CliquePoint](4)at(0.50,0.87){};
        \node[CliquePoint](5)at(1.00,0.00){};
        \node[CliquePoint](6)at(0.50,-0.87){};
        \draw[CliqueEmptyEdge](1)edge[]node[]{}(2);
        \draw[CliqueEmptyEdge](1)edge[]node[]{}(6);
        \draw[CliqueEmptyEdge](2)edge[]node[]{}(3);
        \draw[CliqueEdge](2)edge[]node[CliqueLabel]
            {\begin{math}0\end{math}}(4);
        \draw[CliqueEdge](2)edge[]node[CliqueLabel]
            {\begin{math}0\end{math}}(5);
        \draw[CliqueEdge](2)edge[]node[CliqueLabel]
            {\begin{math}\Dtt_1\end{math}}(6);
        \draw[CliqueEdge](3)edge[]node[CliqueLabel]
            {\begin{math}0\end{math}}(4);
        \draw[CliqueEmptyEdge](4)edge[]node[]{}(5);
        \draw[CliqueEdge](5)edge[]node[CliqueLabel]
            {\begin{math}0\end{math}}(6);
    \end{tikzpicture}}
    +
    \Hsf_{
    \begin{tikzpicture}[scale=0.8,Centering]
        \node[CliquePoint](1)at(-0.50,-0.87){};
        \node[CliquePoint](2)at(-1.00,-0.00){};
        \node[CliquePoint](3)at(-0.50,0.87){};
        \node[CliquePoint](4)at(0.50,0.87){};
        \node[CliquePoint](5)at(1.00,0.00){};
        \node[CliquePoint](6)at(0.50,-0.87){};
        \draw[CliqueEmptyEdge](1)edge[]node[]{}(2);
        \draw[CliqueEmptyEdge](1)edge[]node[]{}(6);
        \draw[CliqueEmptyEdge](2)edge[]node[]{}(3);
        \draw[CliqueEdge](2)edge[bend right=30]node[CliqueLabel]
            {\begin{math}0\end{math}}(4);
        \draw[CliqueEdge](2)edge[bend left=30]node[CliqueLabel]
            {\begin{math}\Dtt_1\end{math}}(6);
        \draw[CliqueEdge](3)edge[]node[CliqueLabel]
            {\begin{math}0\end{math}}(4);
        \draw[CliqueEmptyEdge](4)edge[]node[]{}(5);
        \draw[CliqueEdge](5)edge[]node[CliqueLabel]
            {\begin{math}0\end{math}}(6);
    \end{tikzpicture}}\,,
\end{equation}
\begin{equation}
    \Ksf_{
    \begin{tikzpicture}[scale=0.6,Centering]
        \node[CliquePoint](1)at(-0.71,-0.71){};
        \node[CliquePoint](2)at(-0.71,0.71){};
        \node[CliquePoint](3)at(0.71,0.71){};
        \node[CliquePoint](4)at(0.71,-0.71){};
        \draw[CliqueEmptyEdge](1)edge[]node[]{}(2);
        \draw[CliqueEmptyEdge](1)edge[]node[]{}(4);
        \draw[CliqueEdge](2)edge[]node[CliqueLabel]
            {\begin{math}0\end{math}}(3);
        \draw[CliqueEdge](2)edge[]node[CliqueLabel]
            {\begin{math}\Dtt_1\end{math}}(4);
        \draw[CliqueEdge](3)edge[]node[CliqueLabel]
            {\begin{math}0\end{math}}(4);
    \end{tikzpicture}}
    \circ_2
    \Ksf_{
    \begin{tikzpicture}[scale=0.6,Centering]
        \node[CliquePoint](1)at(-0.71,-0.71){};
        \node[CliquePoint](2)at(-0.71,0.71){};
        \node[CliquePoint](3)at(0.71,0.71){};
        \node[CliquePoint](4)at(0.71,-0.71){};
        \draw[CliqueEmptyEdge](1)edge[]node[]{}(2);
        \draw[CliqueEdge](1)edge[]node[CliqueLabel]
            {\begin{math}0\end{math}}(3);
        \draw[CliqueEdge](1)edge[]node[CliqueLabel]
            {\begin{math}0\end{math}}(4);
        \draw[CliqueEdge](2)edge[]node[CliqueLabel]
            {\begin{math}0\end{math}}(3);
        \draw[CliqueEmptyEdge](3)edge[]node[]{}(4);
    \end{tikzpicture}}
    =
    \Ksf_{
    \begin{tikzpicture}[scale=0.8,Centering]
        \node[CliquePoint](1)at(-0.50,-0.87){};
        \node[CliquePoint](2)at(-1.00,-0.00){};
        \node[CliquePoint](3)at(-0.50,0.87){};
        \node[CliquePoint](4)at(0.50,0.87){};
        \node[CliquePoint](5)at(1.00,0.00){};
        \node[CliquePoint](6)at(0.50,-0.87){};
        \draw[CliqueEmptyEdge](1)edge[]node[]{}(2);
        \draw[CliqueEmptyEdge](1)edge[]node[]{}(6);
        \draw[CliqueEmptyEdge](2)edge[]node[]{}(3);
        \draw[CliqueEdge](2)edge[]node[CliqueLabel]
            {\begin{math}0\end{math}}(4);
        \draw[CliqueEdge](2)edge[]node[CliqueLabel]
            {\begin{math}0\end{math}}(5);
        \draw[CliqueEdge](2)edge[]node[CliqueLabel]
            {\begin{math}\Dtt_1\end{math}}(6);
        \draw[CliqueEdge](3)edge[]node[CliqueLabel]
            {\begin{math}0\end{math}}(4);
        \draw[CliqueEmptyEdge](4)edge[]node[]{}(5);
        \draw[CliqueEdge](5)edge[]node[CliqueLabel]
            {\begin{math}0\end{math}}(6);
    \end{tikzpicture}}
    +
    \Ksf_{
    \begin{tikzpicture}[scale=0.8,Centering]
        \node[CliquePoint](1)at(-0.50,-0.87){};
        \node[CliquePoint](2)at(-1.00,-0.00){};
        \node[CliquePoint](3)at(-0.50,0.87){};
        \node[CliquePoint](4)at(0.50,0.87){};
        \node[CliquePoint](5)at(1.00,0.00){};
        \node[CliquePoint](6)at(0.50,-0.87){};
        \draw[CliqueEmptyEdge](1)edge[]node[]{}(2);
        \draw[CliqueEmptyEdge](1)edge[]node[]{}(6);
        \draw[CliqueEmptyEdge](2)edge[]node[]{}(3);
        \draw[CliqueEdge](2)edge[bend right=30]node[CliqueLabel]
            {\begin{math}0\end{math}}(4);
        \draw[CliqueEdge](2)edge[bend left=30]node[CliqueLabel]
            {\begin{math}\Dtt_1\end{math}}(6);
        \draw[CliqueEdge](3)edge[]node[CliqueLabel]
            {\begin{math}0\end{math}}(4);
        \draw[CliqueEmptyEdge](4)edge[]node[]{}(5);
        \draw[CliqueEdge](5)edge[]node[CliqueLabel]
            {\begin{math}0\end{math}}(6);
    \end{tikzpicture}}\,.
\end{equation}
\end{subequations}
\medbreak

%%%%%%%%%%%%%%%%%%%%%%%%%%%%%%%%%%%%%%%%%%%%%%%%%%%%%%%%%%%%%%%%%%%%%%%%
\subsubsection{Rational functions} \label{subsubsec:rational_functions}
The graded vector space of all commutative rational functions
$\K(\Ubb)$, where $\Ubb$ is the infinite commutative alphabet
$\{u_1, u_2, \dots\}$, has the structure of an operad $\RatFct$
introduced by Loday~\cite{Lod10} and is defined as follows. Let
$\RatFct(n)$ be the subspace $\K(u_1, \dots, u_n)$ of $\K(\Ubb)$ and
\begin{equation}
    \RatFct := \bigoplus_{n \geq 1} \RatFct(n).
\end{equation}
Observe that since $\RatFct$ is a graded space, each rational function
has an arity. Hence, by setting $f_1(u_1) := 1$ and $f_2(u_1, u_2) := 1$,
$f_1$ is of arity $1$ while $f_2$ is of arity $2$, so that $f_1$ and
$f_2$ are considered as different rational functions. The partial
composition of two rational functions $f \in \RatFct(n)$ and
$g \in \RatFct(m)$ is defined by
\begin{equation} \label{equ:partial_composition_RatFct}
    f \circ_i g :=
    f\left(u_1, \dots, u_{i - 1}, u_i + \dots + u_{i + m - 1},
        u_{i + m}, \dots, u_{n + m - 1}\right)
    \;
    g\left(u_i, \dots, u_{i + m - 1}\right).
\end{equation}
The rational function $f$ of $\RatFct(1)$ defined by $f(u_1) := 1$ is
the unit of~$\RatFct$. As shown by Loday, this operad is (nontrivially)
isomorphic to the operad $\Mould$ introduced by Chapoton~\cite{Cha07}.
\medbreak

Let us assume that $\Mca$ is a \Def{$\Z$-graded unitary magma}, that
is a unitary magma such that there exists a unitary magma morphism
$\theta : \Mca \to \Z$. We say that $\theta$ is a \Def{rank function}
of $\Mca$. In this context, let
\begin{equation}
    \Frac_\theta : \Cli\Mca \to \RatFct
\end{equation}
be the linear map defined, for any $\Mca$-clique $\Pfr$, by
\begin{equation} \label{equ:definition_frac_clique}
    \Frac_\theta(\Pfr) :=
    \prod_{(x, y) \in \Arcs_\Pfr}
    \left(u_x + \dots + u_{y - 1}\right)^{\theta(\Pfr(x, y))}.
\end{equation}
For instance, by considering the unitary magma $\Z$ together with its
identity map $\Id$ as rank function, one has
\begin{equation}
    \Frac_\Id\left(
    \begin{tikzpicture}[scale=.85,Centering]
        \node[CliquePoint](1)at(-0.43,-0.90){};
        \node[CliquePoint](2)at(-0.97,-0.22){};
        \node[CliquePoint](3)at(-0.78,0.62){};
        \node[CliquePoint](4)at(-0.00,1.00){};
        \node[CliquePoint](5)at(0.78,0.62){};
        \node[CliquePoint](6)at(0.97,-0.22){};
        \node[CliquePoint](7)at(0.43,-0.90){};
        \draw[CliqueEdge](1)edge[]node[CliqueLabel]
            {\begin{math}-1\end{math}}(2);
        \draw[CliqueEdge](1)edge[bend left=30]node[CliqueLabel,near start]
            {\begin{math}2\end{math}}(5);
        \draw[CliqueEdge](1)edge[]node[CliqueLabel]
            {\begin{math}1\end{math}}(7);
        \draw[CliqueEmptyEdge](2)edge[]node[CliqueLabel]{}(3);
        \draw[CliqueEmptyEdge](3)edge[]node[CliqueLabel]{}(4);
        \draw[CliqueEdge](3)edge[bend left=30]node[CliqueLabel,near start]
            {\begin{math}-2\end{math}}(7);
        \draw[CliqueEdge](4)edge[]node[CliqueLabel]
            {\begin{math}3\end{math}}(5);
        \draw[CliqueEmptyEdge](5)edge[]node[CliqueLabel]{}(6);
        \draw[CliqueEmptyEdge](6)edge[]node[CliqueLabel]{}(7);
        \draw[CliqueEdge](5)edge[]node[CliqueLabel]
            {\begin{math}-1\end{math}}(7);
    \end{tikzpicture}\right)
    =
    \frac{
        \left(u_1 + u_2 + u_3 + u_4\right)^2
        \left(u_1 + u_2 + u_3 + u_4 + u_5 + u_6\right)
        u_4^3
    }{u_1 \left(u_3 + u_4 + u_5 + u_6\right)^2 \left(u_5 + u_6\right)}.
\end{equation}
\medbreak

\begin{Theorem} \label{thm:rat_fct_cliques}
    Let $\Mca$ be a $\Z$-graded unitary magma and $\theta$ be a rank
    function of $\Mca$. The map $\Frac_\theta$ is an operad morphism
    from $\Cli\Mca$ to~$\RatFct$.
\end{Theorem}
\begin{proof}
    For the sake of brevity of notation, for all positive integers $x < y$, we denote
    by $\Ubb_{x, y}$ the sums $u_x + \dots + u_{y - 1}$. Let $\Pfr$
    and $\Qfr$ be two $\Mca$-cliques of respective arities $n$ and $m$,
    and $i \in [n]$ be an integer. From the definition of the partial
    composition of $\Cli\Mca$, the
    one (see~\eqref{equ:partial_composition_RatFct}) of $\RatFct$, and the
    fact that $\theta$ is a unitary magma morphism, we have
    \begin{equation}\begin{split}
        \Frac_\theta(\Pfr) \circ_i \Frac_\theta(\Qfr)
        & =
        \left(\Frac_\theta(\Pfr)\right)
            \left(u_1, \dots,u_{i - 1}, \Ubb_{i, i + m}, u_{i + m},
            \dots, u_{n + m - 1}\right)
        \;
        \left(\Frac_\theta(\Qfr)\right)
            \left(u_i, \dots, u_{i + m - 1}\right) \\
        & =
        \left(
        \prod_{1 \leq x < y \leq i - 1}
        \Ubb_{x, y}^{\theta(\Pfr(x, y))}
        \right)
        \left(
        \prod_{i + 1 \leq x < y \leq n + 1}
        \Ubb_{x + m - 1, y + m - 1}^{\theta(\Pfr(x, y))}
        \right)
        \Ubb_{i, i + m}^{\theta(\Pfr_i)} \\
        & \qquad
        \left(
        \prod_{1 \leq x < y \leq m + 1}
        \Ubb_{x + i - 1, y + i - 1}^{\theta(\Qfr(x, y))}
        \right) \\
        & =
        \left(
        \prod_{1 \leq x < y \leq i - 1}
        \Ubb_{x, y}^{\theta(\Pfr(x, y))}
        \right)
        \left(
        \prod_{i + 1 \leq x < y \leq n + 1}
        \Ubb_{x + m - 1, y + m - 1}^{\theta(\Pfr(x, y))}
        \right)
        \Ubb_{i, i + m}
            ^{\theta(\Pfr_i) + \theta(\Qfr_0)} \\
        & \qquad
        \left(
        \prod_{\substack{
            1 \leq x < y \leq m + 1 \\
            (x, y) \ne (1, m + 1)
        }}
        \Ubb_{x + i - 1, y + i - 1}^{\theta(\Qfr(x, y))}
        \right) \\
        & =
        \left(
        \prod_{1 \leq x < y \leq i - 1}
        \Ubb_{x, y}^{\theta(\Pfr(x, y))}
        \right)
        \left(
        \prod_{i + 1 \leq x < y \leq n + 1}
        \Ubb_{x + m - 1, y + m - 1}^{\theta(\Pfr(x, y))}
        \right)
        \Ubb_{i, i + m}
            ^{\theta(\Pfr_i \Op \Qfr_0)} \\
        & \qquad
        \left(
        \prod_{\substack{
            1 \leq x < y \leq m + 1 \\
            (x, y) \ne (1, m + 1)
        }}
        \Ubb_{x + i - 1, y + i - 1}^{\theta(\Qfr(x, y))}
        \right) \\
        & =
        \prod_{(x, y) \in \Arcs_{\Pfr \circ_i \Qfr}}
        \Ubb_{x, y}^{\theta((\Pfr \circ_i \Qfr)(x, y))} \\
        & =
        \Frac_\theta(\Pfr \circ_i \Qfr).
    \end{split}\end{equation}
    Moreover, since $\theta(\Unit_\Mca) = 0$, one has
    $\Frac_\theta\left(\UnitClique\right) = 1$, so that $\Frac_\theta$
    sends the unit of $\Cli\Mca$ to the unit of $\RatFct$. Therefore,
    $\Frac_\theta$ is an operad morphism.
\end{proof}
\medbreak

The operad morphism $\Frac_\theta$ is not injective. Indeed, by
considering the magma $\Z$ together with its identity map $\Id$ as rank
function, one has for instance
\begin{subequations}
\begin{equation} \label{equ:frac_not_injective_1}
    \Frac_\Id\left(
    \TriangleXEE{1}{}{}
    - \TriangleEXE{}{1}{}
    - \TriangleEEX{}{}{1}
    \right)
    = (u_1 + u_2) - u_1 - u_2
    = 0,
\end{equation}
\begin{equation} \label{equ:frac_not_injective_2}
    \Frac_\Id\left(
    \begin{tikzpicture}[scale=0.6,Centering]
        \node[CliquePoint](1)at(-0.71,-0.71){};
        \node[CliquePoint](2)at(-0.71,0.71){};
        \node[CliquePoint](3)at(0.71,0.71){};
        \node[CliquePoint](4)at(0.71,-0.71){};
        \draw[CliqueEmptyEdge](1)edge[]node[]{}(2);
        \draw[CliqueEmptyEdge](1)edge[]node[]{}(4);
        \draw[CliqueEdge](2)edge[]node[CliqueLabel]
            {\begin{math}-1\end{math}}(3);
        \draw[CliqueEdge](3)edge[]node[CliqueLabel]
            {\begin{math}-1\end{math}}(4);
    \end{tikzpicture}
    -
    \begin{tikzpicture}[scale=0.6,Centering]
        \node[CliquePoint](1)at(-0.71,-0.71){};
        \node[CliquePoint](2)at(-0.71,0.71){};
        \node[CliquePoint](3)at(0.71,0.71){};
        \node[CliquePoint](4)at(0.71,-0.71){};
        \draw[CliqueEmptyEdge](1)edge[]node[]{}(2);
        \draw[CliqueEmptyEdge](1)edge[]node[]{}(4);
        \draw[CliqueEmptyEdge](2)edge[]node[]{}(3);
        \draw[CliqueEdge](2)edge[]node[CliqueLabel]
            {\begin{math}-1\end{math}}(4);
        \draw[CliqueEdge](3)edge[]node[CliqueLabel]
            {\begin{math}-1\end{math}}(4);
    \end{tikzpicture}
    -
    \begin{tikzpicture}[scale=0.6,Centering]
        \node[CliquePoint](1)at(-0.71,-0.71){};
        \node[CliquePoint](2)at(-0.71,0.71){};
        \node[CliquePoint](3)at(0.71,0.71){};
        \node[CliquePoint](4)at(0.71,-0.71){};
        \draw[CliqueEmptyEdge](1)edge[]node[]{}(2);
        \draw[CliqueEmptyEdge](1)edge[]node[]{}(4);
        \draw[CliqueEmptyEdge](3)edge[]node[]{}(4);
        \draw[CliqueEdge](2)edge[]node[CliqueLabel]
            {\begin{math}-1\end{math}}(3);
        \draw[CliqueEdge](2)edge[]node[CliqueLabel]
            {\begin{math}-1\end{math}}(4);
    \end{tikzpicture}
    \right)
    =
    \frac{1}{u_2 u_3}
    -
    \frac{1}{(u_2 + u_3) u_3}
    -
    \frac{1}{u_2 (u_2 + u_3)}
    = 0.
\end{equation}
\end{subequations}
\medbreak

\begin{Proposition} \label{prop:rat_fct_cliques_map_Laurent_polynomials}
    The subspace of $\RatFct$ of all Laurent polynomials on~$\Ubb$ is
    the image by $\Frac_\Id : \Cli\Z \to \RatFct$ of the subspace of
    $\Cli\Z$ consisting in the linear span of all $\Z$-bubbles.
\end{Proposition}
\begin{proof}
    First, by Theorem~\ref{thm:rat_fct_cliques}, $\Frac_\Id$ is a
    well-defined operad morphism from $\Cli\Z$ to $\RatFct$. Let
    $u_1^{\alpha_1} \dots u_n^{\alpha_n}$ be a Laurent monomial, where
    $\alpha_1, \dots, \alpha_n \in \Z$ and $n \geq 1$. Consider also the
    $\Z$-clique $\Pfr_\alpha$ of arity $n + 1$ satisfying
    \begin{equation}
        \Pfr_\alpha(x, y) :=
        \begin{cases}
            \alpha_x & \mbox{if } y = x + 1, \\
            0 & \mbox{otherwise}.
        \end{cases}
    \end{equation}
    Observe that $\Pfr_\alpha$ is a $\Z$-bubble. By definition of
    $\Frac_\Id$, we have
    $\Frac_\Id(\Pfr_\alpha) = u_1^{\alpha_1} \dots u_n^{\alpha_n}$.
    Now, since a Laurent polynomial is a linear combination of some
    Laurent monomials, by the linearity of $\Frac_\Id$, the statement of
    the proposition follows.
\end{proof}
\medbreak

%On the homogeneous subspace $\Cli\Mca(n)$ of elements of arity
%$n \geq 1$ of $\Cli\Mca$, let the product
%\begin{equation}
%    \Op : \Cli\Mca(n) \otimes \Cli\Mca(n) \to \Cli\Mca(n)
%\end{equation}
%defined linearly, for each $\Mca$-cliques $\Pfr$ and $\Qfr$ in
%$\Cli\Mca(n)$, by
%\begin{equation}
%    (\Pfr \Op \Qfr)(x, y) := \Pfr(x, y) \Op \Qfr(x, y),
%\end{equation}
%where $(x, y)$ is any arc such that $1 \leq x < y \leq n + 1$.

For any $n \geq 1$, let
\begin{equation}
    \Op : \Cli\Mca(n) \otimes \Cli\Mca(n) \to \Cli\Mca(n)
\end{equation}
be the product defined for all $\Mca$-cliques $\Pfr$ and $\Qfr$ by
\begin{equation}
    (\Pfr \Op \Qfr)(x, y) := \Pfr(x, y) \Op \Qfr(x, y),
\end{equation}
where $(x, y)$ is any arc such that $1 \leq x < y \leq n + 1$,
and then extended linearly.
For instance, in $\Cli\Z$,
\begin{equation}
    \begin{tikzpicture}[scale=0.8,Centering]
        \node[CliquePoint](1)at(-0.50,-0.87){};
        \node[CliquePoint](2)at(-1.00,-0.00){};
        \node[CliquePoint](3)at(-0.50,0.87){};
        \node[CliquePoint](4)at(0.50,0.87){};
        \node[CliquePoint](5)at(1.00,0.00){};
        \node[CliquePoint](6)at(0.50,-0.87){};
        \draw[CliqueEmptyEdge](1)edge[]node[]{}(2);
        \draw[CliqueEmptyEdge](1)edge[]node[]{}(6);
        \draw[CliqueEmptyEdge](2)edge[]node[]{}(3);
        \draw[CliqueEdge](2)edge[bend right=30]node[CliqueLabel]
            {\begin{math}2\end{math}}(4);
        \draw[CliqueEdge](2)edge[bend left=30]node[CliqueLabel]
            {\begin{math}-1\end{math}}(6);
        \draw[CliqueEdge](3)edge[]node[CliqueLabel]
            {\begin{math}1\end{math}}(4);
        \draw[CliqueEmptyEdge](4)edge[]node[]{}(5);
        \draw[CliqueEdge](5)edge[]node[CliqueLabel]
            {\begin{math}-2\end{math}}(6);
    \end{tikzpicture}
    \enspace \Op \enspace
    \begin{tikzpicture}[scale=0.8,Centering]
        \node[CliquePoint](1)at(-0.50,-0.87){};
        \node[CliquePoint](2)at(-1.00,-0.00){};
        \node[CliquePoint](3)at(-0.50,0.87){};
        \node[CliquePoint](4)at(0.50,0.87){};
        \node[CliquePoint](5)at(1.00,0.00){};
        \node[CliquePoint](6)at(0.50,-0.87){};
        \draw[CliqueEmptyEdge](1)edge[]node[]{}(2);
        \draw[CliqueEmptyEdge](1)edge[]node[]{}(6);
        \draw[CliqueEdge](2)edge[]node[CliqueLabel]
            {\begin{math}3\end{math}}(3);
        \draw[CliqueEdge](2)edge[bend right=30]node[CliqueLabel]
            {\begin{math}1\end{math}}(4);
        \draw[CliqueEdge](1)edge[bend left=30]node[CliqueLabel]
            {\begin{math}-1\end{math}}(5);
        \draw[CliqueEdge](3)edge[]node[CliqueLabel]
            {\begin{math}1\end{math}}(4);
        \draw[CliqueEmptyEdge](4)edge[]node[]{}(5);
        \draw[CliqueEdge](5)edge[]node[CliqueLabel]
            {\begin{math}2\end{math}}(6);
    \end{tikzpicture}
    \enspace = \enspace
    \begin{tikzpicture}[scale=0.8,Centering]
        \node[CliquePoint](1)at(-0.50,-0.87){};
        \node[CliquePoint](2)at(-1.00,-0.00){};
        \node[CliquePoint](3)at(-0.50,0.87){};
        \node[CliquePoint](4)at(0.50,0.87){};
        \node[CliquePoint](5)at(1.00,0.00){};
        \node[CliquePoint](6)at(0.50,-0.87){};
        \draw[CliqueEmptyEdge](1)edge[]node[]{}(2);
        \draw[CliqueEmptyEdge](1)edge[]node[]{}(6);
        \draw[CliqueEmptyEdge](5)edge[]node[]{}(6);
        \draw[CliqueEdge](2)edge[]node[CliqueLabel]
            {\begin{math}3\end{math}}(3);
        \draw[CliqueEdge](2)edge[bend right=30]node[CliqueLabel,near end]
            {\begin{math}3\end{math}}(4);
        \draw[CliqueEdge](2)edge[bend left=30]node[CliqueLabel,near end]
            {\begin{math}-1\end{math}}(6);
        \draw[CliqueEdge](3)edge[]node[CliqueLabel]
            {\begin{math}2\end{math}}(4);
        \draw[CliqueEmptyEdge](4)edge[]node[]{}(5);
        \draw[CliqueEdge](1)edge[bend left=30]node[CliqueLabel,near end]
            {\begin{math}-1\end{math}}(5);
    \end{tikzpicture}\,.
\end{equation}
\medbreak

\begin{Proposition} \label{prop:rat_fct_cliques_product}
    Let $\Mca$ be a $\Z$-graded unitary magma and $\theta$ be a rank
    function of~$\Mca$. For any homogeneous elements $f$ and $g$ of
    $\Cli\Mca$ of the same arity,
    \begin{equation} \label{equ:rat_fct_cliques_product}
        \Frac_\theta(f) \Frac_\theta(g) = \Frac_\theta(f \Op g).
    \end{equation}
\end{Proposition}
\begin{proof}
    Let $\Pfr$ and $\Qfr$ be two $\Mca$-cliques of $\Cli\Mca$ of arity
    $n$. By definition of the operation $\Op$ on $\Cli\Mca(n)$ and the
    fact that $\theta$ is a unitary magma morphism,
    \begin{equation}\begin{split}
        \Frac_\theta(\Pfr) \Frac_\theta(\Qfr)
        & =
        \left(
        \prod_{(x, y) \in \Arcs_\Pfr}
        \left(u_x + \dots + u_{y - 1}\right)^{\theta(\Pfr(x, y))}
        \right)
        \left(
        \prod_{(x, y) \in \Arcs_\Qfr}
        \left(u_x + \dots + u_{y - 1}\right)^{\theta(\Qfr(x, y))}
        \right) \\
        & =
        \prod_{1 \leq x < y \leq n + 1}
        \left(u_x + \dots + u_{y - 1}\right)
            ^{\theta(\Pfr(x, y)) + \theta(\Qfr(x, y))} \\
        & =
        \prod_{1 \leq x < y \leq n + 1}
        \left(u_x + \dots + u_{y - 1}\right)
            ^{\theta(\Pfr(x, y) * \Qfr(x, y))} \\
        & = \Frac_\theta(\Pfr \Op \Qfr).
    \end{split}\end{equation}
    By the linearity of $\Frac_\theta$ and of $\Op$,
    \eqref{equ:rat_fct_cliques_product} follows.
\end{proof}
\medbreak

\begin{Proposition} \label{prop:rat_fct_cliques_inverse}
    Let $\Pfr$ be an $\Mca$-clique of $\Cli\Z$. Then,
    \begin{equation}
        \frac{1}{\Frac_\Id(\Pfr)} = \Frac_\Id((\Cli \eta)(\Pfr)),
    \end{equation}
    where $\eta : \Z \to \Z$ is the unitary magma morphism defined by
    $\eta(x) := -x$ for all $x \in \Z$.
\end{Proposition}
\begin{proof}
    Observe that $(\Cli \eta)(\Pfr)$ is the $\Mca$-clique obtained by
    relabeling each arc $(x, y)$ of $\Pfr$ by~$-\Pfr(x, y)$. Hence,
    since $\eta$ is a unitary magma morphism, we have
    \begin{equation}\begin{split}
        \Frac_\Id((\Cli \eta)(\Pfr))
        & =
        \prod_{(x, y) \in \Arcs_\Pfr}
        \left(u_x + \dots + u_{y - 1}\right)^{\theta(-\Pfr(x, y))} \\
        & =
        \prod_{(x, y) \in \Arcs_\Pfr}
        \left(u_x + \dots + u_{y - 1}\right)^{-\theta(\Pfr(x, y))} \\
        & =
        \frac{1}{\Frac_\Id(\Pfr)}
    \end{split}\end{equation}
    as expected.
\end{proof}
\medbreak

%%%%%%%%%%%%%%%%%%%%%%%%%%%%%%%%%%%%%%%%%%%%%%%%%%%%%%%%%%%%%%%%%%%%%%%%
%%%%%%%%%%%%%%%%%%%%%%%%%%%%%%%%%%%%%%%%%%%%%%%%%%%%%%%%%%%%%%%%%%%%%%%%
%%%%%%%%%%%%%%%%%%%%%%%%%%%%%%%%%%%%%%%%%%%%%%%%%%%%%%%%%%%%%%%%%%%%%%%%
\section{Quotients and suboperads}\label{sec:quotients_suboperads}
We define here quotients and suboperads of $\Cli\Mca$, leading to the
construction of some new operads involving various combinatorial objects
which are, basically, $\Mca$-cliques with some restrictions.
\medbreak

%%%%%%%%%%%%%%%%%%%%%%%%%%%%%%%%%%%%%%%%%%%%%%%%%%%%%%%%%%%%%%%%%%%%%%%%
%%%%%%%%%%%%%%%%%%%%%%%%%%%%%%%%%%%%%%%%%%%%%%%%%%%%%%%%%%%%%%%%%%%%%%%%
\subsection{Main substructures}%
\label{subsec:main_substructures}
Most of the natural subfamilies of $\Mca$-cliques that can be described
by simple combinatorial properties as $\Mca$-cliques with restrained
labels for the bases, edges, and diagonals, white $\Mca$-cliques,
$\Mca$-cliques with a fixed maximal crossing number,
$\Mca$-bubbles, $\Mca$-cliques with a fixed maximal value for their
degrees, nesting-free $\Mca$-cliques, and acyclic $\Mca$-cliques
inherit from the algebraic structure of operad of $\Cli\Mca$ and form
quotients and suboperads of $\Cli\Mca$ (see
Table~\ref{tab:main_substructures}).
\begin{table}[ht]
    \centering
    \begin{small}
    \begin{tabular}{c|c|c}
        Operad & Objects & Status with respect to $\Cli\Mca$
            \\ \hline \hline
        $\Lab_{B, E, D}\Mca$ & $\Mca$-cliques with restricted labels
            & Suboperad \\
        $\Whi\Mca$ & White $\Mca$-cliques & Suboperad \\
        $\Cro_k\Mca$ & $\Mca$-cliques of crossings at most $k$
            & Suboperad and quotient \\
        $\Bub\Mca$ & $\Mca$-bubbles
            & Quotient \\
        $\Deg_k\Mca$ & $\Mca$-cliques of degree at most $k$
            & Quotient \\
        $\Nes\Mca$ & Nesting-free $\Mca$-cliques
            & Quotient \\
        $\Acy\Mca$ & Acyclic $\Mca$-cliques
            & Quotient
    \end{tabular}
    \end{small}
    \smallbreak

    \caption{\footnotesize
    Operads constructed as suboperads or quotients of $\Cli\Mca$. All
    these operads depend on a unitary magma $\Mca$ which has, in some
    cases, to satisfy some precise conditions. Some of these operads
    depend also on a nonnegative integer~$k$ or subsets $B$, $E$, and
    $D$ of~$\Mca$.}
    \label{tab:main_substructures}
\end{table}
We construct and briefly study
here these main substructures of~$\Cli\Mca$.
\medbreak

%%%%%%%%%%%%%%%%%%%%%%%%%%%%%%%%%%%%%%%%%%%%%%%%%%%%%%%%%%%%%%%%%%%%%%%%
\subsubsection{Restricting the labels}%
\label{subsubsec:suboperad_Cli_M_labels}
In what follows, if $X$ and $Y$ are two subsets of $\Mca$,
$X \Op Y$ denotes the set $\{x \Op y : x \in X \mbox{ and } y \in Y\}$.
\medbreak

Let $B$, $E$, and $D$ be three subsets of $\Mca$ and
$\Lab_{B, E, D}\Mca$ be the subspace of $\Cli\Mca$ generated by all
$\Mca$-cliques $\Pfr$ such that the bases of $\Pfr$ are labeled by $B$,
all edges of $\Pfr$ are labeled by $E$, and all diagonals of $\Pfr$ are
labeled by~$D$.
\medbreak

\begin{Proposition} \label{prop:suboperad_Cli_M_labels}
    Let $\Mca$ be a unitary magma and $B$, $E$, and $D$ be three subsets
    of $\Mca$. If $\Unit_\Mca \in B$, $\Unit_\Mca \in D$, and
    $E \Op B \subseteq D$, $\Lab_{B, E, D}\Mca$ is a suboperad
    of~$\Cli\Mca$.
\end{Proposition}
\begin{proof}
    First, since $\Unit_\Mca \in B$, the unit $\UnitClique$ of
    $\Cli\Mca$ belongs to $\Lab_{B, E, D}\Mca$. Consider now two
    $\Mca$-cliques $\Pfr$ and $\Qfr$ of $\Lab_{B, E, D}\Mca$ and a
    partial composition $\Rfr := \Pfr \circ_i \Qfr$ for a valid integer
    $i$. By the definition of the partial composition of $\Cli\Mca$, the
    base of $\Rfr$ has the same label as the base of $\Pfr$, and all
    edges of $\Rfr$ have labels coming from the ones of $\Pfr$ and
    $\Qfr$. Moreover, all diagonals of $\Rfr$ are either non-solid, or
    come from diagonals of $\Pfr$ and $\Qfr$, or are the diagonal
    $\Rfr(i, i + |\Qfr|)$ which is labeled by $\Pfr_i \Op \Qfr_0$. Since
    $\Unit_\Mca \in D$, $\Pfr_i \in E$, $\Qfr_0 \in B$, and
    $E \Op B \subseteq D$, all the labels of these diagonals are in $D$.
    For these reasons, $\Rfr$ is in $\Lab_{B, E, D}\Mca$. This implies
    the statement of the proposition.
\end{proof}
\medbreak

\begin{Proposition} \label{prop:suboperad_Cli_M_labels_dimensions}
    Let $\Mca$ be a unitary magma and $B$, $E$, and $D$ be three
    finite subsets of $\Mca$. For all $n \geq 2$,
    \begin{equation} \label{equ:suboperad_Cli_M_labels_dimensions}
        \dim \Lab_{B, E, D}\Mca(n) =
         b e^n d^{(n + 1)(n - 2) / 2},
    \end{equation}
    where $b := \# B$, $e := \# E$, and $d := \# D$.
\end{Proposition}
\begin{proof}
    By Proposition~\ref{prop:dimensions_Cli_M}, there are
    $m^{\binom{n + 1}{2}}$ $\Mca$-cliques of arity $n$, where
    $m := \# \Mca$. Hence, there are
    $m^{\binom{n + 1}{2}} / m^{n + 1}$ $\Mca$-cliques of arity $n$ with
    all edges and the base labeled by $\Unit_\Mca$. This also says that
    there are $d^{\binom{n + 1}{2}} / d^{n + 1}$ $\Mca$-cliques of arity
    $n$ with all diagonals labeled by $D$ and all edges and the base
    labeled by $\Unit_\Mca$. Since an $\Mca$-clique of
    $\Lab_{B, E, D}\Mca(n)$ has its $n$ edges labeled by $E$ and its
    base labeled by $B$, \eqref{equ:suboperad_Cli_M_labels_dimensions}
    follows.
\end{proof}
\medbreak

%%%%%%%%%%%%%%%%%%%%%%%%%%%%%%%%%%%%%%%%%%%%%%%%%%%%%%%%%%%%%%%%%%%%%%%%
\subsubsection{White cliques}%
\label{subsubsec:suboperad_Cli_M_white}
Let $\Whi\Mca$ be the subspace of $\Cli\Mca$ generated by all white
$\Mca$-cliques. Since, by definition of white $\Mca$-cliques,
\begin{equation}
    \Whi\Mca = \Lab_{\{\Unit_\Mca\}, \{\Unit_\Mca\}, \Mca}\Mca,
\end{equation}
by Proposition~\ref{prop:suboperad_Cli_M_labels}, $\Whi\Mca$ is a
suboperad of $\Cli\Mca$. It follows from
Proposition~\ref{prop:suboperad_Cli_M_labels_dimensions} that when
$\Mca$ is finite, the dimensions of $\Whi\Mca$ satisfy, for any
$n \geq 2$,
\begin{equation}
    \dim \Whi\Mca(n) =
    m^{(n + 1)(n - 2) / 2},
\end{equation}
where $m := \# \Mca$.
\medbreak

%%%%%%%%%%%%%%%%%%%%%%%%%%%%%%%%%%%%%%%%%%%%%%%%%%%%%%%%%%%%%%%%%%%%%%%%
\subsubsection{Restricting the crossings}%
\label{subsubsec:quotient_Cli_M_crossings}
Let $k \geq 0$ be an integer and $\Rel_{\Cro_k\Mca}$ be the subspace of
$\Cli\Mca$ generated by all $\Mca$-cliques $\Pfr$ such that
$\Cros(\Pfr) \geq k + 1$. As a quotient of graded vector spaces,
\begin{equation}
    \Cro_k\Mca := \Cli\Mca / \Rel_{\Cro_k\Mca}
\end{equation}
is the linear span of all $\Mca$-cliques $\Pfr$ such that
$\Cros(\Pfr) \leq k$.
\medbreak

\begin{Proposition} \label{prop:quotient_Cli_M_crossings}
    Let $\Mca$ be a unitary magma and $k \geq 0$ be an integer. Then,
    the space $\Cro_k\Mca$ is a quotient operad of $\Cli\Mca$ and is
    isomorphic to the suboperad of~$\Cli\Mca$ restricted to the subspace
    generated by all $\Mca$-cliques with crossing numbers no greater
    than~$k$.
\end{Proposition}
\begin{proof}
    We first prove that $\Cro_k\Mca$ is a quotient of $\Cli\Mca$.
    For this, observe that if $\Pfr$ and $\Qfr$ are two $\Mca$-cliques,
    \begin{equation} \label{equ:quotient_Cli_M_crossings}
        \Cros(\Pfr \circ_i \Qfr) = \max\{\Cros(\Pfr), \Cros(\Qfr)\}
    \end{equation}
    for any valid integer $i$. For this reason, if $\Pfr$ is an
    $\Mca$-clique of $\Rel_{\Cro_k\Mca}$, each clique obtained by a
    partial composition involving $\Pfr$ and other $\Mca$-cliques is
    still in $\Rel_{\Cro_k\Mca}$. This proves that $\Rel_{\Cro_k\Mca}$
    is an operad ideal of $\Cli\Mca$ and hence, that $\Cro_k\Mca$ is a
    quotient of $\Cli\Mca$.
    \smallbreak

    To prove the second part of the statement, consider two
    $\Mca$-cliques $\Pfr$ and $\Qfr$ of $\Cro_k\Mca$.
    By~\eqref{equ:quotient_Cli_M_crossings}, all $\Mca$-cliques
    $\Pfr \circ_i \Qfr$ are still in $\Cro_k\Mca$, for all valid
    integers $i$. Moreover, the unit $\UnitClique$ of $\Cli\Mca$
    belongs to $\Cro_k\Mca$. This implies the desired property.
\end{proof}
\medbreak

For instance, in the operad $\Cro_2\Z$, we have
\begin{equation}
    \begin{tikzpicture}[scale=0.7,Centering]
        \node[CliquePoint](1)at(-0.59,-0.81){};
        \node[CliquePoint](2)at(-0.95,0.31){};
        \node[CliquePoint](3)at(-0.00,1.00){};
        \node[CliquePoint](4)at(0.95,0.31){};
        \node[CliquePoint](5)at(0.59,-0.81){};
        \draw[CliqueEmptyEdge](1)edge[]node[]{}(2);
        \draw[CliqueEdge](1)
            edge[bend right=30]node[CliqueLabel,near start]
            {\begin{math}2\end{math}}(3);
        \draw[CliqueEmptyEdge](1)edge[]node[]{}(5);
        \draw[CliqueEmptyEdge](2)edge[]node[]{}(3);
        \draw[CliqueEdge](3)edge[]node[CliqueLabel]
            {\begin{math}1\end{math}}(4);
        \draw[CliqueEdge](4)edge[]node[CliqueLabel]
            {\begin{math}2\end{math}}(5);
        \draw[CliqueEdge](2)edge[]node[CliqueLabel, near end]
            {\begin{math}1\end{math}}(4);
        \draw[CliqueEdge](2)
            edge[bend left=30]node[CliqueLabel,near end]
            {\begin{math}3\end{math}}(5);
    \end{tikzpicture}
    \circ_3
    \begin{tikzpicture}[scale=0.6,Centering]
        \node[CliquePoint](1)at(-0.71,-0.71){};
        \node[CliquePoint](2)at(-0.71,0.71){};
        \node[CliquePoint](3)at(0.71,0.71){};
        \node[CliquePoint](4)at(0.71,-0.71){};
        \draw[CliqueEmptyEdge](1)edge[]node[]{}(2);
        \draw[CliqueEdge](1)edge[]node[CliqueLabel]
            {\begin{math}2\end{math}}(3);
        \draw[CliqueEmptyEdge](1)edge[]node[]{}(4);
        \draw[CliqueEmptyEdge](2)edge[]node[]{}(3);
        \draw[CliqueEdge](3)edge[]node[CliqueLabel]
            {\begin{math}1\end{math}}(4);
    \end{tikzpicture}
    =
    \begin{tikzpicture}[scale=0.9,Centering]
        \node[CliquePoint](1)at(-0.43,-0.90){};
        \node[CliquePoint](2)at(-0.97,-0.22){};
        \node[CliquePoint](3)at(-0.78,0.62){};
        \node[CliquePoint](4)at(-0.00,1.00){};
        \node[CliquePoint](5)at(0.78,0.62){};
        \node[CliquePoint](6)at(0.97,-0.22){};
        \node[CliquePoint](7)at(0.43,-0.90){};
        \draw[CliqueEmptyEdge](1)edge[]node[]{}(2);
        \draw[CliqueEmptyEdge](1)edge[]node[]{}(7);
        \draw[CliqueEmptyEdge](2)edge[]node[]{}(3);
        \draw[CliqueEmptyEdge](3)edge[]node[]{}(4);
        \draw[CliqueEmptyEdge](4)edge[]node[]{}(5);
        \draw[CliqueEdge](5)edge[]node[CliqueLabel]
            {\begin{math}1\end{math}}(6);
        \draw[CliqueEdge](6)edge[]node[CliqueLabel]
            {\begin{math}2\end{math}}(7);
        \draw[CliqueEdge](1)
            edge[bend right=30]node[CliqueLabel,near end]
            {\begin{math}2\end{math}}(3);
        \draw[CliqueEdge](2)edge[]node[CliqueLabel,near end]
            {\begin{math}1\end{math}}(6);
        \draw[CliqueEdge](2)
            edge[bend left=30]node[CliqueLabel,near end]
            {\begin{math}3\end{math}}(7);
        \draw[CliqueEdge](3)edge[]node[CliqueLabel]
            {\begin{math}2\end{math}}(5);
        \draw[CliqueEdge](3)edge[]node[CliqueLabel]
            {\begin{math}1\end{math}}(6);
    \end{tikzpicture}\,.
\end{equation}
\medbreak

When $0 \leq k' \leq k$ are integers, by
Proposition~\ref{prop:quotient_Cli_M_crossings}, $\Cro_k\Mca$ and
$\Cro_{k'}\Mca$ are both quotients and suboperads of $\Cli\Mca$. First,
since any $\Mca$-clique of $\Cro_{k'}\Mca$ is also an $\Mca$-clique of
$\Cro_k\Mca$, $\Cro_{k'}\Mca$ is a suboperad of~$\Cro_k\Mca$. Second,
since $\Rel_{\Cro_k\Mca}$ is a subspace of $\Rel_{\Cro_{k'}\Mca}$,
$\Cro_{k'}\Mca$ is a quotient of~$\Cro_k\Mca$.
\medbreak

Observe that $\Cro_0\Mca$ is the linear span of all noncrossing
$\Mca$-cliques. We can see these objects as noncrossing
configurations~\cite{FN99} where the edges and bases are colored by
elements of $\Mca$ and the diagonals by elements of $\bar{\Mca}$. The
operad $\Cro_0\Mca$ has a lot of combinatorial and algebraic properties
and will be studied in detail in~\cite{Cliques2}.
\medbreak

%%%%%%%%%%%%%%%%%%%%%%%%%%%%%%%%%%%%%%%%%%%%%%%%%%%%%%%%%%%%%%%%%%%%%%%%
\subsubsection{Bubbles}
\label{subsubsec:quotient_Cli_M_bubbles}
Let $\Rel_{\Bub\Mca}$ be the subspace of $\Cli\Mca$ generated by all
$\Mca$-cliques that are not bubbles. As a quotient of graded vector
spaces,
\begin{equation}
    \Bub\Mca := \Cli\Mca / \Rel_{\Bub\Mca}
\end{equation}
is the linear span of all $\Mca$-bubbles.
\medbreak

\begin{Proposition} \label{prop:quotient_Cli_M_bubbles}
    Let $\Mca$ be a unitary magma. Then, the space $\Bub_\Mca$ is a
    quotient operad of $\Cli\Mca$.
\end{Proposition}
\begin{proof}
    If $\Pfr$ and $\Qfr$ are two $\Mca$-cliques, all solid diagonals
    of $\Pfr$ and $\Qfr$ appear in $\Pfr \circ_i \Qfr$, for any valid
    integer $i$. For this reason, if $\Pfr$ is an $\Mca$-clique of
    $\Rel_{\Bub\Mca}$, each $\Mca$-clique obtained by a partial
    composition involving $\Pfr$ and other $\Mca$-cliques is still in
    $\Rel_{\Bub\Mca}$. This proves that $\Rel_{\Bub\Mca}$ is an operad
    ideal of $\Cli\Mca$ and implies the statement of the proposition.
\end{proof}
\medbreak

For instance, in the operad $\Bub\Z$, we have
\vspace{-1.75em}
\begin{multicols}{2}
\begin{subequations}
\begin{equation}
    \begin{tikzpicture}[scale=0.6,Centering]
        \node[CliquePoint](1)at(-0.59,-0.81){};
        \node[CliquePoint](2)at(-0.95,0.31){};
        \node[CliquePoint](3)at(-0.00,1.00){};
        \node[CliquePoint](4)at(0.95,0.31){};
        \node[CliquePoint](5)at(0.59,-0.81){};
        \draw[CliqueEmptyEdge](1)edge[]node[]{}(2);
        \draw[CliqueEmptyEdge](1)edge[]node[]{}(5);
        \draw[CliqueEmptyEdge](2)edge[]node[]{}(3);
        \draw[CliqueEdge](3)edge[]node[CliqueLabel]
            {\begin{math}1\end{math}}(4);
        \draw[CliqueEdge](4)edge[]node[CliqueLabel]
            {\begin{math}2\end{math}}(5);
    \end{tikzpicture}
    \circ_2
    \begin{tikzpicture}[scale=0.5,Centering]
        \node[CliquePoint](1)at(-0.71,-0.71){};
        \node[CliquePoint](2)at(-0.71,0.71){};
        \node[CliquePoint](3)at(0.71,0.71){};
        \node[CliquePoint](4)at(0.71,-0.71){};
        \draw[CliqueEmptyEdge](1)edge[]node[]{}(2);
        \draw[CliqueEmptyEdge](1)edge[]node[]{}(4);
        \draw[CliqueEmptyEdge](2)edge[]node[]{}(3);
        \draw[CliqueEdge](3)edge[]node[CliqueLabel]
            {\begin{math}1\end{math}}(4);
    \end{tikzpicture}
    =
    \begin{tikzpicture}[scale=0.8,Centering]
        \node[CliquePoint](1)at(-0.43,-0.90){};
        \node[CliquePoint](2)at(-0.97,-0.22){};
        \node[CliquePoint](3)at(-0.78,0.62){};
        \node[CliquePoint](4)at(-0.00,1.00){};
        \node[CliquePoint](5)at(0.78,0.62){};
        \node[CliquePoint](6)at(0.97,-0.22){};
        \node[CliquePoint](7)at(0.43,-0.90){};
        \draw[CliqueEmptyEdge](1)edge[]node[]{}(2);
        \draw[CliqueEmptyEdge](1)edge[]node[]{}(7);
        \draw[CliqueEmptyEdge](2)edge[]node[]{}(3);
        \draw[CliqueEmptyEdge](3)edge[]node[]{}(4);
        \draw[CliqueEmptyEdge](4)edge[]node[]{}(5);
        \draw[CliqueEdge](4)edge[]node[CliqueLabel]
            {\begin{math}1\end{math}}(5);
        \draw[CliqueEdge](5)edge[]node[CliqueLabel]
            {\begin{math}1\end{math}}(6);
        \draw[CliqueEdge](6)edge[]node[CliqueLabel]
            {\begin{math}2\end{math}}(7);
    \end{tikzpicture}\,,
\end{equation}
\begin{equation}
    \begin{tikzpicture}[scale=0.6,Centering]
        \node[CliquePoint](1)at(-0.59,-0.81){};
        \node[CliquePoint](2)at(-0.95,0.31){};
        \node[CliquePoint](3)at(-0.00,1.00){};
        \node[CliquePoint](4)at(0.95,0.31){};
        \node[CliquePoint](5)at(0.59,-0.81){};
        \draw[CliqueEmptyEdge](1)edge[]node[]{}(2);
        \draw[CliqueEmptyEdge](1)edge[]node[]{}(5);
        \draw[CliqueEmptyEdge](2)edge[]node[]{}(3);
        \draw[CliqueEdge](3)edge[]node[CliqueLabel]
            {\begin{math}-1\end{math}}(4);
        \draw[CliqueEdge](4)edge[]node[CliqueLabel]
            {\begin{math}2\end{math}}(5);
    \end{tikzpicture}
    \circ_3
    \begin{tikzpicture}[scale=0.5,Centering]
        \node[CliquePoint](1)at(-0.71,-0.71){};
        \node[CliquePoint](2)at(-0.71,0.71){};
        \node[CliquePoint](3)at(0.71,0.71){};
        \node[CliquePoint](4)at(0.71,-0.71){};
        \draw[CliqueEmptyEdge](1)edge[]node[]{}(2);
        \draw[CliqueEmptyEdge](1)edge[]node[]{}(4);
        \draw[CliqueEmptyEdge](2)edge[]node[]{}(3);
        \draw[CliqueEdge](3)edge[]node[CliqueLabel]
            {\begin{math}1\end{math}}(4);
        \draw[CliqueEdge](1)edge[]node[CliqueLabel]
            {\begin{math}1\end{math}}(4);
    \end{tikzpicture}
    =
    \begin{tikzpicture}[scale=0.8,Centering]
        \node[CliquePoint](1)at(-0.43,-0.90){};
        \node[CliquePoint](2)at(-0.97,-0.22){};
        \node[CliquePoint](3)at(-0.78,0.62){};
        \node[CliquePoint](4)at(-0.00,1.00){};
        \node[CliquePoint](5)at(0.78,0.62){};
        \node[CliquePoint](6)at(0.97,-0.22){};
        \node[CliquePoint](7)at(0.43,-0.90){};
        \draw[CliqueEmptyEdge](1)edge[]node[]{}(2);
        \draw[CliqueEmptyEdge](1)edge[]node[]{}(7);
        \draw[CliqueEmptyEdge](2)edge[]node[]{}(3);
        \draw[CliqueEmptyEdge](3)edge[]node[]{}(4);
        \draw[CliqueEmptyEdge](4)edge[]node[]{}(5);
        \draw[CliqueEdge](5)edge[]node[CliqueLabel]
            {\begin{math}1\end{math}}(6);
        \draw[CliqueEdge](6)edge[]node[CliqueLabel]
            {\begin{math}2\end{math}}(7);
    \end{tikzpicture}\,,
\end{equation}

\begin{equation}
    \begin{tikzpicture}[scale=0.6,Centering]
        \node[CliquePoint](1)at(-0.59,-0.81){};
        \node[CliquePoint](2)at(-0.95,0.31){};
        \node[CliquePoint](3)at(-0.00,1.00){};
        \node[CliquePoint](4)at(0.95,0.31){};
        \node[CliquePoint](5)at(0.59,-0.81){};
        \draw[CliqueEmptyEdge](1)edge[]node[]{}(2);
        \draw[CliqueEmptyEdge](1)edge[]node[]{}(5);
        \draw[CliqueEmptyEdge](2)edge[]node[]{}(3);
        \draw[CliqueEdge](3)edge[]node[CliqueLabel]
            {\begin{math}1\end{math}}(4);
        \draw[CliqueEdge](4)edge[]node[CliqueLabel]
            {\begin{math}2\end{math}}(5);
    \end{tikzpicture}
    \circ_3
    \begin{tikzpicture}[scale=0.5,Centering]
        \node[CliquePoint](1)at(-0.71,-0.71){};
        \node[CliquePoint](2)at(-0.71,0.71){};
        \node[CliquePoint](3)at(0.71,0.71){};
        \node[CliquePoint](4)at(0.71,-0.71){};
        \draw[CliqueEmptyEdge](1)edge[]node[]{}(2);
        \draw[CliqueEmptyEdge](1)edge[]node[]{}(4);
        \draw[CliqueEmptyEdge](2)edge[]node[]{}(3);
        \draw[CliqueEdge](3)edge[]node[CliqueLabel]
            {\begin{math}1\end{math}}(4);
    \end{tikzpicture}
    = 0,
\end{equation}
\begin{equation}
    \begin{tikzpicture}[scale=0.6,Centering]
        \node[CliquePoint](1)at(-0.59,-0.81){};
        \node[CliquePoint](2)at(-0.95,0.31){};
        \node[CliquePoint](3)at(-0.00,1.00){};
        \node[CliquePoint](4)at(0.95,0.31){};
        \node[CliquePoint](5)at(0.59,-0.81){};
        \draw[CliqueEmptyEdge](1)edge[]node[]{}(2);
        \draw[CliqueEmptyEdge](1)edge[]node[]{}(5);
        \draw[CliqueEmptyEdge](2)edge[]node[]{}(3);
        \draw[CliqueEdge](3)edge[]node[CliqueLabel]
            {\begin{math}1\end{math}}(4);
        \draw[CliqueEdge](4)edge[]node[CliqueLabel]
            {\begin{math}2\end{math}}(5);
    \end{tikzpicture}
    \circ_2
    \begin{tikzpicture}[scale=0.5,Centering]
        \node[CliquePoint](1)at(-0.71,-0.71){};
        \node[CliquePoint](2)at(-0.71,0.71){};
        \node[CliquePoint](3)at(0.71,0.71){};
        \node[CliquePoint](4)at(0.71,-0.71){};
        \draw[CliqueEmptyEdge](1)edge[]node[]{}(2);
        \draw[CliqueEmptyEdge](2)edge[]node[]{}(3);
        \draw[CliqueEdge](3)edge[]node[CliqueLabel]
            {\begin{math}1\end{math}}(4);
        \draw[CliqueEdge](1)edge[]node[CliqueLabel]
            {\begin{math}2\end{math}}(4);
    \end{tikzpicture}
    = 0.
\end{equation}
\end{subequations}
\end{multicols}
\medbreak

When $\Mca$ is finite, the dimensions of $\Bub\Mca$ satisfy, for any
$n \geq 2$,
\begin{equation}
    \dim \Bub\Mca(n) = m^{n + 1},
\end{equation}
where $m := \# \Mca$.
\medbreak

%%%%%%%%%%%%%%%%%%%%%%%%%%%%%%%%%%%%%%%%%%%%%%%%%%%%%%%%%%%%%%%%%%%%%%%%
\subsubsection{Restricting the degrees}%
\label{subsubsec:quotient_Cli_M_degrees}
Let $k \geq 0$ be an integer and $\Rel_{\Deg_k\Mca}$ be the subspace of
$\Cli\Mca$ generated by all $\Mca$-cliques $\Pfr$ such that
$\Degr(\Pfr) \geq k + 1$. As a quotient of graded vector spaces,
\begin{equation}
    \Deg_k\Mca := \Cli\Mca / \Rel_{\Deg_k\Mca}
\end{equation}
is the linear span of all $\Mca$-cliques $\Pfr$ such that
$\Degr(\Pfr) \leq k$.
\medbreak

\begin{Proposition} \label{prop:quotient_Cli_M_degrees}
    Let $\Mca$ be a unitary magma without nontrivial unit divisors and
    $k \geq 0$ be an integer. Then, the space $\Deg_k\Mca$ is a quotient
    operad of $\Cli\Mca$.
\end{Proposition}
\begin{proof}
    Since $\Mca$ has no nontrivial unit divisors, for any $\Mca$-cliques
    $\Pfr$ and $\Qfr$ of $\Cli\Mca$, each solid arc of $\Pfr$ (resp.
    $\Qfr$) gives rise to a solid arc in $\Pfr \circ_i \Qfr$, for any
    valid integer $i$. Hence,
    \begin{equation}
        \Degr(\Pfr \circ_i \Qfr) \geq \max\{\Degr(\Pfr), \Degr(\Qfr)\},
    \end{equation}
    and then, if $\Pfr$ is an $\Mca$-clique of $\Rel_{\Deg_k\Mca}$,
    each $\Mca$-clique obtained by a partial composition involving
    $\Pfr$ and other $\Mca$-cliques is still in $\Rel_{\Deg_k\Mca}$.
    This proves that $\Rel_{\Deg_k\Mca}$ is an operad ideal of
    $\Cli\Mca$ and implies the statement of the proposition.
\end{proof}
\medbreak

For instance, in the operad $\Deg_3\Dbb_2$ (observe that $\Dbb_2$ is
a unitary magma without nontrivial unit divisors), we have
\vspace{-1.75em}
\begin{multicols}{2}
\begin{subequations}
\begin{equation}
    \begin{tikzpicture}[scale=0.6,Centering]
        \node[CliquePoint](1)at(-0.59,-0.81){};
        \node[CliquePoint](2)at(-0.95,0.31){};
        \node[CliquePoint](3)at(-0.00,1.00){};
        \node[CliquePoint](4)at(0.95,0.31){};
        \node[CliquePoint](5)at(0.59,-0.81){};
        \draw[CliqueEmptyEdge](1)edge[]node[]{}(2);
        \draw[CliqueEmptyEdge](1)edge[]node[]{}(5);
        \draw[CliqueEdge](2)edge[]node[CliqueLabel]
            {\begin{math}\Dtt_1\end{math}}(3);
        \draw[CliqueEmptyEdge](3)edge[]node[]{}(4);
        \draw[CliqueEdge](4)edge[]node[CliqueLabel]
            {\begin{math}0\end{math}}(5);
        \draw[CliqueEmptyEdge](1)edge[]node[]{}(2);
        \draw[CliqueEdge](1)
            edge[bend left=30]node[CliqueLabel,near start]
            {\begin{math}0\end{math}}(4);
        \draw[CliqueEdge](2)
            edge[bend left=30]node[CliqueLabel,near start]
            {\begin{math}\Dtt_1\end{math}}(5);
    \end{tikzpicture}
    \circ_2
    \begin{tikzpicture}[scale=0.5,Centering]
        \node[CliquePoint](1)at(-0.71,-0.71){};
        \node[CliquePoint](2)at(-0.71,0.71){};
        \node[CliquePoint](3)at(0.71,0.71){};
        \node[CliquePoint](4)at(0.71,-0.71){};
        \draw[CliqueEmptyEdge](1)edge[]node[]{}(2);
        \draw[CliqueEdge](1)edge[]node[CliqueLabel]
            {\begin{math}0\end{math}}(4);
        \draw[CliqueEmptyEdge](2)edge[]node[]{}(3);
        \draw[CliqueEdge](3)edge[]node[CliqueLabel]
            {\begin{math}0\end{math}}(4);
        \draw[CliqueEdge](2)edge[]node[CliqueLabel]
            {\begin{math}\Dtt_1\end{math}}(4);
    \end{tikzpicture}
    =
    \begin{tikzpicture}[scale=0.8,Centering]
        \node[CliquePoint](1)at(-0.43,-0.90){};
        \node[CliquePoint](2)at(-0.97,-0.22){};
        \node[CliquePoint](3)at(-0.78,0.62){};
        \node[CliquePoint](4)at(-0.00,1.00){};
        \node[CliquePoint](5)at(0.78,0.62){};
        \node[CliquePoint](6)at(0.97,-0.22){};
        \node[CliquePoint](7)at(0.43,-0.90){};
        \draw[CliqueEmptyEdge](1)edge[]node[]{}(2);
        \draw[CliqueEmptyEdge](1)edge[]node[]{}(7);
        \draw[CliqueEmptyEdge](2)edge[]node[]{}(3);
        \draw[CliqueEmptyEdge](3)edge[]node[]{}(4);
        \draw[CliqueEmptyEdge](4)edge[]node[]{}(5);
        \draw[CliqueEdge](3)edge[bend right=30]node[CliqueLabel]
            {\begin{math}\Dtt_1\end{math}}(5);
        \draw[CliqueEdge](4)edge[]node[CliqueLabel]
            {\begin{math}0\end{math}}(5);
        \draw[CliqueEmptyEdge](5)edge[]node[CliqueLabel]{}(6);
        \draw[CliqueEdge](6)edge[]node[CliqueLabel]
            {\begin{math}0\end{math}}(7);
        \draw[CliqueEdge](2)edge[bend left=30]node[CliqueLabel,near end]
            {\begin{math}\Dtt_1\end{math}}(7);
        \draw[CliqueEdge](2)edge[bend right=30]node[CliqueLabel]
            {\begin{math}0\end{math}}(5);
    \end{tikzpicture}\,,
\end{equation}

\begin{equation}
    \begin{tikzpicture}[scale=0.6,Centering]
        \node[CliquePoint](1)at(-0.59,-0.81){};
        \node[CliquePoint](2)at(-0.95,0.31){};
        \node[CliquePoint](3)at(-0.00,1.00){};
        \node[CliquePoint](4)at(0.95,0.31){};
        \node[CliquePoint](5)at(0.59,-0.81){};
        \draw[CliqueEmptyEdge](1)edge[]node[]{}(2);
        \draw[CliqueEmptyEdge](1)edge[]node[]{}(5);
        \draw[CliqueEdge](2)edge[]node[CliqueLabel]
            {\begin{math}\Dtt_1\end{math}}(3);
        \draw[CliqueEmptyEdge](3)edge[]node[]{}(4);
        \draw[CliqueEdge](4)edge[]node[CliqueLabel]
            {\begin{math}0\end{math}}(5);
        \draw[CliqueEmptyEdge](1)edge[]node[]{}(2);
        \draw[CliqueEdge](1)
            edge[bend left=30]node[CliqueLabel,near start]
            {\begin{math}0\end{math}}(4);
        \draw[CliqueEdge](2)
            edge[bend left=30]node[CliqueLabel,near start]
            {\begin{math}\Dtt_1\end{math}}(5);
    \end{tikzpicture}
    \circ_3
    \begin{tikzpicture}[scale=0.5,Centering]
        \node[CliquePoint](1)at(-0.71,-0.71){};
        \node[CliquePoint](2)at(-0.71,0.71){};
        \node[CliquePoint](3)at(0.71,0.71){};
        \node[CliquePoint](4)at(0.71,-0.71){};
        \draw[CliqueEmptyEdge](1)edge[]node[]{}(2);
        \draw[CliqueEdge](1)edge[]node[CliqueLabel]
            {\begin{math}0\end{math}}(4);
        \draw[CliqueEmptyEdge](2)edge[]node[]{}(3);
        \draw[CliqueEdge](3)edge[]node[CliqueLabel]
            {\begin{math}0\end{math}}(4);
        \draw[CliqueEdge](2)edge[]node[CliqueLabel]
            {\begin{math}\Dtt_1\end{math}}(4);
    \end{tikzpicture}
    = 0.
\end{equation}
\end{subequations}
\end{multicols}
\medbreak

When $0 \leq k' \leq k$ are integers, by
Proposition~\ref{prop:quotient_Cli_M_degrees}, $\Deg_k\Mca$ and
$\Deg_{k'}\Mca$ are both quotient operads of $\Cli\Mca$. Moreover,
since $\Rel_{\Deg_k\Mca}$ is a subspace of $\Rel_{\Deg_{k'}\Mca}$,
$\Deg_{k'}\Mca$ is a quotient operad of $\Deg_k\Mca$.
\medbreak

Observe that $\Deg_0\Mca$ is the linear span of all $\Mca$-cliques
without solid arcs. If $\Pfr$ and $\Qfr$ are such $\Mca$-cliques, all
partial compositions $\Pfr \circ_i \Qfr$ are equal to the unique
$\Mca$-clique without solid arcs of arity $|\Pfr| + |\Qfr| - 1$. For
this reason, $\Deg_0\Mca$ is the associative operad~$\As$.
\medbreak

Any skeleton of an $\Mca$-clique of arity $n$ of $\Deg_1\Mca$ can be
seen as a partition of the set $[n + 1]$ into singletons or pairs.
Therefore, $\Deg_1\Mca$ can be seen as an operad on such colored
partitions, where each pair of the partitions has one color from the
set $\bar{\Mca}$. In the operad $\Deg_1\Dbb_0$ (observe that $\Dbb_0$ is
the only unitary magma without nontrivial unit divisors on two
elements), one has for instance
\vspace{-1.75em}
\begin{multicols}{2}
\begin{subequations}
\begin{equation} \label{equ:example_involutions_1}
    \begin{tikzpicture}[scale=0.6,Centering]
        \node[CliquePoint](1)at(-0.59,-0.81){};
        \node[CliquePoint](2)at(-0.95,0.31){};
        \node[CliquePoint](3)at(-0.00,1.00){};
        \node[CliquePoint](4)at(0.95,0.31){};
        \node[CliquePoint](5)at(0.59,-0.81){};
        \draw[CliqueEmptyEdge](1)edge[]node[]{}(2);
        \draw[CliqueEmptyEdge](1)edge[]node[]{}(5);
        \draw[CliqueEmptyEdge](2)edge[]node[]{}(3);
        \draw[CliqueEmptyEdge](3)edge[]node[]{}(4);
        \draw[CliqueEmptyEdge](4)edge[]node[]{}(5);
        \draw[CliqueEdge](1)edge[bend left=30]node[CliqueLabel]
            {\begin{math}0\end{math}}(4);
    \end{tikzpicture}
    \circ_2
    \begin{tikzpicture}[scale=0.5,Centering]
        \node[CliquePoint](1)at(-0.71,-0.71){};
        \node[CliquePoint](2)at(-0.71,0.71){};
        \node[CliquePoint](3)at(0.71,0.71){};
        \node[CliquePoint](4)at(0.71,-0.71){};
        \draw[CliqueEmptyEdge](1)edge[]node[]{}(2);
        \draw[CliqueEmptyEdge](1)edge[]node[]{}(4);
        \draw[CliqueEmptyEdge](3)edge[]node[]{}(4);
        \draw[CliqueEdge](1)edge[]node[CliqueLabel,near start]
            {\begin{math}0\end{math}}(3);
        \draw[CliqueEmptyEdge](2)edge[]node[]{}(3);
        \draw[CliqueEdge](2)edge[]node[CliqueLabel,near end]
            {\begin{math}0\end{math}}(4);
    \end{tikzpicture}
    =
    \begin{tikzpicture}[scale=0.8,Centering]
        \node[CliquePoint](1)at(-0.43,-0.90){};
        \node[CliquePoint](2)at(-0.97,-0.22){};
        \node[CliquePoint](3)at(-0.78,0.62){};
        \node[CliquePoint](4)at(-0.00,1.00){};
        \node[CliquePoint](5)at(0.78,0.62){};
        \node[CliquePoint](6)at(0.97,-0.22){};
        \node[CliquePoint](7)at(0.43,-0.90){};
        \draw[CliqueEmptyEdge](1)edge[]node[]{}(2);
        \draw[CliqueEmptyEdge](1)edge[]node[]{}(7);
        \draw[CliqueEmptyEdge](2)edge[]node[]{}(3);
        \draw[CliqueEmptyEdge](3)edge[]node[]{}(4);
        \draw[CliqueEmptyEdge](4)edge[]node[]{}(5);
        \draw[CliqueEdge](1)edge[bend left=30]node[CliqueLabel]
            {\begin{math}0\end{math}}(6);
        \draw[CliqueEmptyEdge](4)edge[]node[]{}(5);
        \draw[CliqueEmptyEdge](5)edge[]node[CliqueLabel]{}(6);
        \draw[CliqueEmptyEdge](6)edge[]node[]{}(7);
        \draw[CliqueEdge](2)
            edge[bend right=30]node[CliqueLabel,near start]
            {\begin{math}0\end{math}}(4);
        \draw[CliqueEdge](3)edge[bend right=30]node[CliqueLabel,near end]
            {\begin{math}0\end{math}}(5);
    \end{tikzpicture}\,,
\end{equation}

\begin{equation} \label{equ:example_involutions_2}
    \begin{tikzpicture}[scale=0.6,Centering]
        \node[CliquePoint](1)at(-0.59,-0.81){};
        \node[CliquePoint](2)at(-0.95,0.31){};
        \node[CliquePoint](3)at(-0.00,1.00){};
        \node[CliquePoint](4)at(0.95,0.31){};
        \node[CliquePoint](5)at(0.59,-0.81){};
        \draw[CliqueEmptyEdge](1)edge[]node[]{}(2);
        \draw[CliqueEmptyEdge](1)edge[]node[]{}(5);
        \draw[CliqueEmptyEdge](2)edge[]node[]{}(3);
        \draw[CliqueEmptyEdge](3)edge[]node[]{}(4);
        \draw[CliqueEmptyEdge](1)edge[]node[]{}(2);
        \draw[CliqueEmptyEdge](4)edge[]node[]{}(5);
        \draw[CliqueEdge](1)edge[bend left=30]node[CliqueLabel]
            {\begin{math}0\end{math}}(4);
    \end{tikzpicture}
    \circ_3
    \begin{tikzpicture}[scale=0.5,Centering]
        \node[CliquePoint](1)at(-0.71,-0.71){};
        \node[CliquePoint](2)at(-0.71,0.71){};
        \node[CliquePoint](3)at(0.71,0.71){};
        \node[CliquePoint](4)at(0.71,-0.71){};
        \draw[CliqueEmptyEdge](1)edge[]node[]{}(2);
        \draw[CliqueEmptyEdge](1)edge[]node[]{}(4);
        \draw[CliqueEmptyEdge](3)edge[]node[]{}(4);
        \draw[CliqueEdge](1)edge[]node[CliqueLabel,near start]
            {\begin{math}0\end{math}}(3);
        \draw[CliqueEmptyEdge](2)edge[]node[]{}(3);
        \draw[CliqueEdge](2)edge[]node[CliqueLabel,near end]
            {\begin{math}0\end{math}}(4);
    \end{tikzpicture}
    = 0.
\end{equation}
\end{subequations}
\end{multicols}
\medbreak

By seeing each solid arc $(x, y)$ of an $\Mca$-clique $\Pfr$ of
$\Deg_1\Dbb_0$ of arity $n$ as the transposition exchanging the letter
$x$ and the letter $y$, we can interpret $\Pfr$ as an involution of
$\mathfrak{S}_{n + 1}$ made of the product of these transpositions.
Hence, $\Deg_1\Dbb_0$ can be seen as an operad on involutions. Under
this point of view, the partial
compositions~\eqref{equ:example_involutions_1}
and~\eqref{equ:example_involutions_2} translate on permutations as
\vspace{-1.75em}
\begin{multicols}{2}
\begin{subequations}
\begin{equation}
    42315 \circ_2 3412 = 6452317,
\end{equation}

\begin{equation}
    42315 \circ_3 3412 = 0.
\end{equation}
\end{subequations}
\end{multicols}
\noindent Equivalently, by the Robinson-Schensted correspondence (see
for instance~\cite{Lot02}), $\Deg_1\Dbb_0$ is an operad on standard Young
tableaux. The dimensions of the operad $\Deg_1\Dbb_0$ begin by
\begin{equation}
    1, 4, 10, 26, 76, 232, 764, 2620,
\end{equation}
and form, except for the first terms, Sequence~\OEIS{A000085}
of~\cite{Slo}.
Moreover, when $\# \Mca = 3$, the dimensions of
$\Deg_1\Mca$ begin by
\begin{equation}
    1, 7, 25, 81, 331, 1303, 5937, 26785,
\end{equation}
and form, except for the first terms, Sequence~\OEIS{A047974}
of~\cite{Slo}.
\medbreak

Besides, any skeleton of an $\Mca$-clique of $\Deg_2\Mca$ can be seen as
a \Def{thunderstorm graph}, {\em i.e.}, a graph where connected
components are cycles or paths. Therefore, $\Deg_2\Mca$ can be seen as
an operad on such colored graphs, where the arcs of the graphs have one
color from the set $\bar{\Mca}$. When $\# \Mca = 2$, the
dimensions of this operad begin by
\begin{equation}
    1, 8, 41, 253, 1858, 15796, 152219, 1638323,
\end{equation}
and form, except for the first terms, Sequence~\OEIS{A136281}
of~\cite{Slo}.
\medbreak

%%%%%%%%%%%%%%%%%%%%%%%%%%%%%%%%%%%%%%%%%%%%%%%%%%%%%%%%%%%%%%%%%%%%%%%%
\subsubsection{Nesting-free cliques}%
\label{subsubsec:quotient_Cli_M_Inf}
Let $\Rel_{\Nes\Mca}$ be the subspace of $\Cli\Mca$ generated by all
$\Mca$-cliques that are not nesting-free.  As a quotient of graded
vector spaces,
\begin{equation}
    \Nes\Mca := \Cli\Mca / \Rel_{\Nes\Mca}
\end{equation}
is the linear span of all nesting-free $\Mca$-cliques.
\medbreak

\begin{Proposition} \label{prop:quotient_Cli_M_nesting_free}
    Let $\Mca$ be a unitary magma without nontrivial unit divisors.
    Then, the space $\Nes\Mca$ is a quotient operad of $\Cli\Mca$.
\end{Proposition}
\begin{proof}
    Since $\Mca$ has no nontrivial unit divisors, for any $\Mca$-cliques
    $\Pfr$ and $\Qfr$ of $\Cli\Mca$, each solid arc of $\Pfr$ (resp.
    $\Qfr$) gives rise to a solid arc in $\Pfr \circ_i \Qfr$, for any
    valid integer $i$. For this reason, if $\Pfr$ is an $\Mca$-clique of
    $\Rel_{\Nes\Mca}$, $\Pfr$ is not nesting-free and each
    $\Mca$-clique obtained by a partial composition involving $\Pfr$ and
    other $\Mca$-cliques is still not nesting-free and thus, belongs
    to $\Rel_{\Nes\Mca}$. This proves that $\Rel_{\Nes\Mca}$ is an
    operad ideal of $\Cli\Mca$ and implies the statement of the
    proposition.
\end{proof}
\medbreak

For instance, in the operad $\Nes\Dbb_2$,
\vspace{-1.75em}
\begin{multicols}{2}
\begin{subequations}
\begin{equation}
    \begin{tikzpicture}[scale=0.6,Centering]
        \node[CliquePoint](1)at(-0.59,-0.81){};
        \node[CliquePoint](2)at(-0.95,0.31){};
        \node[CliquePoint](3)at(-0.00,1.00){};
        \node[CliquePoint](4)at(0.95,0.31){};
        \node[CliquePoint](5)at(0.59,-0.81){};
        \draw[CliqueEmptyEdge](1)edge[]node[]{}(2);
        \draw[CliqueEmptyEdge](1)edge[]node[]{}(5);
        \draw[CliqueEmptyEdge](2)edge[]node[]{}(3);
        \draw[CliqueEmptyEdge](3)edge[]node[]{}(4);
        \draw[CliqueEmptyEdge](4)edge[]node[]{}(5);
        \draw[CliqueEdge](1)
            edge[bend right=30]node[CliqueLabel,near start]
            {\begin{math}0\end{math}}(3);
        \draw[CliqueEdge](2)
            edge[bend right=30]node[CliqueLabel,near end]
            {\begin{math}\Dtt_1\end{math}}(4);
    \end{tikzpicture}
    \circ_4
    \begin{tikzpicture}[scale=0.5,Centering]
        \node[CliquePoint](1)at(-0.71,-0.71){};
        \node[CliquePoint](2)at(-0.71,0.71){};
        \node[CliquePoint](3)at(0.71,0.71){};
        \node[CliquePoint](4)at(0.71,-0.71){};
        \draw[CliqueEmptyEdge](1)edge[]node[]{}(4);
        \draw[CliqueEmptyEdge](3)edge[]node[]{}(4);
        \draw[CliqueEdge](1)edge[]node[CliqueLabel]
            {\begin{math}\Dtt_1\end{math}}(2);
        \draw[CliqueEdge](2)edge[]node[CliqueLabel]
            {\begin{math}0\end{math}}(3);
    \end{tikzpicture}
    =
    \begin{tikzpicture}[scale=0.8,Centering]
        \node[CliquePoint](1)at(-0.43,-0.90){};
        \node[CliquePoint](2)at(-0.97,-0.22){};
        \node[CliquePoint](3)at(-0.78,0.62){};
        \node[CliquePoint](4)at(-0.00,1.00){};
        \node[CliquePoint](5)at(0.78,0.62){};
        \node[CliquePoint](6)at(0.97,-0.22){};
        \node[CliquePoint](7)at(0.43,-0.90){};
        \draw[CliqueEmptyEdge](1)edge[]node[]{}(2);
        \draw[CliqueEmptyEdge](1)edge[]node[]{}(7);
        \draw[CliqueEmptyEdge](2)edge[]node[]{}(3);
        \draw[CliqueEmptyEdge](3)edge[]node[]{}(4);
        \draw[CliqueEmptyEdge](6)edge[]node[]{}(7);
        \draw[CliqueEdge](4)edge[]node[CliqueLabel]
            {\begin{math}\Dtt_1\end{math}}(5);
        \draw[CliqueEdge](5)edge[]node[CliqueLabel]
            {\begin{math}0\end{math}}(6);
        \draw[CliqueEdge](1)
            edge[bend right=30]node[CliqueLabel,near start]
            {\begin{math}0\end{math}}(3);
        \draw[CliqueEdge](2)edge[bend right=30]node[CliqueLabel,near end]
            {\begin{math}\Dtt_1\end{math}}(4);
    \end{tikzpicture}\,,
\end{equation}

\begin{equation}
    \begin{tikzpicture}[scale=0.6,Centering]
        \node[CliquePoint](1)at(-0.59,-0.81){};
        \node[CliquePoint](2)at(-0.95,0.31){};
        \node[CliquePoint](3)at(-0.00,1.00){};
        \node[CliquePoint](4)at(0.95,0.31){};
        \node[CliquePoint](5)at(0.59,-0.81){};
        \draw[CliqueEmptyEdge](1)edge[]node[]{}(2);
        \draw[CliqueEmptyEdge](1)edge[]node[]{}(5);
        \draw[CliqueEmptyEdge](2)edge[]node[]{}(3);
        \draw[CliqueEmptyEdge](3)edge[]node[]{}(4);
        \draw[CliqueEmptyEdge](4)edge[]node[]{}(5);
        \draw[CliqueEdge](1)
            edge[bend right=30]node[CliqueLabel,near start]
            {\begin{math}0\end{math}}(3);
        \draw[CliqueEdge](2)edge[bend right=30]node[CliqueLabel,near end]
            {\begin{math}\Dtt_1\end{math}}(4);
    \end{tikzpicture}
    \circ_3
    \begin{tikzpicture}[scale=0.5,Centering]
        \node[CliquePoint](1)at(-0.71,-0.71){};
        \node[CliquePoint](2)at(-0.71,0.71){};
        \node[CliquePoint](3)at(0.71,0.71){};
        \node[CliquePoint](4)at(0.71,-0.71){};
        \draw[CliqueEmptyEdge](1)edge[]node[]{}(4);
        \draw[CliqueEmptyEdge](3)edge[]node[]{}(4);
        \draw[CliqueEdge](1)edge[]node[CliqueLabel]
            {\begin{math}\Dtt_2\end{math}}(2);
        \draw[CliqueEdge](2)edge[]node[CliqueLabel]
            {\begin{math}0\end{math}}(3);
    \end{tikzpicture}
    = 0.
\end{equation}
\end{subequations}
\end{multicols}
\medbreak

Observe that in the same way as considering $\Mca$-cliques of crossing
numbers $k$ or less leads to quotients $\Cro_k\Mca$ of
$\Cli\Mca$ (see Section~\ref{subsubsec:quotient_Cli_M_crossings}), it is
possible to define analogous quotients $\Nes_k\Mca$ spanned by
$\Mca$-cliques having solid arcs that nest at most $k$ other ones.
\medbreak

Recall that a \Def{Dyck path} of \Def{size} $n$ is a word $u$ on
$\{\Att, \Btt\}$ of length $2n$ such that $|u|_\Att = |u|_\Btt$ and,
for each prefix $v$ of $u$, $|v|_\Att \geq |v|_\Btt$.
\medbreak

\begin{Lemma} \label{lem:bijection_Inf_M_Dyck_paths}
    Let $\Mca$ be a finite unitary magma without nontrivial unit
    divisors. For all $n \geq 2$, the set of all $\Mca$-cliques of
    $\Nes\Mca(n)$ is in one-to-one correspondence with the set of all
    Dyck paths of size $n + 1$ wherein letters $\Att$ at even positions
    are colored by~$\bar{\Mca}$. Moreover, there is a correspondence
    between these two sets that sends any $\Mca$-clique of $\Nes\Mca(n)$
    with $k$ solid edges to a Dyck path with exactly $k$ letters $\Att$
    at even positions, for any $0 \leq k \leq n$.
\end{Lemma}
\begin{proof}
    In this proof, we denote by $\Att_c$ the letter $\Att$ of a Dyck
    path colored by $c \in \bar{\Mca}$. Given an $\Mca$-clique $\Pfr$ of
    $\Nes\Mca(n)$, we decorate each vertex $x$ of $\Pfr$ by
    \begin{enumerate}[fullwidth,label={(\it\arabic*)}]
        \item \label{item:bijection_Inf_M_Dyck_paths_1}
        $\Att\Att_c$ if $x$ has one outcoming arc and no incoming arc,
        where $c$ is the label of the outcoming arc from~$x$;
        \item \label{item:bijection_Inf_M_Dyck_paths_2}
        $\Btt\Btt$ if $x$ has no outcoming arc and one incoming arc;
        \item \label{item:bijection_Inf_M_Dyck_paths_3}
        $\Btt\Att_c$ if $x$ has both one outcoming arc and one incoming
        arc, where $c$ is the label of the outcoming arc from~$x$;
        \item \label{item:bijection_Inf_M_Dyck_paths_4}
        $\Att\Btt$ otherwise.
    \end{enumerate}
    Let $\phi$ be the map sending $\Pfr$ to the word obtained by
    concatenating the decorations of the vertices of $\Pfr$ thus
    described, read from $1$ to~$n + 1$.
    \smallbreak

    We show that $\phi$ is a bijection between the two sets of the
    statement of the lemma. First, observe that since $\Pfr$ is
    nesting-free, for each vertex $y$ of $\Pfr$, there is at most one
    incoming arc to $y$ and one outcoming arc from $y$. For this reason,
    for any vertex $y$ of $\Pfr$, the total number of incoming arcs to
    vertices $x \leq y$ of $\Pfr$ is smaller than or equal to the total
    number of outcoming arcs to vertices $x \leq y$ of $\Pfr$, and the
    total number of vertices having an incoming arc is equal to the
    total number of vertices having an outcoming arc in $\Pfr$. Thus, by
    forgetting the colorings of its letters, the word $\phi(\Pfr)$ is
    a Dyck path.
    \smallbreak

    Besides, given a Dyck path $u$ of size $n + 1$ wherein letters
    $\Att$ at even positions are colored by $\bar{\Mca}$, one can build
    a unique $\Mca$-clique $\Pfr$ of $\Nes\Mca(n)$ such that
    $\phi(\Pfr) = u$. Indeed, by reading the letters of $u$ two by two,
    one knows the number of outcoming and incoming arcs for each vertex
    of $\Pfr$. Since $\Pfr$ is nesting-free, there is one unique way
    to connect these vertices by solid diagonals without creating
    nestings of arcs. Moreover,
    by~\ref{item:bijection_Inf_M_Dyck_paths_1},
    \ref{item:bijection_Inf_M_Dyck_paths_2},
    \ref{item:bijection_Inf_M_Dyck_paths_3},
    and~\ref{item:bijection_Inf_M_Dyck_paths_4}, the colors of the
    letters $\Att$ at even positions allow to label the solid arcs of
    $\Pfr$. Hence $\phi$ is a bijection as claimed.
    \smallbreak

    Finally, by definition of $\phi$, we observe that if $\Pfr$ has
    exactly $k$ solid arcs, the Dyck path $\phi(\Pfr)$ has exactly $k$
    occurrences of the letter $\Att$ at even positions. This implies the
    whole statement of the lemma.
\end{proof}
\medbreak

Let $\Nar(n, k)$ be the \Def{Narayana number}~\cite{Nar55} defined
for all $0 \leq k \leq n - 2$ by
\begin{equation}
    \Nar(n, k) := \frac{1}{k + 1} \binom{n - 2}{k} \binom{n - 1}{k}.
\end{equation}
The number of Dyck paths of size $n - 1$ and exactly $k$ occurrences of
the factor $\Att \Btt$ is $\Nar(n, k)$. Equivalently, this is also the
number of binary trees with $n$ leaves and exactly $k$ internal nodes
having an internal node as a left child.
\medbreak

\begin{Proposition} \label{prop:dimensions_Inf_M}
    Let $\Mca$ be a finite unitary magma without nontrivial unit
    divisors. For all $n \geq 2$,
    \begin{equation} \label{equ:dimensions_Inf_M}
        \dim \Nes\Mca(n) =
        \sum_{0 \leq k \leq n}
        (m - 1)^k \; \Nar(n + 2, k),
    \end{equation}
    where $m := \# \Mca$.
\end{Proposition}
\begin{proof}
    It is known from~\cite{Sul98} that the number of Dyck paths of size
    $n + 1$ with $k$ occurrences of the letter $\Att$ at even positions
    is the Narayana number $\Nar(n + 2, k)$. Hence, by using this
    property together with Lemma~\ref{lem:bijection_Inf_M_Dyck_paths},
    we obtain that the number of nesting-free $\Mca$-cliques of size
    $n$ with $k$ solid arcs is $(m - 1)^k \; \Nar(n + 2, k)$. Therefore,
    since a nesting-free $\Mca$-clique of arity $n$ can have at most
    $n$ solid arcs, \eqref{equ:dimensions_Inf_M} holds.
\end{proof}
\medbreak

The skeletons of the $\Mca$-cliques of $\Nes\Mca$ of arities greater
than $1$ are the graphs such that, if $\{x, y\}$ and $\{x', y'\}$ are
two arcs such that $x \leq x' < y' \leq y$, then $x = x'$ and $y = y'$.
Therefore, $\Nes\Mca$ can be seen as an operad on such colored graphs,
where the arcs of the graphs have one color from the set $\bar{\Mca}$.
Equivalently, as Lemma~\ref{lem:bijection_Inf_M_Dyck_paths} shows,
$\Nes\Mca$ can be seen as an operad of Dyck paths where letters $\Att$
at even positions are colored by~$\bar{\Mca}$.
\medbreak

By Proposition~\ref{prop:dimensions_Inf_M}, when $\# \Mca = 2$, the
dimensions of $\Nes\Mca$ begin by
\begin{equation}
    1, 5, 14, 42, 132, 429, 1430, 4862,
\end{equation}
and form, except for the first terms, Sequence~\OEIS{A000108}
of~\cite{Slo}. When $\# \Mca = 3$, the dimensions of
$\Nes\Mca$ begin by
\begin{equation}
    1, 11, 45, 197, 903, 4279, 20793, 103049,
\end{equation}
and form, except for the first terms, Sequence~\OEIS{A001003}
of~\cite{Slo}. When $\# \Mca = 4$, the dimensions of
$\Nes\Mca$ begin by
\begin{equation}
    1, 19, 100, 562, 3304, 20071, 124996, 793774,
\end{equation}
and form, except for the first terms, Sequence~\OEIS{A007564}
of~\cite{Slo}.
\medbreak

%%%%%%%%%%%%%%%%%%%%%%%%%%%%%%%%%%%%%%%%%%%%%%%%%%%%%%%%%%%%%%%%%%%%%%%%
\subsubsection{Acyclic decorated cliques}%
\label{subsubsec:quotient_Cli_M_acyclic}
Let $\Rel_{\Acy\Mca}$ be the subspace of $\Cli\Mca$ generated by all
$\Mca$-cliques that are not acyclic. As a quotient of graded vector
spaces,
\begin{equation}
    \Acy\Mca := \Cli\Mca / \Rel_{\Acy\Mca}
\end{equation}
is the linear span of all acyclic $\Mca$-cliques.
\medbreak

\begin{Proposition} \label{prop:quotient_Cli_M_acyclic}
    Let $\Mca$ be a unitary magma without nontrivial unit divisors.
    Then, the space $\Acy\Mca$ is a quotient operad of $\Cli\Mca$.
\end{Proposition}
\begin{proof}
    Since $\Mca$ has no nontrivial unit divisors, for any $\Mca$-cliques
    $\Pfr$ and $\Qfr$ of $\Cli\Mca$, each solid arc of $\Pfr$ (resp.
    $\Qfr$) gives rise to a solid arc in $\Pfr \circ_i \Qfr$, for any
    valid integer $i$. For this reason, if $\Pfr$ is an $\Mca$-clique of
    $\Rel_{\Acy\Mca}$, $\Pfr$ is not acyclic and each $\Mca$-clique
    obtained by a partial composition involving $\Pfr$ and other
    $\Mca$-cliques is still not acyclic and thus, belongs to
    $\Rel_{\Acy\Mca}$. This proves that $\Rel_{\Acy\Mca}$ is an operad
    ideal of $\Cli\Mca$ and implies the statement of the proposition.
\end{proof}
\medbreak

For instance, in the operad $\Acy\Dbb_2$,
\vspace{-1.75em}
\begin{multicols}{2}
\begin{subequations}
\begin{equation}
    \begin{tikzpicture}[scale=0.7,Centering]
        \node[CliquePoint](1)at(-0.59,-0.81){};
        \node[CliquePoint](2)at(-0.95,0.31){};
        \node[CliquePoint](3)at(-0.00,1.00){};
        \node[CliquePoint](4)at(0.95,0.31){};
        \node[CliquePoint](5)at(0.59,-0.81){};
        \draw[CliqueEmptyEdge](1)edge[]node[]{}(5);
        \draw[CliqueEmptyEdge](2)edge[]node[]{}(3);
        \draw[CliqueEmptyEdge](3)edge[]node[]{}(4);
        \draw[CliqueEmptyEdge](4)edge[]node[]{}(5);
        \draw[CliqueEdge](1)edge[]node[CliqueLabel]
            {\begin{math}0\end{math}}(2);
        \draw[CliqueEdge](2)
            edge[bend left=20]node[CliqueLabel,near start]
            {\begin{math}0\end{math}}(5);
        \draw[CliqueEdge](3)
            edge[bend right=20]node[CliqueLabel,near start]
            {\begin{math}\Dtt_1\end{math}}(5);
        \draw[CliqueEdge](1)edge[]node[CliqueLabel,near end]
            {\begin{math}0\end{math}}(4);
    \end{tikzpicture}
    \circ_1
    \begin{tikzpicture}[scale=0.6,Centering]
        \node[CliquePoint](1)at(-0.71,-0.71){};
        \node[CliquePoint](2)at(-0.71,0.71){};
        \node[CliquePoint](3)at(0.71,0.71){};
        \node[CliquePoint](4)at(0.71,-0.71){};
        \draw[CliqueEmptyEdge](1)edge[]node[]{}(2);
        \draw[CliqueEmptyEdge](2)edge[]node[]{}(3);
        \draw[CliqueEmptyEdge](3)edge[]node[]{}(4);
        \draw[CliqueEdge](1)edge[]node[CliqueLabel]
            {\begin{math}\Dtt_1\end{math}}(3);
        \draw[CliqueEdge](1)edge[]node[CliqueLabel]
            {\begin{math}\Dtt_1\end{math}}(4);
    \end{tikzpicture}
    =
    \begin{tikzpicture}[scale=0.9,Centering]
        \node[CliquePoint](1)at(-0.43,-0.90){};
        \node[CliquePoint](2)at(-0.97,-0.22){};
        \node[CliquePoint](3)at(-0.78,0.62){};
        \node[CliquePoint](4)at(-0.00,1.00){};
        \node[CliquePoint](5)at(0.78,0.62){};
        \node[CliquePoint](6)at(0.97,-0.22){};
        \node[CliquePoint](7)at(0.43,-0.90){};
        \draw[CliqueEmptyEdge](1)edge[]node[]{}(2);
        \draw[CliqueEmptyEdge](1)edge[]node[]{}(7);
        \draw[CliqueEmptyEdge](2)edge[]node[]{}(3);
        \draw[CliqueEmptyEdge](3)edge[]node[]{}(4);
        \draw[CliqueEmptyEdge](4)edge[]node[]{}(5);
        \draw[CliqueEmptyEdge](5)edge[]node[]{}(6);
        \draw[CliqueEmptyEdge](6)edge[]node[]{}(7);
        \draw[CliqueEdge](1)edge[bend right=20]node[CliqueLabel,near end]
            {\begin{math}\Dtt_1\end{math}}(3);
        \draw[CliqueEdge](1)edge[]node[CliqueLabel]
            {\begin{math}0\end{math}}(4);
        \draw[CliqueEdge](1)
            edge[bend left=30]node[CliqueLabel,near start]
            {\begin{math}0\end{math}}(6);
        \draw[CliqueEdge](4)edge[]node[CliqueLabel,near start]
            {\begin{math}0\end{math}}(7);
        \draw[CliqueEdge](5)
            edge[bend right=20]node[CliqueLabel,near start]
            {\begin{math}\Dtt_1\end{math}}(7);
    \end{tikzpicture}\,,
\end{equation}

\begin{equation}
    \begin{tikzpicture}[scale=0.7,Centering]
        \node[CliquePoint](1)at(-0.59,-0.81){};
        \node[CliquePoint](2)at(-0.95,0.31){};
        \node[CliquePoint](3)at(-0.00,1.00){};
        \node[CliquePoint](4)at(0.95,0.31){};
        \node[CliquePoint](5)at(0.59,-0.81){};
        \draw[CliqueEmptyEdge](1)edge[]node[]{}(5);
        \draw[CliqueEmptyEdge](2)edge[]node[]{}(3);
        \draw[CliqueEmptyEdge](3)edge[]node[]{}(4);
        \draw[CliqueEmptyEdge](4)edge[]node[]{}(5);
        \draw[CliqueEdge](1)edge[]node[CliqueLabel]
            {\begin{math}0\end{math}}(2);
        \draw[CliqueEdge](2)
            edge[bend left=20]node[CliqueLabel,near start]
            {\begin{math}0\end{math}}(5);
        \draw[CliqueEdge](3)
            edge[bend right=20]node[CliqueLabel,near start]
            {\begin{math}\Dtt_1\end{math}}(5);
        \draw[CliqueEdge](1)edge[]node[CliqueLabel,near end]
            {\begin{math}0\end{math}}(4);
    \end{tikzpicture}
    \circ_3
    \begin{tikzpicture}[scale=0.6,Centering]
        \node[CliquePoint](1)at(-0.71,-0.71){};
        \node[CliquePoint](2)at(-0.71,0.71){};
        \node[CliquePoint](3)at(0.71,0.71){};
        \node[CliquePoint](4)at(0.71,-0.71){};
        \draw[CliqueEmptyEdge](1)edge[]node[]{}(2);
        \draw[CliqueEmptyEdge](2)edge[]node[]{}(3);
        \draw[CliqueEmptyEdge](3)edge[]node[]{}(4);
        \draw[CliqueEdge](1)edge[]node[CliqueLabel]
            {\begin{math}\Dtt_2\end{math}}(3);
        \draw[CliqueEdge](1)edge[]node[CliqueLabel]
            {\begin{math}\Dtt_1\end{math}}(4);
    \end{tikzpicture}
    = 0.
\end{equation}
\end{subequations}
\end{multicols}
\medbreak

The skeletons of the $\Mca$-cliques of $\Acy\Mca$ of arities greater
than $1$ are acyclic graphs or equivalently, forests of non-rooted trees.
Therefore, $\Acy\Mca$ can be seen as an operad on colored forests of
trees, where the edges of the trees of the forests have one color from
the set $\bar{\Mca}$. When $\# \Mca = 2$, the dimensions of $\Acy\Mca$
begin by
\begin{equation}
    1, 7, 38, 291, 2932, 36961, 561948, 10026505,
\end{equation}
and form, except for the first terms, Sequence~\OEIS{A001858}
of~\cite{Slo}.
\medbreak

%%%%%%%%%%%%%%%%%%%%%%%%%%%%%%%%%%%%%%%%%%%%%%%%%%%%%%%%%%%%%%%%%%%%%%%%
%%%%%%%%%%%%%%%%%%%%%%%%%%%%%%%%%%%%%%%%%%%%%%%%%%%%%%%%%%%%%%%%%%%%%%%%
\subsection{Secondary substructures}%
\label{subsec:secondary_substructures}
Some more substructures of $\Cli\Mca$ are constructed and briefly
studied here. They are constructed by mixing some of the constructions
of the seven main substructures of $\Cli\Mca$ defined in
Section~\ref{subsec:main_substructures} in the following sense.
\medbreak

For any operad $\Oca$ and operad ideals $\Rel_1$ and $\Rel_2$ of $\Oca$,
the space $\Rel_1 + \Rel_2$ is still an operad ideal of $\Oca$, and
$\Oca / \left(\Rel_1 + \Rel_2\right)$ is a quotient of both $\Oca / \Rel_1$ and
$\Oca / \Rel_2$. Moreover, if $\Oca'$ is a suboperad of $\Oca$ and
$\Rel$ is an operad ideal of $\Oca$, the space $\Rel \cap \Oca'$ is an
operad ideal of $\Oca'$, and $\Oca' / \left(\Rel \cap \Oca'\right)$ is a quotient of
$\Oca'$ and a suboperad of $\Oca / \Rel$. For these reasons
(straightforwardly provable), we can combine the constructions of the
previous section to build plenty new suboperads and quotients
of~$\Cli\Mca$ (see Table~\ref{tab:secondary_substructures}).
\begin{table}[ht]
    \centering
    \begin{tabular}{c|c|c}
        Operad & Objects & Ideal of $\Cli\Mca$ \\ \hline \hline
        $\WNC\Mca$ & White noncrossing cliques &
            $\Rel_{\Cro_0\Mca} \cap \Whi\Mca$ \\
        $\Paths\Mca$ & Forests of paths &
            $\Rel_{\Deg_2\Mca} + \Rel_{\Acy\Mca}$ \\
        $\Forests\Mca$ & Forests &
            $\Rel_{\Cro_0\Mca} + \Rel_{\Acy\Mca}$ \\
        $\Motzkin\Mca$ & Motzkin configurations &
            $\Rel_{\Cro_0\Mca} + \Rel_{\Deg_1\Mca}$ \\
        $\Diss\Mca$ & Dissections of polygons &
            $\left(\Rel_{\Cro_0\Mca} + \Rel_{\Deg_1\Mca}\right)
                \cap \Whi\Mca$ \\
        $\Luc\Mca$ & Lucas configurations &
            $\Rel_{\Bub\Mca} + \Rel_{\Deg_1\Mca}$
    \end{tabular}
    \smallbreak

    \caption{\footnotesize
    Operads obtained as quotients of $\Cli\Mca$ by mixing certain ideals
    of~$\Cli\Mca$. All these operads depend on a unitary magma $\Mca$
    which has, in some cases, to satisfy some precise conditions.}
    \label{tab:secondary_substructures}
\end{table}
\medbreak

%%%%%%%%%%%%%%%%%%%%%%%%%%%%%%%%%%%%%%%%%%%%%%%%%%%%%%%%%%%%%%%%%%%%%%%%
\subsubsection{Colored white noncrossing configurations}
When $\Mca$ is a unitary magma, let
\begin{equation}
    \WNC\Mca := \Whi\Mca / \Rel_{\Cro_0\Mca} \cap \Whi\Mca.
\end{equation}
The $\Mca$-cliques of $\WNC\Mca$ are white noncrossing $\Mca$-cliques.
\medbreak

When $\# \Mca = 2$, the dimensions of $\WNC\Mca$ begin by
\begin{equation}
    1, 1, 3, 11, 45, 197, 903, 4279,
\end{equation}
and form Sequence~\OEIS{A001003} of~\cite{Slo}. When $\# \Mca = 3$, the
dimensions of $\WNC\Mca$ begin by
\begin{equation}
    1, 1, 5, 31, 215, 1597, 12425, 99955,
\end{equation}
and form Sequence~\OEIS{A269730} of~\cite{Slo}. Observe that these
dimensions are shifted versions the ones of the
$\gamma$-polytridendriform operads $\TDendr_\gamma$~\cite{Gir16} with
$\gamma := \# \Mca - 1$.
\medbreak

%%%%%%%%%%%%%%%%%%%%%%%%%%%%%%%%%%%%%%%%%%%%%%%%%%%%%%%%%%%%%%%%%%%%%%%%
\subsubsection{Colored forests of paths}
When $\Mca$ is a unitary magma without nontrivial unit divisors, let
\begin{equation}
    \Paths\Mca := \Cli\Mca / \left(\Rel_{\Deg_2\Mca} + \Rel_{\Acy\Mca}\right).
\end{equation}
The skeletons of the $\Mca$-cliques of $\Paths\Mca$ are forests of
non-rooted trees that are paths. Therefore, $\Paths\Mca$ can be seen as
an operad on such colored graphs, where the arcs of the graphs have one
color from the set~$\bar{\Mca}$.
\medbreak

When $\# \Mca = 2$,
the dimensions of $\Paths\Mca$ begin by
\begin{equation}
    1, 7, 34, 206, 1486, 12412, 117692, 1248004,
\end{equation}
an form, except for the first terms, Sequence~\OEIS{A011800}
of~\cite{Slo}.
\medbreak

%%%%%%%%%%%%%%%%%%%%%%%%%%%%%%%%%%%%%%%%%%%%%%%%%%%%%%%%%%%%%%%%%%%%%%%%
\subsubsection{Colored forests}
When $\Mca$ is a unitary magma without nontrivial unit divisors, let
\begin{equation}
    \Forests\Mca := \Cli\Mca / \left(\Rel_{\Cro_0\Mca} + \Rel_{\Acy\Mca}\right).
\end{equation}
The skeletons of the $\Mca$-cliques of $\Forests\Mca$ are forests of
rooted trees having no arcs $\{x, y\}$ and $\{x', y'\}$ satisfying
$x < x' < y < y'$. Therefore, $\Forests\Mca$ can be seen as an operad
on such colored forests, where the edges of the forests have one color
from the set $\bar{\Mca}$. When $\# \Mca = 2$, the dimensions of
$\Forests\Mca$ begin by
\begin{equation}
    1, 7, 33, 81, 1083, 6854, 45111, 305629,
\end{equation}
and form, except for the first terms, Sequence~\OEIS{A054727},
of~\cite{Slo}.
\medbreak

%%%%%%%%%%%%%%%%%%%%%%%%%%%%%%%%%%%%%%%%%%%%%%%%%%%%%%%%%%%%%%%%%%%%%%%%
\subsubsection{Colored Motzkin configurations}%
\label{subsubsec:Motzkin_configurations}
When $\Mca$ is a unitary magma without nontrivial unit divisors, let
\begin{equation}
    \Motzkin\Mca := \Cli\Mca / \left(\Rel_{\Cro_0\Mca} + \Rel_{\Deg_1\Mca}\right).
\end{equation}
The skeletons of the $\Mca$-cliques of $\Motzkin\Mca$ are configurations
of non-intersecting chords on a circle. Equivalently, these objects are
graphs of involutions (see
Section~\ref{subsubsec:quotient_Cli_M_degrees}) having no arcs
$\{x, y\}$ and $\{x', y'\}$ satisfying $x < x' < y < y'$. These objects
are enumerated by Motzkin numbers~\cite{Mot48}. Therefore,
$\Motzkin\Mca$ can be seen as an operad on such colored graphs, where
the arcs of the graphs have one color from the set $\bar{\Mca}$. When
$\# \Mca = 2$, the dimensions of $\Motzkin\Mca$ begin by
\begin{equation}
    1, 4, 9, 21, 51, 127, 323, 835,
\end{equation}
and form, except for the first terms, Sequence~\OEIS{A001006},
of~\cite{Slo}.
\medbreak

%%%%%%%%%%%%%%%%%%%%%%%%%%%%%%%%%%%%%%%%%%%%%%%%%%%%%%%%%%%%%%%%%%%%%%%%
\subsubsection{Colored dissections of polygons}
When $\Mca$ is a unitary magma without nontrivial unit divisors, let
\begin{equation}
    \Diss\Mca :=
    \Whi\Mca / \left(\left(\Rel_{\Cro_0\Mca} + \Rel_{\Deg_1\Mca}\right)
        \cap \Whi\Mca\right).
\end{equation}
The skeletons of the $\Mca$-cliques of $\Diss\Mca$ are \Def{strict
dissections of polygons}, that are graphs of Motzkin configurations
with no arcs of the form $\{x, x + 1\}$ or $\{1, n + 1\}$, where $n + 1$
is the number of vertices of the graphs. Therefore, $\Diss\Mca$ can be
seen as an operad on such colored graphs, where the arcs of the graphs
have one color from the set $\bar{\Mca}$. When $\# \Mca = 2$, the
dimensions of $\Diss\Mca$ begin by
\begin{equation}
    1, 1, 3, 6, 13, 29, 65, 148,
\end{equation}
and form, except for the first terms, Sequence~\OEIS{A093128}
of~\cite{Slo}.
\medbreak

%%%%%%%%%%%%%%%%%%%%%%%%%%%%%%%%%%%%%%%%%%%%%%%%%%%%%%%%%%%%%%%%%%%%%%%%
\subsubsection{Colored Lucas configurations}
When $\Mca$ is a unitary magma without nontrivial unit divisors, let
\begin{equation}
    \Luc\Mca := \Cli\Mca / \left(\Rel_{\Bub\Mca} + \Rel_{\Deg_1\Mca}\right).
\end{equation}
The skeletons of the $\Mca$-cliques of $\Luc\Mca$ are graphs such that
all vertices are of degree at most $1$ and all arcs are of the form
$\{x, x + 1\}$ or $\{1, n + 1\}$, where $n + 1$ is the number of
vertices of the graphs. Therefore, $\Luc\Mca$ can be seen as an operad
on such colored graphs, where the arcs of the graphs have one color
from the set $\bar{\Mca}$. When $\# \Mca = 2$, the dimensions of
$\Luc\Mca$ begin by
\begin{equation}
    1, 4, 7, 11, 18, 29, 47, 76,
\end{equation}
and form, except for the first terms, Sequence~\OEIS{A000032}
of~\cite{Slo}.
\medbreak

%%%%%%%%%%%%%%%%%%%%%%%%%%%%%%%%%%%%%%%%%%%%%%%%%%%%%%%%%%%%%%%%%%%%%%%%
%%%%%%%%%%%%%%%%%%%%%%%%%%%%%%%%%%%%%%%%%%%%%%%%%%%%%%%%%%%%%%%%%%%%%%%%
\subsection{Relations between substructures}
The suboperads and quotients of $\Cli\Mca$ constructed in
Sections~\ref{subsec:main_substructures}
and~\ref{subsec:secondary_substructures} are linked by injective or
surjective operad morphisms. To establish these, we begin with the
following lemma.
\medbreak

\begin{Lemma} \label{lem:inclusion_families_cliques}
    Let $\Mca$ be a unitary magma. Then,
    \begin{enumerate}[fullwidth,label={(\it\roman*)}]
        \item \label{item:inclusion_families_cliques_1}
        the space $\Rel_{\Acy\Mca}$ is a subspace of $\Rel_{\Deg_1\Mca}$;
        \item \label{item:inclusion_families_cliques_2}
        the spaces $\Rel_{\Nes\Mca}$ and $\Rel_{\Bub\Mca}$ are subspaces
        of $\Rel_{\Deg_0\Mca}$;
        \item \label{item:inclusion_families_cliques_3}
        the spaces $\Rel_{\Cro_0\Mca}$ and $\Rel_{\Deg_2\Mca}$ are
        subspaces of $\Rel_{\Bub\Mca}$;
        \item \label{item:inclusion_families_cliques_4}
        the spaces $\Rel_{\Deg_2\Mca}$ and $\Rel_{\Acy\Mca}$ are
        subspaces of $\Rel_{\Nes\Mca}$.
    \end{enumerate}
\end{Lemma}
\begin{proof}
    All the spaces appearing in the statement of the lemma are subspaces
    of $\Cli\Mca$ generated by some subfamilies of $\Mca$-cliques.
    Therefore, to prove the assertions of the lemma, we shall prove
    inclusions of adequate subfamilies of such objects.
    \smallbreak

    If $\Pfr$ is an $\Mca$-clique of $\Rel_{\Acy\Mca}$, by definition,
    $\Pfr$ has a cycle formed by solid arcs. Hence, $\Pfr$ has in
    particular a solid arc and a vertex of degree $2$ or more. For this
    reason, since $\Rel_{\Deg_1\Mca}$ is the linear span of all
    $\Mca$-cliques of degree $2$ or more, $\Pfr$ is in
    $\Rel_{\Deg_1\Mca}$. This
    implies~\ref{item:inclusion_families_cliques_1}.
    \smallbreak

    If $\Pfr$ is an $\Mca$-clique of $\Rel_{\Nes\Mca}$ or
    $\Rel_{\Bub\Mca}$, by definition, $\Pfr$ has in particular a solid
    arc. Hence, since $\Rel_{\Deg_0\Mca}$ is the linear span of all
    $\Mca$-cliques with at least one vertex with a positive degree,
    $\Pfr$ is in $\Rel_{\Deg_0\Mca}$. This
    implies~\ref{item:inclusion_families_cliques_2}.
    \smallbreak

    If $\Pfr$ is an $\Mca$-clique of $\Rel_{\Cro_0\Mca}$ or
    $\Rel_{\Deg_2\Mca}$, $\Pfr$ has in particular a solid diagonal.
    Indeed, when $\Pfr$ is in $\Rel_{\Cro_0\Mca}$ this property is
    immediate. When $\Pfr$ is in $\Rel_{\Deg_2\Mca}$, since $\Pfr$ has a
    vertex $x$ of degree $3$ or more, the skeleton of $\Pfr$ has three
    arcs $\{x, y_1\}$, $\{x, y_2\}$, and $\{x, y_3\}$ with
    $y_i \ne x - 1$, $y_i \ne x + 1$, and $y_i \ne |\Pfr| + 1$ for at
    least one $i \in [3]$, so that the arc
    $(\min\{x, y_i\}, \max\{x, y_i\})$ is a solid diagonal of $\Pfr$. For
    this reason, since $\Rel_{\Bub\Mca}$ is the linear span of all
    $\Mca$-cliques with at least one solid diagonal, $\Pfr$ is in
    $\Rel_{\Bub\Mca}$. This
    implies~\ref{item:inclusion_families_cliques_3}.
    \smallbreak

    If $\Pfr$ is an $\Mca$-clique of $\Rel_{\Deg_2\Mca}$ or
    $\Rel_{\Acy\Mca}$, $\Pfr$ has in particular a solid arc nested in
    another one. Indeed, when $\Pfr$ is in $\Rel_{\Deg_2\Mca}$, since
    $\Pfr$ has a vertex $x$ of a degree $3$ or more, the skeleton
    of $\Pfr$ has three arcs $\{x, y_1\}$, $\{x, y_2\}$, and
    $\{x, y_3\}$. One can check that for all relative orders between
    the vertices $x$, $y_1$, $y_2$, and $y_3$, one of these arcs
    is nested in another one, so that $\Pfr$ is not nesting-free. When
    $\Pfr$ is in $\Rel_{\Acy\Mca}$, $\Pfr$ contains a cycle formed by
    solid arcs. Let $x_1$, $x_2$, \dots, $x_k$, $k \geq 3$, be the
    vertices of $\Pfr$ that form this cycle. We can assume without loss
    of generality that $x_1 \leq x_i$ for all $i \in [k]$ and thus, that
    $(x_1, x_2)$ and $(x_1, x_k)$ are solid arcs of $\Pfr$ being part of
    the cycle. Then, when $x_2 < x_k$, since
    $x_1 \leq x_1 < x_2 \leq x_k$, the arc $(x_1, x_2)$ is nested in
    $(x_1, x_k)$. Otherwise, $x_k < x_2$, and since
    $x_1 \leq x_1 < x_k \leq x_2$, the arc $(x_1, x_k)$ is nested in
    $(x_1, x_2)$. For these reasons, since $\Rel_{\Nes\Mca}$ is the
    linear span of all $\Mca$-cliques that are non nesting-free,
    $\Pfr$ is in $\Rel_{\Nes\Mca}$. This
    implies~\ref{item:inclusion_families_cliques_4}.
\end{proof}
\medbreak

%%%%%%%%%%%%%%%%%%%%%%%%%%%%%%%%%%%%%%%%%%%%%%%%%%%%%%%%%%%%%%%%%%%%%%%%
\subsubsection{Relations between the main substructures}
Here we list and explain the morphisms between the main substructures of
$\Cli\Mca$. First, Lemma~\ref{lem:inclusion_families_cliques} implies
that there are surjective operad morphisms from $\Acy\Mca$ to
$\Deg_1\Mca$, from $\Nes\Mca$ to $\Deg_0\Mca$, from $\Bub\Mca$ to
$\Deg_0\Mca$, from $\Cro_0\Mca$ to $\Bub\Mca$, from $\Deg_2\Mca$ to
$\Bub\Mca$, from $\Deg_2\Mca$ to $\Nes\Mca$, and from $\Acy\Mca$
to~$\Nes\Mca$. Second, when $B$, $E$, and $D$ are subsets of $\Mca$ such
that $\Unit_\Mca \in B$, $\Unit_\Mca \in E$, and $E \Op B \subseteq D$,
$\Whi\Mca$ is a suboperad of $\Lab_{B,E,D}\Mca$. Finally, there is a
surjective operad morphism from $\Whi\Mca$ to the associative operad
$\As$ sending any $\Mca$-clique $\Pfr$ of $\Whi\Mca$ to the unique basis
element of $\As$ of the same arity as the one of~$\Pfr$. The relations
between the main suboperads and quotients of $\Cli\Mca$ built here are
summarized in the diagram of operad morphisms of
Figure~\ref{fig:diagram_main_operads}.
\begin{figure}[ht]
    \centering
    \scalebox{.73}{
    \begin{tikzpicture}[xscale=1.2,yscale=1.1,Centering]
        \node[text=Col2](CliM)at(8,10)
            {\begin{math}\Cli\Mca\end{math}};
        \node[text=Col4](AcyM)at(4,8)
            {\begin{math}\Acy\Mca\end{math}};
        \node[text=Col4](DegkM)at(6,8)
            {\begin{math}\Deg_k\Mca\end{math}};
        \node[text=Col1](CrokM)at(10,8)
            {\begin{math}\Cro_k\Mca\end{math}};
        \node[text=Col1](LabM)at(12,8)
            {\begin{math}\Lab_{B, E, D}\Mca\end{math}};
        \node[text=Col4](Deg2M)at(6,6)
            {\begin{math}\Deg_2\Mca\end{math}};
        \node[text=Col1](Cro0M)at(10,6)
            {\begin{math}\Cro_0\Mca\end{math}};
        \node[text=Col4](NesM)at(4,4)
            {\begin{math}\Nes\Mca\end{math}};
        \node[text=Col4](Deg1M)at(6,4)
            {\begin{math}\Deg_1\Mca\end{math}};
        \node[text=Col4](BubM)at(8,4)
            {\begin{math}\Bub\Mca\end{math}};
        \node[text=Col1](WhiM)at(12,4)
            {\begin{math}\Whi\Mca\end{math}};
        \node[text=Col4](Deg0M)at(8,2)
            {\begin{math}\Deg_0\Mca\end{math}};
        \draw[Surjection](CliM)--(AcyM);
        \draw[Surjection](CliM)--(DegkM);
        \draw[Surjection](CliM)to[bend right=15](CrokM);
        \draw[Injection](CrokM)to[bend right=15](CliM);
        \draw[Injection](LabM)--(CliM);
        \draw[Surjection](AcyM)--(NesM);
        \draw[Surjection](AcyM)--(Deg1M);
        \draw[Surjection](DegkM)--(Deg2M);
        \draw[Surjection](CrokM)to[bend right=15](Cro0M);
        \draw[Injection](Cro0M)to[bend right=15](CrokM);
        \draw[Injection](WhiM)--(LabM);
        \draw[Surjection](Deg2M)--(NesM);
        \draw[Surjection](Deg2M)--(Deg1M);
        \draw[Surjection](Deg2M)--(BubM);
        \draw[Surjection](Cro0M)--(BubM);
        \draw[Surjection](NesM)--(Deg0M);
        \draw[Surjection](Deg1M)--(Deg0M);
        \draw[Surjection](BubM)--(Deg0M);
        \draw[Surjection](WhiM)--(Deg0M);
    \end{tikzpicture}}
    \vspace{-.5em}
    \caption{\footnotesize
    The diagram of the main suboperads and quotients of $\Cli\Mca$.
    Arrows~$\rightarrowtail$ (resp.~$\twoheadrightarrow$) are injective
    (resp. surjective) operad morphisms. Here, $\Mca$ is a unitary magma
    without nontrivial unit divisors, $k$ is a positive integer, and
    $B$, $E$, and $D$ are subsets of $\Mca$ such that
    $\Unit_\Mca \in B$, $\Unit_\Mca \in E$, and $E \Op B \subseteq D$.}
    \label{fig:diagram_main_operads}
\end{figure}
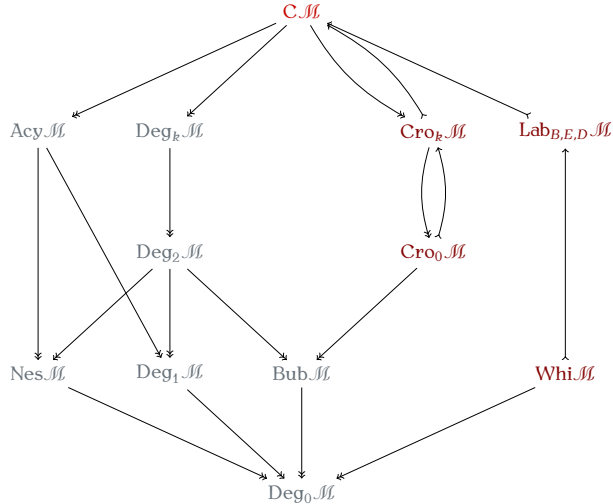
\medbreak

%%%%%%%%%%%%%%%%%%%%%%%%%%%%%%%%%%%%%%%%%%%%%%%%%%%%%%%%%%%%%%%%%%%%%%%%
\subsubsection{Relations between the secondary and main substructures}
Here we list and explain the morphisms between the secondary and main
substructures of $\Cli\Mca$. First, immediately from their definitions,
$\WNC\Mca$ is a suboperad of $\Cro_0\Mca$ and a quotient of~$\Whi\Mca$,
$\Paths\Mca$ is both a quotient of $\Deg_2\Mca$ and $\Acy\Mca$,
$\Forests\Mca$ is both a quotient of $\Cro_0\Mca$ and $\Acy\Mca$,
$\Motzkin\Mca$ is both a quotient of $\Cro_0\Mca$ and $\Deg_1\Mca$,
$\Diss\Mca$ is a suboperad of $\Motzkin\Mca$ and a quotient of
$\WNC\Mca$, and $\Luc\Mca$ is both a quotient of $\Bub\Mca$ and
$\Deg_1\Mca$. Moreover, since by
Lemma~\ref{lem:inclusion_families_cliques}, $\Rel_{\Acy\Mca}$ is a
subspace of $\Rel_{\Deg_1\Mca}$, $\Rel_{\Deg_2\Mca}$ and
$\Rel_{\Acy\Mca}$ are subspaces of $\Rel_{\Nes\Mca}$, and
$\Rel_{\Cro_0\Mca}$ is a subspace of $\Rel_{\Bub\Mca}$, we respectively
have that $\Rel_{\Deg_2\Mca} + \Rel_{\Acy\Mca}$ is a subspace of both
$\Rel_{\Deg_1\Mca}$ and $\Rel_{\Nes\Mca}$,
$\Rel_{\Cro_0\Mca} + \Rel_{\Acy\Mca}$ is a subspace of
$\Rel_{\Cro_0\Mca} + \Rel_{\Deg_1\Mca}$, and
$\Rel_{\Cro_0\Mca} + \Rel_{\Deg_1\Mca}$ is a subspace of
$\Rel_{\Bub\Mca} + \Rel_{\Deg_1\Mca}$. For these reasons, there are
surjective operad morphisms from $\Paths\Mca$ to $\Deg_1\Mca$, from
$\Paths\Mca$ to~$\Nes\Mca$, from $\Forests\Mca$ to~$\Motzkin\Mca$, and
from $\Motzkin\Mca$ to~$\Luc\Mca$. The relations between the secondary
suboperads and quotients of $\Cli\Mca$ built here are summarized in the
diagram of operad morphisms of
Figure~\ref{fig:diagram_secondary_operads}.
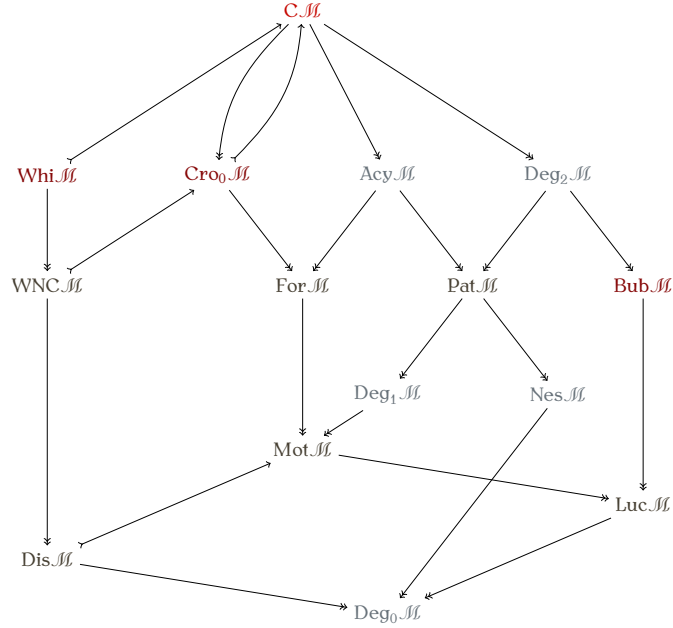
\begin{figure}[ht]
    \centering
    \scalebox{.73}{
    \begin{tikzpicture}[xscale=1.55,yscale=1,Centering]
        \node[text=Col2](CliM)at(13,3)
            {\begin{math}\Cli\Mca\end{math}};
        \node[text=Col4](Deg0M)at(14,-8)
            {\begin{math}\Deg_0\Mca\end{math}};
        \node[text=Col1](Cro0M)at(12,0)
            {\begin{math}\Cro_0\Mca\end{math}};
        \node[text=Col1](WhiM)at(10,0)
            {\begin{math}\Whi\Mca\end{math}};
        \node[text=Col4](AcyM)at(14,0)
            {\begin{math}\Acy\Mca\end{math}};
        \node[text=Col4](Deg2M)at(16,0)
            {\begin{math}\Deg_2\Mca\end{math}};
        \node[text=Col6](WNCM)at(10,-2)
            {\begin{math}\WNC\Mca\end{math}};
        \node[text=Col6](ForM)at(13,-2)
            {\begin{math}\Forests\Mca\end{math}};
        \node[text=Col6](PatM)at(15,-2)
            {\begin{math}\Paths\Mca\end{math}};
        \node[text=Col1](BubM)at(17,-2)
            {\begin{math}\Bub\Mca\end{math}};
        \node[text=Col4](Deg1M)at(14,-4)
            {\begin{math}\Deg_1\Mca\end{math}};
        \node[text=Col4](NesM)at(16,-4)
            {\begin{math}\Nes\Mca\end{math}};
        \node[text=Col6](MotM)at(13,-5)
            {\begin{math}\Motzkin\Mca\end{math}};
        \node[text=Col6](DisM)at(10,-7)
            {\begin{math}\Diss\Mca\end{math}};
        \node[text=Col6](LucM)at(17,-6)
            {\begin{math}\Luc\Mca\end{math}};
        \draw[Injection](WNCM)--(Cro0M);
        \draw[Injection](DisM)--(MotM);
        \draw[Surjection](Cro0M)--(ForM);
        \draw[Surjection](Deg2M)--(BubM);
        \draw[Surjection](Deg2M)--(PatM);
        \draw[Surjection](AcyM)--(PatM);
        \draw[Surjection](AcyM)--(ForM);
        \draw[Surjection](PatM)--(NesM);
        \draw[Surjection](PatM)--(Deg1M);
        \draw[Surjection](ForM)--(MotM);
        \draw[Surjection](MotM)--(LucM);
        \draw[Surjection](WNCM)--(DisM);
        \draw[Surjection](WhiM)--(WNCM);
        \draw[Surjection](Deg1M)--(MotM);
        \draw[Surjection](BubM)--(LucM);
        \draw[Surjection](DisM)--(Deg0M);
        \draw[Surjection](LucM)--(Deg0M);
        \draw[Injection](WhiM)--(CliM);
        \draw[Injection](Cro0M)to[bend right=15](CliM);
        \draw[Surjection](CliM)to[bend right=15](Cro0M);
        \draw[Surjection](CliM)--(AcyM);
        \draw[Surjection](CliM)--(Deg2M);
        \draw[Surjection](NesM)--(Deg0M);
    \end{tikzpicture}}
    \vspace{-.5em}
    \caption{\footnotesize
    The diagram of the secondary suboperads and quotients of $\Cli\Mca$
    together with some of their related main suboperads and quotients of
    $\Cli\Mca$. Arrows~$\rightarrowtail$ (resp.~$\twoheadrightarrow$)
    are injective (resp. surjective) operad morphisms. Here, $\Mca$ is a
    unitary magma without nontrival unit divisors.}
    \label{fig:diagram_secondary_operads}
\end{figure}
\medbreak

%%%%%%%%%%%%%%%%%%%%%%%%%%%%%%%%%%%%%%%%%%%%%%%%%%%%%%%%%%%%%%%%%%%%%%%%
%%%%%%%%%%%%%%%%%%%%%%%%%%%%%%%%%%%%%%%%%%%%%%%%%%%%%%%%%%%%%%%%%%%%%%%%
%%%%%%%%%%%%%%%%%%%%%%%%%%%%%%%%%%%%%%%%%%%%%%%%%%%%%%%%%%%%%%%%%%%%%%%%
\section{Concrete constructions}%
\label{sec:concrete_constructions}
The clique construction provides alternative definitions of known
operads. We explore here the cases of the operads $\MT$
and $\DMT$ of multi-tildes and double multi-tildes, and the gravity
operad~$\Grav$.
\medbreak

%%%%%%%%%%%%%%%%%%%%%%%%%%%%%%%%%%%%%%%%%%%%%%%%%%%%%%%%%%%%%%%%%%%%%%%%
%%%%%%%%%%%%%%%%%%%%%%%%%%%%%%%%%%%%%%%%%%%%%%%%%%%%%%%%%%%%%%%%%%%%%%%%
\subsection{Operads from language theory}
We provide constructions of two operads coming from formal language
theory by using the clique construction.
\medbreak

%%%%%%%%%%%%%%%%%%%%%%%%%%%%%%%%%%%%%%%%%%%%%%%%%%%%%%%%%%%%%%%%%%%%%%%%
\subsubsection{Multi-tildes}
Multi-tildes are operators introduced in~\cite{CCM11} in the context of
formal language theory as a convenient way to express regular languages.
Let, for any $n \geq 1$, $P_n$ be the set
\begin{equation}
    P_n :=  \left\{(x, y) \in [n]^2 : x \leq y\right\}.
\end{equation}
A \Def{multi-tilde} is a pair $(n, \Sfr)$ where $n$ is a positive
integer and $\Sfr$ is a subset of $P_n$. The \Def{arity} of the
multi-tilde $(n, \Sfr)$ is $n$.
\medbreak

As shown in~\cite{LMN13}, the graded (by the arity) collection of all
multi-tildes admits a very natural structure of an operad. This operad,
denoted by $\MT$, is defined as follows. The partial composition
$(n, \Sfr) \circ_i (m, \Tfr)$, $i \in [n]$, of two multi-tildes
$(n, \Sfr)$ and $(m, \Tfr)$ is defined by
\begin{equation}
    (n, \Sfr) \circ_i (m, \Tfr) :=
    \left(n + m - 1,
    \left\{\Shift_i^m(x, y) : (x, y) \in \Sfr\right\}
    \cup
    \left\{\Shift_0^i(x, y) : (x, y) \in \Tfr\right\}\right),
\end{equation}
where
\begin{equation}
    \Shift_j^p(x, y) :=
    \begin{cases}
        (x, y) & \mbox{if } y \leq i - 1, \\
        (x, y + p - 1) & \mbox{if } x \leq i \leq y, \\
        (x + p - 1, y + p - 1) & \mbox{otherwise}.
    \end{cases}
\end{equation}
For instance, one has
\begin{subequations}
\begin{equation} \label{equ:example_composition_MT_1}
    (5, \{(1, 5), (2, 4), (4, 5)\}) \circ_4 (6, \{(2, 2), (4, 6)\}) \\
    = (10, \{(1, 10), (2, 9), (4, 10), (5, 5), (7, 9)\}),
\end{equation}
\begin{equation} \label{equ:example_composition_MT_2}
    (5, \{(1, 5), (2, 4), (4, 5)\}) \circ_5 (6, \{(2, 2), (4, 6)\}) \\
    = (10, \{(1, 10), (2, 4), (4, 10), (6, 6), (8, 10)\}).
\end{equation}
\end{subequations}
Observe that the multi-tilde $(1, \emptyset)$ is the unit of~$\MT$.
\medbreak

Let $\phi_{\MT} : \MT \to \Cli\Dbb_0$ be the linear map defined as
follows. For any multi-tilde $(n, \Sfr)$ different from
$(1, \{(1, 1)\})$, $\phi_{\MT}((n, \Sfr))$ is the $\Dbb_0$-clique of
arity $n$ defined, for any $1 \leq x < y \leq n + 1$, by
\begin{equation} \label{equ:isomorphism_MT_Cli_M}
    \phi_{\MT}((n, \Sfr))(x, y) :=
    \begin{cases}
        0 & \mbox{if } (x, y - 1) \in \Sfr, \\
        \Unit & \mbox{otherwise}.
    \end{cases}
\end{equation}
For instance,
\begin{equation}
    \phi_{\MT}((5, \{(1, 5), (2, 4), (4, 5)\}))
    =
    \begin{tikzpicture}[scale=.7,Centering]
        \node[CliquePoint](1)at(-0.50,-0.87){};
        \node[CliquePoint](2)at(-1.00,-0.00){};
        \node[CliquePoint](3)at(-0.50,0.87){};
        \node[CliquePoint](4)at(0.50,0.87){};
        \node[CliquePoint](5)at(1.00,0.00){};
        \node[CliquePoint](6)at(0.50,-0.87){};
        \draw[CliqueEdge](1)edge[]node[CliqueLabel]
            {\begin{math}0\end{math}}(6);
        \draw[CliqueEmptyEdge](1)edge[]node[CliqueLabel]{}(2);
        \draw[CliqueEmptyEdge](2)edge[]node[CliqueLabel]{}(3);
        \draw[CliqueEdge](2)edge[]node[CliqueLabel,near start]
            {\begin{math}0\end{math}}(5);
        \draw[CliqueEmptyEdge](3)edge[]node[CliqueLabel]{}(4);
        \draw[CliqueEdge](4)
            edge[bend right=30]node[CliqueLabel,near start]
            {\begin{math}0\end{math}}(6);
        \draw[CliqueEmptyEdge](4)edge[]node[CliqueLabel]{}(5);
        \draw[CliqueEmptyEdge](5)edge[]node[CliqueLabel]{}(6);
    \end{tikzpicture}\,.
\end{equation}
\medbreak

\begin{Proposition} \label{prop:construction_MT}
    The operad $\Cli\Dbb_0$ is isomorphic to the suboperad
    of $\MT$ consisting in the linear span of all multi-tildes except
    the nontrivial multi-tilde $(1, \{(1, 1)\})$ of arity $1$. Moreover,
    $\phi_{\MT}$ is an isomorphism between these two operads.
\end{Proposition}
\begin{proof}
    A direct consequence of the
    definition~\eqref{equ:isomorphism_MT_Cli_M} of $\phi_{\MT}$ is that
    this map is an isomorphism of vector spaces. Moreover, it follows
    from the definitions of the partial compositions of $\MT$ and
    $\Cli\Dbb_0$ that $\phi_{\MT}$ is an operad morphism.
\end{proof}
\medbreak

By Proposition~\ref{prop:construction_MT}, one can interpret the partial
compositions~\eqref{equ:example_composition_MT_1}
and~\eqref{equ:example_composition_MT_2} of multi-tildes as partial
compositions of $\Dbb_0$-cliques. This give respectively
\begin{subequations}
\begin{equation}
    \begin{tikzpicture}[scale=.7,Centering]
        \node[CliquePoint](1)at(-0.50,-0.87){};
        \node[CliquePoint](2)at(-1.00,-0.00){};
        \node[CliquePoint](3)at(-0.50,0.87){};
        \node[CliquePoint](4)at(0.50,0.87){};
        \node[CliquePoint](5)at(1.00,0.00){};
        \node[CliquePoint](6)at(0.50,-0.87){};
        \draw[CliqueEdge](1)edge[]node[CliqueLabel]
            {\begin{math}0\end{math}}(6);
        \draw[CliqueEmptyEdge](1)edge[]node[CliqueLabel]{}(2);
        \draw[CliqueEmptyEdge](2)edge[]node[CliqueLabel]{}(3);
        \draw[CliqueEdge](2)edge[]node[CliqueLabel,near start]
            {\begin{math}0\end{math}}(5);
        \draw[CliqueEmptyEdge](3)edge[]node[CliqueLabel]{}(4);
        \draw[CliqueEdge](4)
            edge[bend right=30]node[CliqueLabel,near start]
            {\begin{math}0\end{math}}(6);
        \draw[CliqueEmptyEdge](4)edge[]node[CliqueLabel]{}(5);
        \draw[CliqueEmptyEdge](5)edge[]node[CliqueLabel]{}(6);
    \end{tikzpicture}
    \enspace \circ_4 \enspace
    \begin{tikzpicture}[scale=.8,Centering]
        \node[CliquePoint](1)at(-0.43,-0.90){};
        \node[CliquePoint](2)at(-0.97,-0.22){};
        \node[CliquePoint](3)at(-0.78,0.62){};
        \node[CliquePoint](4)at(-0.00,1.00){};
        \node[CliquePoint](5)at(0.78,0.62){};
        \node[CliquePoint](6)at(0.97,-0.22){};
        \node[CliquePoint](7)at(0.43,-0.90){};
        \draw[CliqueEmptyEdge](1)edge[]node[CliqueLabel]{}(2);
        \draw[CliqueEmptyEdge](1)edge[]node[CliqueLabel]{}(7);
        \draw[CliqueEdge](2)edge[]node[CliqueLabel]
            {\begin{math}0\end{math}}(3);
        \draw[CliqueEmptyEdge](3)edge[]node[CliqueLabel]{}(4);
        \draw[CliqueEmptyEdge](4)edge[]node[CliqueLabel]{}(5);
        \draw[CliqueEdge](4)edge[bend right=30]node[CliqueLabel]
            {\begin{math}0\end{math}}(7);
        \draw[CliqueEmptyEdge](5)edge[]node[CliqueLabel]{}(6);
        \draw[CliqueEmptyEdge](6)edge[]node[CliqueLabel]{}(7);
    \end{tikzpicture}
    \enspace = \enspace
    \begin{tikzpicture}[scale=1.1,Centering]
        \node[CliquePoint](1)at(-0.28,-0.96){};
        \node[CliquePoint](2)at(-0.76,-0.65){};
        \node[CliquePoint](3)at(-0.99,-0.14){};
        \node[CliquePoint](4)at(-0.91,0.42){};
        \node[CliquePoint](5)at(-0.54,0.84){};
        \node[CliquePoint](6)at(-0.00,1.00){};
        \node[CliquePoint](7)at(0.54,0.84){};
        \node[CliquePoint](8)at(0.91,0.42){};
        \node[CliquePoint](9)at(0.99,-0.14){};
        \node[CliquePoint](10)at(0.76,-0.65){};
        \node[CliquePoint](11)at(0.28,-0.96){};
        \draw[CliqueEdge](1)edge[]node[CliqueLabel]
            {\begin{math}0\end{math}}(11);
        \draw[CliqueEdge](2)
            edge[bend left=30]node[CliqueLabel,near start]
            {\begin{math}0\end{math}}(10);
        \draw[CliqueEdge](4)edge[bend left=30]node[CliqueLabel]
            {\begin{math}0\end{math}}(11);
        \draw[CliqueEdge](5)edge[]node[CliqueLabel]
            {\begin{math}0\end{math}}(6);
        \draw[CliqueEdge](7)edge[bend right=30]node[CliqueLabel]
            {\begin{math}0\end{math}}(10);
        \draw[CliqueEmptyEdge](1)edge[]node[CliqueLabel]{}(2);
        \draw[CliqueEmptyEdge](2)edge[]node[CliqueLabel]{}(3);
        \draw[CliqueEmptyEdge](3)edge[]node[CliqueLabel]{}(4);
        \draw[CliqueEmptyEdge](4)edge[]node[CliqueLabel]{}(5);
        \draw[CliqueEmptyEdge](6)edge[]node[CliqueLabel]{}(7);
        \draw[CliqueEmptyEdge](7)edge[]node[CliqueLabel]{}(8);
        \draw[CliqueEmptyEdge](8)edge[]node[CliqueLabel]{}(9);
        \draw[CliqueEmptyEdge](9)edge[]node[CliqueLabel]{}(10);
        \draw[CliqueEmptyEdge](10)edge[]node[CliqueLabel]{}(11);
    \end{tikzpicture}\,,
\end{equation}
\begin{equation}
    \begin{tikzpicture}[scale=.7,Centering]
        \node[CliquePoint](1)at(-0.50,-0.87){};
        \node[CliquePoint](2)at(-1.00,-0.00){};
        \node[CliquePoint](3)at(-0.50,0.87){};
        \node[CliquePoint](4)at(0.50,0.87){};
        \node[CliquePoint](5)at(1.00,0.00){};
        \node[CliquePoint](6)at(0.50,-0.87){};
        \draw[CliqueEdge](1)edge[]node[CliqueLabel]
            {\begin{math}0\end{math}}(6);
        \draw[CliqueEmptyEdge](1)edge[]node[CliqueLabel]{}(2);
        \draw[CliqueEmptyEdge](2)edge[]node[CliqueLabel]{}(3);
        \draw[CliqueEdge](2)edge[]node[CliqueLabel,near start]
            {\begin{math}0\end{math}}(5);
        \draw[CliqueEmptyEdge](3)edge[]node[CliqueLabel]{}(4);
        \draw[CliqueEdge](4)
            edge[bend right=30]node[CliqueLabel,near start]
            {\begin{math}0\end{math}}(6);
        \draw[CliqueEmptyEdge](4)edge[]node[CliqueLabel]{}(5);
        \draw[CliqueEmptyEdge](5)edge[]node[CliqueLabel]{}(6);
    \end{tikzpicture}
    \enspace \circ_5 \enspace
    \begin{tikzpicture}[scale=.8,Centering]
        \node[CliquePoint](1)at(-0.43,-0.90){};
        \node[CliquePoint](2)at(-0.97,-0.22){};
        \node[CliquePoint](3)at(-0.78,0.62){};
        \node[CliquePoint](4)at(-0.00,1.00){};
        \node[CliquePoint](5)at(0.78,0.62){};
        \node[CliquePoint](6)at(0.97,-0.22){};
        \node[CliquePoint](7)at(0.43,-0.90){};
        \draw[CliqueEmptyEdge](1)edge[]node[CliqueLabel]{}(2);
        \draw[CliqueEmptyEdge](1)edge[]node[CliqueLabel]{}(7);
        \draw[CliqueEdge](2)edge[]node[CliqueLabel]
            {\begin{math}0\end{math}}(3);
        \draw[CliqueEmptyEdge](3)edge[]node[CliqueLabel]{}(4);
        \draw[CliqueEmptyEdge](4)edge[]node[CliqueLabel]{}(5);
        \draw[CliqueEdge](4)edge[bend right=30]node[CliqueLabel]
            {\begin{math}0\end{math}}(7);
        \draw[CliqueEmptyEdge](5)edge[]node[CliqueLabel]{}(6);
        \draw[CliqueEmptyEdge](6)edge[]node[CliqueLabel]{}(7);
    \end{tikzpicture}
    \enspace = \enspace
    \begin{tikzpicture}[scale=1.1,Centering]
        \node[CliquePoint](1)at(-0.28,-0.96){};
        \node[CliquePoint](2)at(-0.76,-0.65){};
        \node[CliquePoint](3)at(-0.99,-0.14){};
        \node[CliquePoint](4)at(-0.91,0.42){};
        \node[CliquePoint](5)at(-0.54,0.84){};
        \node[CliquePoint](6)at(-0.00,1.00){};
        \node[CliquePoint](7)at(0.54,0.84){};
        \node[CliquePoint](8)at(0.91,0.42){};
        \node[CliquePoint](9)at(0.99,-0.14){};
        \node[CliquePoint](10)at(0.76,-0.65){};
        \node[CliquePoint](11)at(0.28,-0.96){};
        \draw[CliqueEdge](1)edge[]node[CliqueLabel]
            {\begin{math}0\end{math}}(11);
        \draw[CliqueEdge](2)
            edge[bend right=30]node[CliqueLabel,near start]
            {\begin{math}0\end{math}}(5);
        \draw[CliqueEdge](4)edge[bend left=30]node[CliqueLabel]
            {\begin{math}0\end{math}}(11);
        \draw[CliqueEdge](6)edge[]node[CliqueLabel]
            {\begin{math}0\end{math}}(7);
        \draw[CliqueEdge](8)edge[bend right=30]node[CliqueLabel]
            {\begin{math}0\end{math}}(11);
        \draw[CliqueEmptyEdge](1)edge[]node[CliqueLabel]{}(2);
        \draw[CliqueEmptyEdge](2)edge[]node[CliqueLabel]{}(3);
        \draw[CliqueEmptyEdge](3)edge[]node[CliqueLabel]{}(4);
        \draw[CliqueEmptyEdge](4)edge[]node[CliqueLabel]{}(5);
        \draw[CliqueEmptyEdge](5)edge[]node[CliqueLabel]{}(6);
        \draw[CliqueEmptyEdge](7)edge[]node[CliqueLabel]{}(8);
        \draw[CliqueEmptyEdge](8)edge[]node[CliqueLabel]{}(9);
        \draw[CliqueEmptyEdge](9)edge[]node[CliqueLabel]{}(10);
        \draw[CliqueEmptyEdge](10)edge[]node[CliqueLabel]{}(11);
    \end{tikzpicture}\,.
\end{equation}
\end{subequations}
\medbreak

%%%%%%%%%%%%%%%%%%%%%%%%%%%%%%%%%%%%%%%%%%%%%%%%%%%%%%%%%%%%%%%%%%%%%%%%
\subsubsection{Double multi-tildes}
Double multi-tildes are natural generalizations of multi-tildes,
introduced in~\cite{GLMN16}. A \Def{double multi-tilde} is a triple
$(n, \Sfr, \Tfr)$ where $(n, \Tfr)$ and $(n, \Sfr)$ are both multi-tildes
of the same arity $n$. The \Def{arity} of the double multi-tilde
$(n, \Sfr, \Tfr)$ is $n$. As shown in~\cite{GLMN16}, the linear span of
all double multi-tildes admits a structure of an operad. This operad,
denoted by $\DMT$, is defined as follows. For any $n \geq 1$, $\DMT(n)$
is the linear span of all double multi-tildes of arity $n$ and the
partial composition $(n, \Sfr, \Tfr) \circ_i (m, \Ufr, \Vfr)$,
$i \in [n]$, of two double multi-tildes $(n, \Sfr, \Tfr)$ and
$(m, \Ufr, \Vfr)$ is defined linearly by
\begin{equation} \label{equ:partial_composition_DMT}
    (n, \Sfr, \Tfr) \circ_i (m, \Ufr, \Vfr) :=
    (n, \Sfr \circ_i \Ufr, \Tfr \circ_i \Vfr),
\end{equation}
where the two partial compositions $\circ_i$ of the right member
of~\eqref{equ:partial_composition_DMT} are the ones of~$\MT$. We can
observe that $\DMT$ is isomorphic to the Hadamard product $\MT * \MT$.
For instance, one has
\begin{equation} \label{equ:example_composition_DMT}
    (3, \{(2, 2)\}, \{(1, 2), (1, 3)\})
    \circ_2
    (2, \{(1, 1)\}, \{(1, 2)\})
    =
    (4, \{(2, 2), (2, 3)\}, \{(1, 3), (1, 4), (2, 3)\}).
\end{equation}
The unit of $\DMT$ is $(1, \emptyset, \emptyset)$.
\medbreak

Consider now the operad $\Cli\Dbb_0^2$ and let
$\phi_{\DMT} : \DMT \to \Cli\Dbb_0^2$ be the linear map defined as
follows. The image by $\phi_{\DMT}$ of $(1, \emptyset, \emptyset)$ is
the unit of $\Cli\Dbb_0^2$ and, for any double multi-tilde
$(n, \Sfr, \Tfr)$ of arity $n \geq 2$, $\phi_{\DMT}((n, \Sfr, \Tfr))$ is
the $\Dbb_0^2$-clique of arity $n$ defined, for any
$1 \leq x < y \leq n + 1$, by
\begin{equation} \label{equ:isomorphism_DMT_Cli_M}
    \phi_{\DMT}((n, \Sfr, \Tfr))(x, y) :=
    \begin{cases}
        (0, \Unit) & \mbox{if } (x, y - 1) \in \Sfr
            \mbox{ and } (x, y - 1) \notin \Tfr, \\
        (\Unit, 0) & \mbox{if } (x, y - 1) \notin \Sfr
            \mbox{ and } (x, y - 1) \in \Tfr, \\
        (0, 0) & \mbox{if } (x, y - 1) \in \Sfr
            \mbox{ and } (x, y - 1) \in \Tfr, \\
        (\Unit, \Unit) & \mbox{otherwise}.
    \end{cases}
\end{equation}
For instance,
\begin{equation}
    \phi_{\DMT}((4, \{(2, 2), (2, 3)\}, \{(1, 3), (1, 4), (2, 3)\}))
    =
    \begin{tikzpicture}[scale=1.0,Centering]
        \node[CliquePoint](1)at(-0.59,-0.81){};
        \node[CliquePoint](2)at(-0.95,0.31){};
        \node[CliquePoint](3)at(-0.00,1.00){};
        \node[CliquePoint](4)at(0.95,0.31){};
        \node[CliquePoint](5)at(0.59,-0.81){};
        \draw[CliqueEmptyEdge](1)edge[]node[CliqueLabel]{}(2);
        \draw[CliqueEdge](1)edge[bend left=30]node[CliqueLabel]
            {\begin{math}(\Unit, 0)\end{math}}(4);
        \draw[CliqueEdge](1)edge[]node[CliqueLabel]
            {\begin{math}(\Unit, 0)\end{math}}(5);
        \draw[CliqueEdge](2)edge[]node[CliqueLabel]
            {\begin{math}(0, \Unit)\end{math}}(3);
        \draw[CliqueEdge](2)edge[]node[CliqueLabel]
            {\begin{math}(0, 0)\end{math}}(4);
        \draw[CliqueEmptyEdge](3)edge[]node[CliqueLabel]{}(4);
        \draw[CliqueEmptyEdge](4)edge[]node[CliqueLabel]{}(5);
    \end{tikzpicture}\,.
\end{equation}
\medbreak

\begin{Proposition} \label{prop:construction_DMT}
    The operad $\Cli\Dbb_0^2$ is isomorphic to the suboperad of $\DMT$
    consisting in the linear span of all double multi-tildes except the
    three nontrivial double multi-tildes of arity $1$. Moreover,
    $\phi_{\DMT}$ is an isomorphism between these two operads.
\end{Proposition}
\begin{proof}
    There are two ways to prove the first assertion of the statement of
    the proposition. On the one hand, this property follows from
    Proposition~\ref{prop:Cli_M_Cartesian_product} and
    Proposition~\ref{prop:construction_MT}. On the other hand, the whole
    statement of the proposition is a direct consequence of the
    definition~\eqref{equ:isomorphism_DMT_Cli_M} of $\phi_{\DMT}$,
    showing that $\phi_{\DMT}$ is an isomorphism of vector spaces, and,
    from the definitions of the partial compositions of $\DMT$ and
    $\Cli\Dbb_0^2$ showing that $\phi_{\DMT}$ is an operad morphism.
\end{proof}
\medbreak

By Proposition~\ref{prop:construction_DMT}, one can interpret the
partial composition~\eqref{equ:example_composition_DMT} of double
multi-tildes as a partial composition of $\Dbb_0^2$-cliques. This gives
\begin{equation}
    \begin{tikzpicture}[scale=.7,Centering]
        \node[CliquePoint](1)at(-0.71,-0.71){};
        \node[CliquePoint](2)at(-0.71,0.71){};
        \node[CliquePoint](3)at(0.71,0.71){};
        \node[CliquePoint](4)at(0.71,-0.71){};
        \draw[CliqueEmptyEdge](1)edge[]node[CliqueLabel]{}(2);
        \draw[CliqueEdge](1)edge[]node[CliqueLabel]
            {\begin{math}(\Unit, 0)\end{math}}(3);
        \draw[CliqueEdge](1)edge[]node[CliqueLabel]
            {\begin{math}(\Unit, 0)\end{math}}(4);
        \draw[CliqueEdge](2)edge[]node[CliqueLabel]
            {\begin{math}(0, \Unit)\end{math}}(3);
        \draw[CliqueEmptyEdge](3)edge[]node[CliqueLabel]{}(4);
    \end{tikzpicture}
    \enspace \circ_2 \enspace
    \begin{tikzpicture}[scale=.6,Centering]
        \node[CliquePoint](1)at(-0.87,-0.50){};
        \node[CliquePoint](2)at(-0.00,1.00){};
        \node[CliquePoint](3)at(0.87,-0.50){};
        \draw[CliqueEdge](1)edge[]node[CliqueLabel]
            {\begin{math}(0, \Unit)\end{math}}(2);
        \draw[CliqueEdge](1)edge[]node[CliqueLabel]
            {\begin{math}(\Unit, 0)\end{math}}(3);
        \draw[CliqueEmptyEdge](2)edge[]node[CliqueLabel]{}(3);
    \end{tikzpicture}
    \enspace = \enspace
    \begin{tikzpicture}[scale=1.0,Centering]
        \node[CliquePoint](1)at(-0.59,-0.81){};
        \node[CliquePoint](2)at(-0.95,0.31){};
        \node[CliquePoint](3)at(-0.00,1.00){};
        \node[CliquePoint](4)at(0.95,0.31){};
        \node[CliquePoint](5)at(0.59,-0.81){};
        \draw[CliqueEmptyEdge](1)edge[]node[CliqueLabel]{}(2);
        \draw[CliqueEdge](1)edge[bend left=30]node[CliqueLabel]
            {\begin{math}(\Unit, 0)\end{math}}(4);
        \draw[CliqueEdge](1)edge[]node[CliqueLabel]
            {\begin{math}(\Unit, 0)\end{math}}(5);
        \draw[CliqueEdge](2)edge[]node[CliqueLabel]
            {\begin{math}(0, \Unit)\end{math}}(3);
        \draw[CliqueEdge](2)edge[]node[CliqueLabel]
            {\begin{math}(0, 0)\end{math}}(4);
        \draw[CliqueEmptyEdge](3)edge[]node[CliqueLabel]{}(4);
        \draw[CliqueEmptyEdge](4)edge[]node[CliqueLabel]{}(5);
    \end{tikzpicture}\,.
\end{equation}
\medbreak

%%%%%%%%%%%%%%%%%%%%%%%%%%%%%%%%%%%%%%%%%%%%%%%%%%%%%%%%%%%%%%%%%%%%%%%%
%%%%%%%%%%%%%%%%%%%%%%%%%%%%%%%%%%%%%%%%%%%%%%%%%%%%%%%%%%%%%%%%%%%%%%%%
\subsection{Gravity operad}
The \Def{operad of gravity chord diagrams} $\Grav$ is an operad defined
in~\cite{AP15}. This operad is the nonsymmetric version (obtained by
forgetting the actions of the symmetric groups) of the gravity operad, a
symmetric operad introduced by Getzler~\cite{Get94}. Let us describe
this operad.
\medbreak

A \Def{gravity chord diagram} is a $\{\star\}$-configuration $\Cfr$,
where $\star$ is any symbol, satisfying the following conditions. By
denoting by $n$ the size of $\Cfr$, all the edges and the base of $\Cfr$
are labeled (by $\star$), and if $(x, y)$ and $(x', y')$ are two labeled
crossing diagonals of $\Cfr$ such that $x < x'$, the arc $(x', y)$ is
not labeled. In other words, the quadrilateral formed by the vertices
$x$, $x'$, $y$, and $y'$ of $\Cfr$ is such that its side $(x', y)$ is
unlabeled. For instance,
\begin{equation}
    \begin{tikzpicture}[scale=.8,Centering]
        \node[CliquePoint](1)at(-0.38,-0.92){};
        \node[CliquePoint](2)at(-0.92,-0.38){};
        \node[CliquePoint](3)at(-0.92,0.38){};
        \node[CliquePoint](4)at(-0.38,0.92){};
        \node[CliquePoint](5)at(0.38,0.92){};
        \node[CliquePoint](6)at(0.92,0.38){};
        \node[CliquePoint](7)at(0.92,-0.38){};
        \node[CliquePoint](8)at(0.38,-0.92){};
        \draw[CliqueEdgeBlue](1)--(2);
        \draw[CliqueEdgeBlue](1)--(8);
        \draw[CliqueEdgeBlue](2)--(3);
        \draw[CliqueEdgeBlue](2)--(5);
        \draw[CliqueEdgeBlue](2)--(6);
        \draw[CliqueEdgeBlue](2)--(7);
        \draw[CliqueEdgeBlue](3)--(4);
        \draw[CliqueEdgeBlue](3)--(6);
        \draw[CliqueEdgeBlue](4)--(5);
        \draw[CliqueEdgeBlue](5)--(6);
        \draw[CliqueEdgeBlue](6)--(7);
        \draw[CliqueEdgeBlue](7)--(8);
    \end{tikzpicture}
\end{equation}
is a gravity chord diagram of arity $7$ having four labeled diagonals
(observe in particular that, as required, the arc $(3, 5)$ is not
labeled). For any $n \geq 2$, $\Grav(n)$ is the linear span of all
gravity chord  diagrams of size $n$. Moreover, $\Grav(1)$ is the linear
span of the  singleton containing the only polygon of size $1$ where
its only arc is not labeled. The partial composition of $\Grav$ is
defined graphically as follows. For any gravity chord diagrams $\Cfr$
and $\Dfr$ of respective arities $n$ and $m$, and $i \in [n]$, the
gravity chord diagram $\Cfr \circ_i \Dfr$ is obtained by gluing the
base of $\Dfr$ onto the $i$th edge of $\Cfr$, so that the arc
$(i, i + m)$ of $\Cfr \circ_i \Dfr$ is labeled. For example,
\begin{equation}
    \begin{tikzpicture}[scale=.6,Centering]
        \node[CliquePoint](1)at(-0.50,-0.87){};
        \node[CliquePoint](2)at(-1.00,-0.00){};
        \node[CliquePoint](3)at(-0.50,0.87){};
        \node[CliquePoint](4)at(0.50,0.87){};
        \node[CliquePoint](5)at(1.00,0.00){};
        \node[CliquePoint](6)at(0.50,-0.87){};
        \draw[CliqueEdgeBlue](1)--(2);
        \draw[CliqueEdgeBlue](1)--(4);
        \draw[CliqueEdgeBlue](1)--(6);
        \draw[CliqueEdgeBlue](2)--(3);
        \draw[CliqueEdgeBlue](2)--(5);
        \draw[CliqueEdgeBlue](3)--(4);
        \draw[CliqueEdgeBlue](4)--(5);
        \draw[CliqueEdgeBlue](5)--(6);
    \end{tikzpicture}
    \enspace
    \circ_3
    \enspace
    \begin{tikzpicture}[scale=.5,Centering]
        \node[CliquePoint](1)at(-0.71,-0.71){};
        \node[CliquePoint](2)at(-0.71,0.71){};
        \node[CliquePoint](3)at(0.71,0.71){};
        \node[CliquePoint](4)at(0.71,-0.71){};
        \draw[CliqueEdgeBlue](1)--(2);
        \draw[CliqueEdgeBlue](1)--(3);
        \draw[CliqueEdgeBlue](1)--(4);
        \draw[CliqueEdgeBlue](2)--(3);
        \draw[CliqueEdgeBlue](3)--(4);
    \end{tikzpicture}
    \enspace
    =
    \enspace
    \begin{tikzpicture}[scale=.8,Centering]
        \node[CliquePoint](1)at(-0.38,-0.92){};
        \node[CliquePoint](2)at(-0.92,-0.38){};
        \node[CliquePoint](3)at(-0.92,0.38){};
        \node[CliquePoint](4)at(-0.38,0.92){};
        \node[CliquePoint](5)at(0.38,0.92){};
        \node[CliquePoint](6)at(0.92,0.38){};
        \node[CliquePoint](7)at(0.92,-0.38){};
        \node[CliquePoint](8)at(0.38,-0.92){};
        \draw[CliqueEdgeBlue](1)--(2);
        \draw[CliqueEdgeBlue](1)--(6);
        \draw[CliqueEdgeBlue](1)--(8);
        \draw[CliqueEdgeBlue](2)--(3);
        \draw[CliqueEdgeBlue](2)--(7);
        \draw[CliqueEdgeBlue](3)--(4);
        \draw[CliqueEdgeBlue](3)--(5);
        \draw[CliqueEdgeBlue](3)--(6);
        \draw[CliqueEdgeBlue](4)--(5);
        \draw[CliqueEdgeBlue](5)--(6);
        \draw[CliqueEdgeBlue](6)--(7);
        \draw[CliqueEdgeBlue](7)--(8);
    \end{tikzpicture}\,.
\end{equation}
\medbreak

Let $\phi_{\Grav} : \Grav \to \Cli\Dbb_0$ be the linear map defined in
the following way. For any gravity chord diagram $\Cfr$,
$\phi_{\Grav}(\Cfr)$ is the $\Dbb_0$-clique of $\Cli\Dbb_0$ obtained by
replacing all labeled arcs of $\Cfr$ by arcs labeled by $0$ and all
unlabeled arcs by arcs labeled by $\Unit$. For instance,
\begin{equation}
    \phi_{\Grav}\left(
    \begin{tikzpicture}[scale=.8,Centering]
        \node[CliquePoint](1)at(-0.38,-0.92){};
        \node[CliquePoint](2)at(-0.92,-0.38){};
        \node[CliquePoint](3)at(-0.92,0.38){};
        \node[CliquePoint](4)at(-0.38,0.92){};
        \node[CliquePoint](5)at(0.38,0.92){};
        \node[CliquePoint](6)at(0.92,0.38){};
        \node[CliquePoint](7)at(0.92,-0.38){};
        \node[CliquePoint](8)at(0.38,-0.92){};
        \draw[CliqueEdgeBlue](1)--(2);
        \draw[CliqueEdgeBlue](1)--(6);
        \draw[CliqueEdgeBlue](1)--(8);
        \draw[CliqueEdgeBlue](2)--(3);
        \draw[CliqueEdgeBlue](2)--(7);
        \draw[CliqueEdgeBlue](3)--(4);
        \draw[CliqueEdgeBlue](3)--(5);
        \draw[CliqueEdgeBlue](3)--(6);
        \draw[CliqueEdgeBlue](4)--(5);
        \draw[CliqueEdgeBlue](5)--(6);
        \draw[CliqueEdgeBlue](6)--(7);
        \draw[CliqueEdgeBlue](7)--(8);
    \end{tikzpicture}
    \right)
    \enspace
    =
    \enspace
    \begin{tikzpicture}[scale=.8,Centering]
        \node[CliquePoint](1)at(-0.38,-0.92){};
        \node[CliquePoint](2)at(-0.92,-0.38){};
        \node[CliquePoint](3)at(-0.92,0.38){};
        \node[CliquePoint](4)at(-0.38,0.92){};
        \node[CliquePoint](5)at(0.38,0.92){};
        \node[CliquePoint](6)at(0.92,0.38){};
        \node[CliquePoint](7)at(0.92,-0.38){};
        \node[CliquePoint](8)at(0.38,-0.92){};
        \draw[CliqueEdge](1)edge[]node[CliqueLabel]
            {\begin{math}0\end{math}}(2);
        \draw[CliqueEdge](1)edge[bend left=20]node[CliqueLabel]
            {\begin{math}0\end{math}}(6);
        \draw[CliqueEdge](1)edge[]node[CliqueLabel]
            {\begin{math}0\end{math}}(8);
        \draw[CliqueEdge](2)edge[]node[CliqueLabel]
            {\begin{math}0\end{math}}(3);
        \draw[CliqueEdge](2)edge[]node[CliqueLabel,near end]
            {\begin{math}0\end{math}}(7);
        \draw[CliqueEdge](3)edge[]node[CliqueLabel]
            {\begin{math}0\end{math}}(4);
        \draw[CliqueEdge](3)edge[bend right=20]node[CliqueLabel]
            {\begin{math}0\end{math}}(5);
        \draw[CliqueEdge](3)edge[bend right=20]node[CliqueLabel]
            {\begin{math}0\end{math}}(6);
        \draw[CliqueEdge](4)edge[]node[CliqueLabel]
            {\begin{math}0\end{math}}(5);
        \draw[CliqueEdge](5)edge[]node[CliqueLabel]
            {\begin{math}0\end{math}}(6);
        \draw[CliqueEdge](6)edge[]node[CliqueLabel]
            {\begin{math}0\end{math}}(7);
        \draw[CliqueEdge](7)edge[]node[CliqueLabel]
            {\begin{math}0\end{math}}(8);
    \end{tikzpicture}\,.
\end{equation}
\medbreak

Let us say that an $\Mca$-clique $\Pfr$ satisfies the
\Def{gravity condition} if $\Pfr = \UnitClique$, or $\Pfr$ has only
solid edges and bases, and for all crossing diagonals $(x, y)$ and
$(x', y')$ of $\Pfr$ such that $x < x'$,
$\Pfr(x, y) \ne \Unit_\Mca \ne \Pfr(x', y')$
implies~$\Pfr(x', y) = \Unit_\Mca$.
\medbreak

\begin{Proposition} \label{prop:construction_Grav}
    The linear span of all $\Dbb_0$-cliques satisfying the gravity
    condition forms a suboperad of $\Cli\Dbb_0$ isomorphic to $\Grav$.
    Moreover, $\phi_{\Grav}$ is an isomorphism between these two operads.
\end{Proposition}
\begin{proof}
    Let us denote by $\Oca_\Grav$ the subspace of $\Cli\Dbb_0$ described
    in the statement of the proposition. First of all, it follows from
    the definition of the partial composition of $\Cli\Dbb_0$ that
    $\Oca_\Grav$ is closed under the partial composition operation (this
    property can be also seen as a consequence of the fact that the
    partial composition of two gravity chord diagrams is still a gravity
    chord diagram~\cite{AP15}). Hence, and since $\Oca_\Grav$ contains
    the unit of $\Cli\Dbb_0$, $\Oca_\Grav$ is an operad. Second, observe
    that the image of $\phi_{\Grav}$ is the underlying space of
    $\Oca_\Grav$ and, from the definition of the partial composition of
    $\Grav$, one can check that $\phi_{\Grav}$ is an operad morphism.
    Finally, since $\phi_{\Grav}$ is a bijection from $\Grav$ to
    $\Oca_\Grav$, the statement of the proposition follows.
\end{proof}
\medbreak

Proposition~\ref{prop:construction_Grav} shows hence that the operad
$\Grav$ can be built through the clique construction. Moreover, as
explained in~\cite{AP15}, $\Grav$ contains the nonsymmetric version of
the $\Lie$ operad, the symmetric operad describing the category of Lie
algebras. This nonsymmetric version of the Lie operad as been introduced
in~\cite{ST09}. Since $\Lie$ is contained in $\Grav$ as the subspace of
all gravity chord diagrams having the maximal number of labeled diagonals
for each arity, $\Lie$ can be built through the clique construction as
the suboperad of $\Cli\Dbb_0$ containing all the $\Dbb_0$-cliques that
are images by $\phi_{\Grav}$ of such maximal gravity chord diagrams.
\medbreak

Besides, this alternative construction of $\Grav$ leads to the following
generalization for any unitary magma $\Mca$ of the gravity operad. Let
$\Grav_\Mca$ be the linear span of all $\Mca$-cliques satisfying the
gravity condition. It follows from the definition of the partial
composition of $\Cli\Mca$ that $\Grav_\Mca$ is an operad. Moreover,
observe that when $\Mca$ has nontrivial unit divisors, $\Grav_\Mca$ is
not a free operad.
\medbreak

%%%%%%%%%%%%%%%%%%%%%%%%%%%%%%%%%%%%%%%%%%%%%%%%%%%%%%%%%%%%%%%%%%%%%%%%
%%%%%%%%%%%%%%%%%%%%%%%%%%%%%%%%%%%%%%%%%%%%%%%%%%%%%%%%%%%%%%%%%%%%%%%%
%%%%%%%%%%%%%%%%%%%%%%%%%%%%%%%%%%%%%%%%%%%%%%%%%%%%%%%%%%%%%%%%%%%%%%%%
\section*{Conclusion and perspectives}
This work presents and studies the clique construction $\Cli$,
producing operads from unitary magmas. We have seen that $\Cli$ has many
both algebraic and combinatorial properties. Among its most notable
ones, $\Cli\Mca$ admits several quotients involving combinatorial
families of decorated cliques, and contains some already studied
operads. Let us address here some open questions.
\smallbreak

First, we have for the time being no formula to enumerate prime (resp.
white prime, minimal prime) $\Mca$-cliques
(see~\eqref{equ:prime_cliques_numbers_2} (resp.
\eqref{equ:white_prime_cliques_numbers},
\eqref{equ:minimal_prime_cliques_numbers}) for $\#\Mca = 2$). Obtaining
these forms a first combinatorial question.
\smallbreak

When $\Mca$ is a $\Z$-graded unitary magma, a link between $\Cli\Mca$
and the operad of rational functions $\RatFct$ has been
developed in Section~\ref{subsubsec:rational_functions} by means of a
morphism $\Frac_\theta$ between these two operads. We have observed that
$\Frac_\theta$ is not injective (see~\eqref{equ:frac_not_injective_1}
and~\eqref{equ:frac_not_injective_2}). A description of the kernel of
$\Frac_\theta$, even when $\Mca$ is the unitary magma $\Z$, seems not
easy to obtain. Trying to obtain this description is a second
perspective of this work.
\smallbreak

Here is a third perspective. In Section~\ref{sec:quotients_suboperads},
we have defined and briefly studied some quotients and suboperads of
$\Cli\Mca$. In particular, we have considered the quotient $\Deg_1\Mca$
of $\Cli\Mca$, involving $\Mca$-cliques of degree at most $1$. As
mentioned, $\Deg_1\Dbb_0$ is an operad defined on the linear span of
involutions (except the nontrivial involution of $\mathfrak{S}_2$). A
complete study of this operad seems worth considering, including a
description of a minimal generating set, a presentation by generators
and relations, a description of its partial composition on the
$\Hsf$-basis and on the $\Ksf$-basis, and a realization of this operad
in terms of standard Young tableaux.
\medbreak

%%%%%%%%%%%%%%%%%%%%%%%%%%%%%%%%%%%%%%%%%%%%%%%%%%%%%%%%%%%%%%%%%%%%%%%%
%%%%%%%%%%%%%%%%%%%%%%%%%%%%%%%%%%%%%%%%%%%%%%%%%%%%%%%%%%%%%%%%%%%%%%%%
%%%%%%%%%%%%%%%%%%%%%%%%%%%%%%%%%%%%%%%%%%%%%%%%%%%%%%%%%%%%%%%%%%%%%%%%
\bibliographystyle{alpha}
\bibliography{Bibliography}

\end{document}